\documentclass[final,leqno,onefignum,onetabnum]{siamltex1213}

\usepackage{amsmath,amssymb,amsfonts,graphicx}

\newcounter{remark}
\newenvironment{remark}{\addtocounter{remark}{1}\par {\it Remark} {\rm \arabic{remark}.}}{\par}
\newcounter{example}
\newenvironment{example}{\addtocounter{example}{1}\par {\it Example} {\rm \arabic{example}.}}{\par}
\newcommand{\Cinf}{\ensuremath{\mathcal{C}^\infty}}

\newcommand{\D}{\ensuremath{{\mathcal D}}}

\newcommand{\mb}[1]{\ensuremath{\mathbb{#1}}}
\newcommand{\N}{\mb{N}}

\newcommand{\R}{\mb{R}}

\newcommand{\Z}{\mb{Z}}

\newcommand{\G}{\ensuremath{{\mathcal G}}}

\newcommand{\Gc}{\ensuremath{{\cal G}_\mathrm{c}}}

\newcommand{\EM}{\ensuremath{{\mathcal E}_{M}}}

\newcommand{\EMinf}{\ensuremath{{\mathcal E}^\infty_{M}}}

\newcommand{\Neg}{\mathcal{N}}

\newcommand{\Ginf}{\ensuremath{\G^\infty}}

\newcommand{\singsupp}{\operatorname{sing\,supp\hspace{0.5pt}}}
\renewcommand{\supp}{\mathop{\mathrm{supp}}}

\newfont{\bigmath}{cmr12 at 13pt}

\newfont{\grecomath}{cmmi12 at 15pt}

\newcommand{\p}{\ensuremath{\partial}}

\renewcommand{\div}{\operatorname{{div}}}

\newcommand{\eps}{\varepsilon}

\newcommand{\ee}{{\rm e}\hspace{1pt}}

\newcommand{\cC}{{\mathcal C}}
\newcommand{\cD}{{\mathcal D}}
\newcommand{\cE}{{\mathcal E}}

\newcommand{\cG}{{\mathcal G}}

\newcommand{\cN}{{\mathcal N}}
\newcommand{\cO}{{\mathcal O}}

\newcommand{\cS}{{\mathcal S}}

\newcommand{\EMtwo}{\ensuremath{{\cE}_{\mathrm{M},2}}}
\newcommand{\Gtwo}{\ensuremath{\cG_{2,2}}}

\def\Gc{\mathcal{G}_c}

\newcommand{\hpl}{\ensuremath{[0,\infty)\times\R}}
\newcommand{\stripn}{\ensuremath{[0,T]\times\R^d}}
\newcommand{\strip}{\ensuremath{[0,T]\times\R}}
\newcommand{\Arsinh}{\operatorname{Arsinh}}
\newcommand{\sign}{\operatorname{sign}}

\title{Propagation of singularities for generalized solutions to wave equations with discontinuous coefficients}

\author{Hideo Deguchi\thanks{Department of Mathematics, University of Toyama,
   Gofuku 3190, 930-8555 Toyama, Japan (hdegu@sci.u-toyama.ac.jp)}\and
   Michael Oberguggenberger\thanks{Unit of Engineering Mathematics, University of Innsbruck, Technikerstrasse 13, A-6020 Innsbruck, Austria (michael.oberguggenberger@uibk.ac.at)}}

\begin{document}
\maketitle

\begin{abstract}
This article addresses linear hyperbolic partial differential equations with non-smooth coefficients and distributional data.
Solutions are studied in the framework of Colombeau algebras of generalized functions.
Its aim is to prove upper and lower bounds for the singular support of generalized solutions for wave equations with discontinuous coefficients. New existence results with weaker assumptions on the representing families are required and proven. The program is carried through for various types of one- and multidimensional wave equations and hyperbolic systems.
\end{abstract}

\begin{keywords}hyperbolic equations; non-smooth coefficients; Colombeau generalized functions; propagation of singularities\end{keywords}

\begin{AMS}Primary, 35A21, 46F30, 35B65; Secondary, 35L05, 35L45\end{AMS}

\pagestyle{myheadings}
\thispagestyle{plain}
\markboth{HIDEO DEGUCHI, MICHAEL OBERGUGGENBERGER}{WAVE EQUATIONS WITH DISCONTINUOUS COEFFICIENTS}


\section{Introduction}\label{sec : 1}

This paper is an advanced part of a program that aims at studying propagation of singularities and regularity of solutions to linear hyperbolic equations and systems with non-smooth coefficients in algebras of generalized functions. To explain the setting, consider a wave equation of the form
\begin{equation}\label{eq:conswaveIntro}
\begin{array}{l}
   \p_t^2 u(t,x) - \div\big(c(t,x)\nabla u(t,x)\big) = 0, \quad t \geq 0,\ x\in \R^d, \vspace{4pt} \\
   u(0,x) = u_0(x), \quad \p_t u(0,x) = u_1(x), \quad x\in \R^d.
\end{array}
\end{equation}
We are interested in the case of non-smooth coefficients $c(t,x)$ and initial data distributions. The reasons to get into algebras of generalized functions are twofold. First, it is well-known that the existence of classical solutions requires a minimal degree of smoothness of the coefficients, as does the application of the classical machinery of pseudodifferential operators and Fourier integral operators \cite{Colombini:79,Colombini:95, HallerH:2008,Taylor:91}. In intended applications in seismology and continuum mechanics, the coefficients might have jumps at irregular interfaces or might be continuous, but non-differentiable paths of a stochastic process, thereby lacking the kind of smoothness required in the classical theory \cite{HMO:2008,MOFIOSSA:2014}. Second, as the solution is expected to be as singular as the initial data, equation (\ref{eq:conswaveIntro}) involves a nonlinear interaction of coefficient singularities with singularities in the solution, hence a product of distributions.

There is a long history of results (e.g. \cite{GO:2011b,Haneletal:2013,HS:2012,LO:1991,O:1989,O:1992,O:2008}) that equations of this type can be uniquely solved in the Colombeau algebra $\cG$ of generalized functions. This is an algebra containing the space of distributions as a subspace. Non-smooth coefficients and distributional initial data can be imbedded in this algebra, and hence generalized solutions are available in this setting. Members of the Colombeau algebras are equivalence classes of families $(u_{\varepsilon})_{\varepsilon \in (0,1]}$ of smooth functions that are characterized by certain asymptotic properties as $\eps\downarrow 0$. Representatives of the Colombeau solution satisfy
\begin{equation}\label{eq:conswaveReg}
\begin{array}{l}
   \p_t^2 u_\eps(t,x) - \div\big(c_\eps(t,x)\nabla u_\eps(t,x)\big) = h_\eps(t,x), \quad t \geq 0,\ x\in \R^d, \vspace{4pt} \\
   u_\eps(0,x) = u_{0\eps}(x)+ h_{0\eps}(x), \quad \p_t u_\eps(0,x) = u_{1\eps}(x)+h_{1\eps}(x), \quad x\in \R^d
\end{array}
\end{equation}
where $h_\eps, h_{0\eps}, h_{1\eps}$ belong to a certain ideal. The question of regularity and propagation of singularities can be cast in the following form:
\begin{itemize}
\item[(Q)] Is it possible to detect the singular support of the generalized solution from the asymptotic properties of its representing families $(u_{\varepsilon})_{\varepsilon \in (0,1]}$ as $\eps\downarrow 0$?
\end{itemize}
In this respect, a subalgebra $\Ginf$ of $\cG$ has been introduced in \cite{O:1992} which satisfies $\Ginf\cap\cD' = \Cinf$. It plays the same role in the $\cG$-setting as $\Cinf$ does in the $\cD'$-setting. In fact, the essential steps of regularity theory and propagation of singularities for Colombeau solution have been taken in the past decades, when the coefficients are regular ($\Cinf$ or $\Ginf$ with some additional properties). This starts with elliptic and $\Ginf$-hypoelliptic regularity, regularity of solutions with $\Ginf$-regular initial data, propagation of the $\Ginf$-singular support in the case of constant and of $\Cinf$-coefficients. Further, a calculus of pseudodifferential operators and Fourier integral operators has been developed in the Colombeau setting that admits to deduce a number of results on the $\Ginf$-wave front set of the Colombeau solution when the coefficients are $\Ginf$-regular \cite{G:2005,G:2006,G:ISAAC07,GGO:2003,GH:2005,GH:2006,GHO:2009,GO:2011a,Haller:2009,H:2004a,H:2004b,HO:2004,NPS:1998,O:1992,Scarp:2004}.

The central issue of the paper concerns non-smooth coefficients. When imbedded in the Colombeau algebra, they should \emph{not} belong to the subalgebra $\Ginf$ and it should be possible to detect the singularities that they produce in the generalized solution. The difficulty arises that in most of the previous existence results in the Colombeau setting, coefficients are required to satisfy certain additional logarithmic or slow scale estimates as $\eps\downarrow 0$, which makes them $\Ginf$-regular. One way out is to change the scale on which regularity is described, as was done e.g. in \cite{HH:2001}. Another way out is to consider only time-dependent coefficients. A number of results on propagation of singularities for scalar transport equations (including refined properties of the $\Ginf$-wave front set at characteristic curves with kinks) and on one-dimensional wave equations with time dependent coefficients have been obtained in the literature \cite{D:2011,DHO:2013,O:2008}. What has been lacking so far is the important case of $x$-dependent coefficients.

At interfaces of discontinuity, incoming singularities produce a reflected and a transmitted wave. In view of the intended applications in seismology and continuum mechanics, it is crucial to be able to detect the $\Ginf$-singularities carried by these waves. The present paper closes this gap in the literature by determining precise upper and lower bounds on the $\Ginf$-singular support of the solution in a number of paradigmatic cases.

The paper has essentially two parts. In the first part, existence results on wave equations of the type (\ref{eq:conswaveIntro}) in any space dimension are given that do not require a logarithmic regularization scale for the coefficients. Similar results are proven for the one-dimensional, non-conservative case, one-dimensional hyperbolic systems, together with a new existence result (in the algebra $\cG$) for multidimensional wave equations with time dependent coefficients. In the second part, upper and lower bounds on the $\Ginf$-singular support of the generalized solution to the one-dimensional wave equation with an initial point singularity and an $x$-dependent coefficient with a jump singularity are given. The same is done for rotationally symmetric solutions to multidimensional wave equations with $t$-dependent coefficients. The paper is complemented by comparing the generalized solution with the distributional solution of the corresponding transmission problem, when it exists. The paper shows that in paradigmatic cases the Colombeau solution carries the required information on the propagation of singularities and indeed answers question (Q) affirmatively in these cases. This indicates that the Colombeau setting may be appropriate for studying wave propagation phenomena in even more singular media, where no classical counterpart exists.

The plan of the paper is as follows. In Section 2, required concepts of the theory of Colombeau generalized functions are recalled. Section 3 is devoted to new existence and uniqueness results for wave equations with non-smooth coefficients in Colombeau algebras. The main goal is to show that in many situations bounds on the coefficients from above and away from zero suffice -- no logarithmic asymptotic estimates are needed. In Subsection 3.1, this is done for the conservative wave equations with $x$-dependent propagation speed in any space dimension, in Subsection 3.2 for the non-conservative wave equation in one space dimension, in Subsection 3.3 for the wave equation with time-dependent propagation speed in any dimension. Subsection 3.4 on scalar transport equations paves the way for the results on hyperbolic systems in one space dimension and $x$-dependent coefficients in Subsection 3.5. Subsection 3.6 is a side remark on singularities produced by a discontinuity of the coefficient at time $t=0$, if compatibility conditions are lacking, and that this classical phenomenon can also be detected in the Colombeau setting.

Section 4 aims at proving upper \emph{and} lower bounds on the singularities of the Colombeau solution to wave equations, given initial data with a point singularity. Subsection 4.1 treats the one-dimensional case with a coefficient having a jump at $x=0$ and  with a delta function at $x=-1$ as initial data. We prove that the reflected and the transmitted singularity can be detected in terms of the $\Ginf$-singular support of the Colombeau solution, provided the height of the jump satisfies certain bounds. It is also shown that there is no reflected singularity if the coefficient is $\Ginf$-regular. Subsection 4.2 addresses the one-dimensional wave equation with a $t$-dependent coefficient having a jump at $t=1$. As a preliminary step, a regularity result is proven that states that if the initial data are $\Ginf$-regular, then so is the solution off the line $t=1$. Then we prove an upper bound on the singular support of the solutions for arbitrary Colombeau initial data which are $\Ginf$-regular off the origin, generalizing our results from \cite{DHO:2013}. The upper bound consists of the rays emanating from the origin at $t=0$ and the refracted and transmitted rays starting at the singularity at $t=1$. (In \cite{DHO:2013} it was shown that the bound is attained for delta function initial data.) Next, we give sufficient conditions on the initial data such that singularities propagate out from the origin and remain present along the transmitted rays. Again, this is complemented by precise bounds on the singular support of the solution if the coefficient is $\Ginf$-regular. Finally, Subsection 4.3 settles the rotationally symmetric case with $t$-dependent coefficient and delta function initial data in any space dimension. The Appendix complements the paper by addressing the question of existence of associated distributions (distributional limits of the family of regularizations defining the Colombeau solution). Extending our results from \cite{GO:2011b}, we show that in one space dimension, an $x$-dependent coefficient with a single jump at the origin and delta function initial data, the Colombeau solution admits the distributional solution of the corresponding classical transmission problem as associated distribution. On the side, it is seen that the $\Cinf$-singular support of the distributional solution coincides with the $\Ginf$-singular support of the Colombeau solution. As a final note, we remark that most of our existence results are based on $L^2$-versions of Colombeau algebras, due to the energy methods used in the proofs. (Exceptions are hyperbolic systems in one space dimension and the one-dimensional wave equation with $t$-dependent coefficient, in which cases the usual $L^\infty_{\rm loc}$-based algebras suffice.)

\section{Basic concepts from Colombeau theory}\label{sec : 2}

We will employ the {\it special Colombeau algebra} of generalized functions denoted by $\G^{s}$ in \cite{GKOS:2001} (called the {\it simplified Colombeau algebra} in \cite{B:1990}). However, here we will simply use the letter $\G$ instead. In the context of energy estimates, we will also need $L^p-L^q$-based Colombeau algebras, introduced in \cite{BO:1992}.
This section serves to recall the definition and properties required for our purpose. For more details, see e. g. \cite{BO:1992, GKOS:2001, O:1992}.

Given a non-empty open subset $\Omega$ of $\mathbb{R}^n$, the space of real valued, infinitely differentiable functions on $\Omega$ is denoted by $\Cinf(\Omega)$, while $\Cinf(\overline{\Omega})$ refers to the subspace of functions all whose derivatives have a continuous extension up to the closure of $\Omega$.

Let $\Cinf(\Omega)^{(0,1]}$ be the differential algebra of all maps from the interval $(0,1]$ into $\Cinf(\Omega)$. Thus each element of $\Cinf(\Omega)^{(0,1]}$ is a family $(u_{\varepsilon})_{\varepsilon \in (0,1]}$ of real valued smooth functions on $\Omega$. The subalgebra $\EM(\Omega)$ is defined by the elements $(u_{\varepsilon})_{\varepsilon \in (0,1]}$ of $\Cinf(\Omega)^{(0,1]}$ with the property that, for all $K \Subset \Omega$ and $\alpha \in \mathbb{N}_0^n$, there exists $N \ge 0$ such that
\[
	\sup_{x \in K} |\partial^{\alpha} u_{\varepsilon}(x)| = O(\varepsilon^{-N}) \quad {\rm as}\ \varepsilon \downarrow 0.
\]
The ideal $\Neg(\Omega)$ is defined by all elements $(u_{\varepsilon})_{\varepsilon \in (0,1]}$ of $\Cinf(\Omega)^{(0,1]}$ with the property that, for all $K \Subset \Omega$, $\alpha \in\mathbb{N}_0^n$ and $N \ge 0$,
\[
	\sup_{x \in K} |\partial^{\alpha} u_{\varepsilon}(x)| = O(\varepsilon^N) \quad {\rm as}\ \varepsilon \downarrow 0.
\]
{\it The algebra} $\G(\Omega)$ {\it of generalized functions} is defined by the factor space
\[
	\G(\Omega) = \EM(\Omega) / \Neg(\Omega).
\]
The Colombeau algebra $\G(\overline{\Omega})$ on the closure of $\Omega$ is constructed in a similar way: the compact subsets $K$ occurring in the definition are now compact subsets of $\overline{\Omega}$, i.e., may reach up to the boundary. Since $\EM(\overline{\Omega}) \subset \EM(\Omega)$ and $\Neg(\overline{\Omega}) \subset \Neg(\Omega)$, there is a canonical map $\G(\overline{\Omega}) \to \G(\Omega)$. However, this map is not injective, as follows from the fact that $\Neg(\Omega)\cap \EM(\overline{\Omega}) \ne \Neg(\overline{\Omega})$.

Turning to $L^p-L^q$-based Colombeau algebras, we assume for convenience that the open set $\Omega$ has the strong local Lipschitz property (cf.\;\cite{A:1975}), in order to have all required Sobolev imbedding and extension theorems at our disposal.
Let $1 \leq p,q \leq \infty$ and $m \in \Z$. Concerning Sobolev spaces, we employ the usual notations: $W^{m,p}(\Omega)$, $H^{m}(\Omega)=W^{m,2}(\Omega)$, and $W^{\infty,p}(\Omega)=\cap_{m}W^{m,p}(\Omega)$, $W^{-\infty,p}(\Omega)=\cup_{m}W^{-m,p}(\Omega)$.
Due to the Sobolev imbedding theorems,
\[
   W^{\infty,p}(\Omega) \subset W^{\infty,\infty}(\Omega)
\]
for all $p$, and consequently
\[
   W^{\infty,p}(\Omega) \subset W^{\infty,q}(\Omega)
\]
for $p\leq q$; further, these spaces are differential algebras.

Let $\cE_{p}(\Omega)$ be the differential algebra of all maps from the interval $(0,1]$ into $W^{\infty,p}(\Omega)$. The subalgebra $\cE_{M,p}(\Omega)$ comprises the elements $(u_{\varepsilon})_{\varepsilon \in (0,1]}$ of $\cE_{p}(\Omega)$ with the property that, for all $\alpha \in \mathbb{N}_0^n$, there exists $N \ge 0$ such that
\[
	\|\p^{\alpha}u_\varepsilon\|_{L^p(\Omega)} = O(\varepsilon^{-N}) \quad {\rm as}\ \varepsilon \downarrow 0.
\]
Further, denote by  $\cN_{p,q}(\Omega)$ the set of elements $(u_{\varepsilon})_{\varepsilon \in (0,1]}$ of
$\cE_{M,p}(\Omega) \cap \cE_{q}(\Omega)$ with the property that for all $\alpha \in \mathbb{N}_0^n$ and $N \ge 0$,
\[
	\|\p^{\alpha}u_\varepsilon\|_{L^q(\Omega)} = O(\varepsilon^{N}) \quad {\rm as}\ \varepsilon \downarrow 0.
\]
It can be shown \cite{BO:1992} that $\cN_{p,q}(\Omega)$ is a differential ideal in $\cE_{M,p}(\Omega)$, thus
\[
\cG_{p,q}(\Omega) = \cE_{M,p}(\Omega)/\cN_{p,q}(\Omega)
\]
is a differential algebra. It has the following properties:
(a) For $1 \leq p_{1} \leq p_{2} \leq \infty$ and $ 1 \leq q \leq \infty$,
\[
\cG_{p_{1},q}(\Omega) \subset \cG_{p_{2},q}(\Omega).
\]
(b) There exists a canonical mapping from $\G_{p,q}(\Omega)$ into $\G(\Omega)$, which, however, is not injective.

Since $\Omega$ is assumed to have the strong local Lipschitz property, $W^{\infty,p}(\Omega) \subset \Cinf(\overline{\Omega})$. It follows that $\cE_{p}(\Omega) = \cE_{p}(\overline{\Omega})$, thus replacing $\Omega$ by $\overline{\Omega}$ makes no difference in the $L^p-L^q$-based Colombeau algebras:
\[
    \cG_{p,q}(\Omega) = \cG_{p,q}(\overline{\Omega}).
\]
This holds, in particular, when $\Omega$ is a strip $(0,T) \times \mathbb{R}^d$. If $\Omega$ is bounded, we have in addition
\[
    \cG_{p,q}(\Omega) = \cG_{p,q}(\overline{\Omega}) = \cG(\overline{\Omega}),
\]
because then all $L^p$-norms are equivalent with the $L^\infty$-norm, due to the Sobolev imbedding theorem and the boundedness of $\Omega$.

Summarizing, one should keep in mind that $\cG(\Omega)$ is defined by local estimates, $\cG(\overline{\Omega})$ by semi-global estimates (up to the boundary, but local on unbounded parts of $\Omega$), while $\cG_{p,q}(\Omega)$ is defined by global estimates.

We use capital letters for elements of $\G$ to distinguish generalized functions from distributions and denote by $(u_{\varepsilon})_{\varepsilon \in (0,1]}$ a representative of $U \in \G$. Recall that for any $U$, $V \in \G$ and $\alpha \in \mathbb{N}_0^n$, the partial derivative $\partial^{\alpha} U$ is defined as the class of $(\partial^{\alpha}u_{\varepsilon})_{\varepsilon \in (0,1]}$ and the product $UV$ is defined as the class of $(u_{\varepsilon}v_{\varepsilon})_{\varepsilon \in (0,1]}$. Also, for any $U =$\ class of $(u_{\varepsilon}(t,x))_{\varepsilon \in (0,1]} \in \G([0,\infty)\times\mathbb{R^d})$, we can define its restriction $U|_{t = 0} \in \G(\mathbb{R^d})$ to the hyperplane line $\{t = 0\}$ to be the class of $(u_{\varepsilon}(0,x))_{\varepsilon \in (0,1]}$.

Restrictions to open subsets are defined similarly. The support $\supp U$ of a generalized function $U\in\cG(\Omega)$ is the complement of the largest open subset of $\Omega$ on which $U$ vanishes.
\begin{remark}
For $U\in\cG_{p,q}(\Omega)$ and $V\in\cG_{\infty,q}(\Omega)$, the product $UV$ is a well-defined element of $\cG_{p,q}(\Omega)$.
\end{remark}

\begin{remark}
The algebra $\G(\Omega)$ contains the space $\mathcal{E}^{\prime}(\Omega)$ of compactly supported distributions. In fact, the map
\[
	f \mapsto \iota(f) = {\rm class\ of}\ (f \ast \rho_{\varepsilon}\mid_{\Omega})_{\varepsilon \in (0,1]}
\]
defines an imbedding of $\mathcal{E}^{\prime}(\Omega)$ into $\G(\Omega)$, where
\[
	\rho_{\varepsilon}(x) = \dfrac{1}{\varepsilon^n} \rho \left(\dfrac{x}{\varepsilon}\right)
\]
and $\rho$ is a fixed element of $\mathcal{S}(\mathbb{R}^n)$ such that $\int \rho(x)\,dx = 1$ and $\int x^{\alpha}\rho(x)\,dx = 0$ for any $\alpha \in \mathbb{N}_0^n$, $|\alpha| \ge 1$. This can be extended to an imbedding $\iota$ of the space $\D^{\prime}(\Omega)$ of distributions (the extension is unique as a sheaf morphism). Moreover, this imbedding turns $\Cinf(\Omega)$ into a subalgebra of $\G(\Omega)$. By an anologous imbedding, $W^{-\infty,p}(\mathbb{R}^n)$ is contained in $\G_{p,q}(\mathbb{R}^n)$, and $W^{\infty,q}(\mathbb{R}^n)$ is a subalgebra of $\G_{p,q}(\mathbb{R}^n)$ if $p\leq q$.
\end{remark}

\begin{definition}
A generalized function $U \in \G(\Omega)$ is said to be {\it associated with a distribution} $w \in \D^{\prime}(\Omega)$ if it has a representative $(u_{\varepsilon})_{\varepsilon \in (0,1]} \in \EM(\Omega)$ such that
\[
	u_{\varepsilon} \to w \quad {\rm in}\ \D^{\prime}(\Omega) \quad {\rm as} \ \varepsilon \downarrow 0.
\]
We write $U \approx w$ and call $w$ the {\it associated distribution} of $U$ provided $U$ is associated with $w$.
\end{definition}

Regularity theory for linear equations has been based on the subalgebra $\Ginf(\Omega)$ of {\it regular generalized functions} in $\G(\Omega)$ introduced in \cite{O:1992}. It is defined by all elements which have a representative $(u_{\varepsilon})_{\varepsilon \in (0,1]}$ with the property that, for all $K \Subset \Omega$, there exists $N \ge 0$ such that, for all $\alpha \in \mathbb{N}_0^n$,
\begin{equation}\label{eq:Ginf}
	\sup_{x \in K} |\partial^{\alpha} u_{\varepsilon}(x)| = O(\varepsilon^{-N}) \quad {\rm as}\ \varepsilon \downarrow 0.
\end{equation}
We observe that all derivatives of $u_{\varepsilon}$ have locally the same order of growth in $\varepsilon > 0$, unlike elements of $\EM(\Omega)$. This subalgebra $\Ginf(\Omega)$ has the property $\Ginf(\Omega) \cap \iota\big(\D^{\prime}(\Omega)\big) = \Cinf(\Omega)$, see \cite[Theorem 25.2]{O:1992}. Hence, for the purpose of describing the regularity of generalized functions, $\Ginf(\Omega)$ plays the same role for $\G(\Omega)$ as $\Cinf(\Omega)$ does in the setting of distributions. An element $U\in\cG(\Omega)$ is said to be $\Ginf$-regular on an open subset $\omega$ of $\Omega$, if its restriction $U\vert_\omega$ belongs to $\Ginf(\omega)$. An element $U\in\cG_{p,q}(\Omega)$ is said to be $\Ginf$-regular on $\omega$, if the canonical image of $U$ in $\cG(\omega)$ belongs to $\Ginf(\omega)$, or -- equivalently -- if $U$ has a representative $(u_{\varepsilon})_{\varepsilon \in (0,1]}$ in $\cE_{M,p}(\Omega)$ which satisfies the estimate (\ref{eq:Ginf}) on every compact subset $K$ of $\omega$. In either case, the $\Ginf$-singular support (denoted by $\singsupp_{\Ginf}$) of a generalized function is defined as the complement of the largest open set on which the generalized function is regular in the above sense.

We shall also make use of the global versions of regular generalized functions by $\cG_{p,q}^\infty(\Omega)$ defined by having a representative for which there exists $N\geq 0$ such that for all $\alpha \in \N_0^n$,
\[
     \|\p^\alpha u_\eps\|_{L^p(\Omega)} = O(\eps^{-N})\ \rm{as}\ \eps \downarrow 0.
\]
It is clear that every element of $\cG_{p,q}^\infty(\Omega)$ is $\Ginf$-regular in the previous sense, but a $\Ginf$-regular element of $\cG_{p,q}(\Omega)$ need not belong to the global space $\cG_{p,q}^\infty(\Omega)$.

We end this section by recalling various specialized asymptotic properties of nets and of generalized functions. A net $(r_{\varepsilon})_{\varepsilon \in (0,1]}$ is called a \emph{slow scale net} if
\[
	|r_{\varepsilon}|^{p} = O(\varepsilon^{-1})\quad {\rm as}\ \varepsilon \downarrow 0
\]
for every $p \ge 0$. It is called a \emph{log-type net} if
\[
	|r_{\varepsilon}|^{p} = O(|\log\eps|)\quad {\rm as}\ \varepsilon \downarrow 0.
\]
Further, if $|r_{\varepsilon}|^{p} = O(1)$ as $\varepsilon \downarrow 0$, the net is called \emph{bounded}. We refer to \cite{HO:2004} for a detailed discussion of slow scale and of logarithmic nets.

Finally, a generalized function $U$ from $\cG(\Omega)$ or from $\cG_{p,q}(\Omega)$ is called of \emph{bounded}, of \emph{logarithmic}, or of \emph{slow scale type}, if it has a representative $(u_{\varepsilon})_{\varepsilon \in (0,1]}$ such that $\sup_{x\in\Omega}|u_\eps(x)|$ is a bounded, log-type, or slow scale net, respectively. If such an estimate holds on every compact subset of $\Omega$, the generalized function is termed to be of \emph{locally bounded, locally logarithmic, or locally slow scale type}, respectively.

\begin{example}\label{ex : delta}
(a) Let $\varphi$ be a fixed element of $\Cinf(\mathbb{R})$ such that  $\varphi$ is symmetric, $\varphi^{\prime} \ge 0$ on $[-1,0]$, supp\,$\varphi \subset [-1,1]$ and $\int_{\mathbb{R}} \varphi(x)\,dx = 1$. Put $\varphi_{\varepsilon}(x) = \varphi(x/\varepsilon)/\varepsilon$. Then $U \in \G(\mathbb{R})$ defined by the class of $(\varphi_{\varepsilon})_{\varepsilon \in (0,1]}$ is associated with the delta function, and $\singsupp_{\Ginf} U = \{0\}$. \\
(b) On the other hand, if $U \in \G(\mathbb{R})$ is defined as the class of $(\varphi_{h(\varepsilon)})_{\varepsilon \in (0,1]}$, where $h(\eps) > 0$ and $(1/h(\varepsilon))_{\varepsilon \in (0,1]}$ is a slow scale net, then it is associated with the delta function again, but $\singsupp_{\Ginf} U = \emptyset$. In particular, this holds for the logarithmic regularization $h(\eps) = 1/|\log\eps|$.
\end{example}
At this stage, we are in the position to clarify the relations of the singular support of a Colombeau generalized function with the classical singular support of associated or imbedded distributions. First, the relation
$\Ginf(\omega) \cap \iota\big(\D^{\prime}(\omega)\big) = \Cinf(\omega)$ holds for every open subset $\omega$ of $\Omega$, whence
\[
   \singsupp_{\Ginf}\iota(f) = \singsupp f
\]
for every distribution $f\in \cD'(\Omega)$. However, if $U \in \G(\Omega)$ admits $w \in \D^{\prime}(\Omega)$ as associated distribution, there is no general relation between the $\Ginf$-singular support of $U$ and the singular support of $w$. Indeed, the generalized function $U$ from Example\;1
(b) provides an example of a generalized function whose $\Ginf$-singular support is empty, while its associated distribution has nonempty singular support. On the other hand, the $\Ginf$-singular support of the generalized function $V$ represented by $(\varepsilon\varphi_{\varepsilon}(x))_{\varepsilon \in (0,1]}$ is $\{0\}$, while its associated distribution (which equals zero) has empty singular support. Consequently, knowledge of the singular support of the associated distribution does not suffice to determine the $\Ginf$-singular support of a Colombeau generalized function.

Following \cite{NPS:1998,Scarp:2004}, an element $U \in \G(\Omega)$ is said to be \emph{strongly associated} with a distribution $w\in \cD'(\Omega)$, if for every $\psi\in\cD(\Omega)$, $|\int u_\eps(x)\psi(x)dx - \langle w,\psi\rangle| = \cO(\eps^p)$ for some $p>0$. Further, $U$ and $w$ are said to be \emph{equal in the sense of generalized distributions}, if this estimate holds for every $p>0$ (see \cite{C:1985}). If $U$ is strongly associated with $w$, then $\singsupp w \subset \singsupp_{\Ginf} U$. If $U$ and $w$ are equal in the sense of generalized distributions, then $\singsupp w = \singsupp_{\Ginf} U$. However, solutions of partial differential equations are rarely known to be strongly associated with a distribution.

As a last ingredient, we need the notion of a \emph{compactly bounded} or \emph{c-bounded} generalized function
(see \cite[Def.\;1.2.7]{GKOS:2001}). Only the scalar case is required here. An element $U\in\cG(\R)$ is called c-bounded, if it has a representative $(u_\varepsilon)_{\varepsilon \in (0,1]}$ with the property
\begin{equation}\label{cbounded}
\forall K\Subset\R\ \exists K'\Subset\R \ \mbox{such\ that\ } u_\eps(K)\subset K'
\end{equation}
for all sufficiently small $\eps$. The importance of the notion lies in the fact that composition of c-bounded generalized functions is possible. More precisely, if $V\in\cG(\R)$ and $U$ is c-bounded, then $V\circ U$ is a well defined element of $\cG(\R)$.

\section{Existence and uniqueness of generalized solutions}\label{sec : 3}

In order to be able to detect the $\Ginf$-singular support of a generalized solution to the wave equation with non-smooth coefficients, it is necessary to have existence and uniqueness results that do not require logarithmic bounds on the coefficients (as is the case with the classical existence results for hyperbolic systems in $\cG$). Indeed, if the (discontinuous) coefficients of the equation are regularized by convolution with a mollifier $(\varphi_{h(\varepsilon)})_{\varepsilon \in (0,1]}$ as in Example\;1
(b), they become $\Ginf$-regular generalized functions and contribute no $\Ginf$-singularities to the solution, in particular, no reflected or refracted waves, as will be seen below. It is the purpose of this section to provide existence and uniqueness results that can dispense with such slow scale or logarithmic regularizations and get along with standard mollifiers $(\varphi_{\varepsilon})_{\varepsilon \in (0,1]}$.

Indeed, there are a number of cases in which no logarithmic estimates are needed, when the setting of $\Gtwo$ is used: the conservative wave equation with propagation speed depending on $x$ only (in any space dimension), the one-dimensional wave equation with propagation speed depending on $x$, and the wave equation in any space dimension with $t$-depending propagation speed, bounded and of bounded variation. (The one-dimensional, $t$-dependent case was settled in \cite{DHO:2013}.) In all cases, a bound away from zero is required as well.

An alternative approach would consist in measuring the regularity of a generalized function based on an $h(\eps)$-scale, replacing the $\cO(\eps^{-N})$-bound in (\ref{eq:Ginf}) by an $\cO(h(\eps)^{-N})$-bound, as has been suggested in \cite{HH:2001}. This requires choosing $h(\eps)$ appropriately for the specific equation at hand. In this paper, we aim at removing the need for introducing a scale other than the standard scale $h(\eps) = \eps$ in the existence results, in order to be able to apply the standard $\Ginf$-regularity concepts.

\subsection{Existence theorems for conservative wave equations}
\label{subsec : 3.1}

We start with the conservative wave equation
\begin{equation}\label{eq:conswave}
\begin{array}{l}
   \p_t^2 u(t,x) - \div\big(c(x)\nabla u(t,x)\big) = 0, \quad t \geq 0,\ x\in \R^d, \vspace{4pt} \\
   u(0,x) = u_0(x), \quad \p_t u(0,x) = u_1(x), \quad x\in \R^d.
\end{array}
\end{equation}
When considering solutions whose representatives satisfy $L^2$-estimates in $t$ and $x$, we shall have to work on strips $[0,T]\times \R^d$, but with arbitrary $T>0$. Thus looking
for generalized solutions in the Colombeau algebra $\Gtwo([0,T]\times\R^d)$, we interpret Equation (\ref{eq:conswave}) as
\begin{equation}\label{eq:conswaveG}
\begin{array}{lr}
   \p_t^2 U - \div\big(C\nabla U\big) = 0 & {\rm in\ }  \Gtwo(\stripn)\vspace{4pt} \\
   U|_{t = 0} = U_0,\quad \partial_tU|_{t = 0} = U_1 & \mbox{in\ }  \G_{2,2}(\mathbb{R}^d)
\end{array}
\end{equation}
where the coefficient $C$ is a generalized function depending on $x\in\R^d$ only. We make the following assumptions about $C$.
\begin{itemize}
\item[(A1)] $C$ belongs to $\cG_{\infty,2}(\R^d)$, that is, it has a representative $(c_\eps)_{\eps \in (0,1]}$ all whose derivatives have global moderate $L^\infty$-bounds;
\item[(A2)] $C$ is strictly positive, that is, there is $m>0$ such that, for every representative,
    $0 < \eps^m \leq \inf_{x\in\R^d}\,c_\eps(x)$ as $\eps\downarrow 0$.
\end{itemize}
\begin{theorem}\label{thm:conswave}
Assume that $U_0, U_1 \in\Gtwo(\R^d)$ and that the assumptions $(A1)$ and $(A2)$ hold. Let $T>0$. Then the conservative wave equation $(\ref{eq:conswaveG})$ has a unique solution $U \in \Gtwo([0,T]\times \R^d)$.
\end{theorem}

\begin{proof}
Given representatives of $U_0$, $U_1$ and $C$, classical existence theory provides a prospective representative
$(u_\eps)_{\eps \in (0,1]}$ where each $u_\eps$ belongs to $H^\infty(\stripn)$. It satisfies the equation
\begin{equation}\label{eq:conswave_eps}
\begin{array}{l}
   \p_t^2 u_\eps(t,x) - \div\big(c_\eps(x)\nabla u_\eps(t,x)\big) = 0, \quad t\in[0,T],\ x\in \R^d, \vspace{4pt} \\
   u_\eps(0,x) = u_{0\eps}(x), \quad \p_t u_\eps(0,x) = u_{1\eps}(x), \quad x\in \R^d.
\end{array}
\end{equation}
We have to prove that $(u_\eps)_{\eps \in (0,1]}$ and all its derivatives are moderate in the $L^2$-norms.

Multiply the differential equation (\ref{eq:conswave_eps}) by $\p_t u_\eps$ and integrate by parts (with respect to $x$):
\[
    \int_{\R^d}\big(\p_t^2u_\eps\,\p_tu_\eps + c_\eps(x) \nabla u_\eps\cdot\nabla\p_t u_\eps\big)\,dx = 0,
\]
\[
   \frac12\frac{d}{dt}\int_{\R^d}\big(|\p_tu_\eps|^2 + c_\eps(x) |\nabla u_\eps|^2\big)\,dx = 0.
\]
This gives a moderate bound on the $L^2$-norms of $\p_tu_\eps$ and $\nabla u_\eps$ in terms of the initial data. Since
$u_\eps(t,x) = u_{0\eps} + \int_0^t\p_t u_\eps(s,x)\,d s$, this also entails moderate bounds on the $L^2$-norm of $u_\eps$ on the strip $\stripn$.

Next, $\p_tu_\eps$ satisfies the wave equation as well, thus the same argument gives moderate bounds on the $L^2$-norms of
$\p_t^2u_\eps$ and $\p_t\nabla u_\eps$. From the equation we obtain
\begin{equation}\label{eq:waveexpanded}
   \p_t^2 u_\eps - c_\eps\Delta u_\eps - \nabla c_\eps\cdot\nabla u_\eps = 0
\end{equation}
and in turn moderate bounds on the $L^2$-norm of $\Delta u_\eps$, involving also the $L^\infty$-norm of $\nabla c_\eps$.
But $\|u_\eps\|_{L^2(\R^d)}+ \|\nabla u_\eps\|_{L^2(\R^d)} + \|\Delta u_\eps\|_{L^2(\R^d)}$ is an equivalent norm on $H^2(\R^d)$, thus a moderate bound on the $H^2$-norm of $u_\eps$ on $\stripn$ is obtained.

Next, we are aiming at moderate bounds on the $H^3$-norm of $u_\eps$. Considering that $\p_t^2u_\eps$ satisfies again the wave equation (\ref{eq:conswave_eps}), one gets moderate bounds on $\p_t^3u_\eps$ and $\nabla\p_t^2u_\eps$. Denoting by $\p$ one of the $x$-derivatives $\p_{x_i}$, we get from (\ref{eq:waveexpanded}):
\[
   \p_t^2\p u_\eps - c_\eps\Delta\p u_\eps - (\p c_\eps)\Delta u_\eps - \nabla\p c_\eps\cdot\nabla u_\eps - \nabla c_\eps\cdot \nabla\p u_\eps = 0.
\]
All terms except $c_\eps\Delta\p u_\eps$ are already known to have moderate bounds by the previous steps, thus the equation entails a moderate bound on the $L^2$-norm of $\Delta\p u_\eps$, and in turn on the $H^2$-norm of $\p u_\eps$. It remains to estimate terms of the form $\p_{x_i}\p_{x_j}\p_t u_\eps$. But an estimate of $\Delta\p_t u_\eps$ follows again from (\ref{eq:waveexpanded}), and the equivalent-norms-argument completes the step. Collecting all terms, we see that the $H^3$-norm of $u_\eps$ has moderate bounds.
In the same way, the higher order derivatives can be estimated.

To prove uniqueness, we consider the inhomogeneous equation
\begin{equation}
\begin{array}{l}
   \p_t^2 u_\eps(t,x) - \div\big(c_\eps(x)\nabla u_\eps(t,x)\big) = h_\eps(t,x), \quad t\in[0,T],\ x\in \R^d, \vspace{4pt} \\
   u_\eps(0,x) = h_{0\eps}(x), \quad \p_t u_\eps(0,x) = h_{1\eps}(x), \quad x\in \R^d
\end{array}
\end{equation}
where $h_\eps$ belongs to $\cN_{2,2}(\stripn)$ and $h_{0\eps}, h_{1\eps}$ belong to $\cN_{2,2}(\R^d)$. We have to prove that $u_\eps$ belongs to $\cN_{2,2}(\stripn)$ as well. Thanks to \cite[Prop. 3.4]{G:2005} it suffices to prove the negligibility of the $L^2$-norm of $u_\eps$ (no derivatives needed). As before, we get
\[
   \frac12\frac{d}{dt}\int_{\R^d}\big(|\p_tu_\eps|^2 + c_\eps(x) |\nabla u_\eps|^2\big)\,dx = \int_{\R^d} h_\eps\p_t u_\eps \,dx.
\]
Integrating, inserting the initial data and using the nonnegativity of $c_\eps$ gives
\begin{eqnarray*}
   \|\p_tu_\eps(t,\cdot)\|_{L^2(\R^d)}^2 &\leq& \|h_{1\eps}\|_{L^2(\R^d)}^2 + \|c_{\eps}\|_{L^\infty(\R^d)}\|\nabla h_{0\eps}\|_{L^2(\R^d)}^2 \\
   & +& \int_0^t\|h_\eps(s,\cdot)\|_{L^2(\R^d)}^2\,ds\ +\ \int_0^t\|\p_tu_\eps(s,\cdot)\|_{L^2(\R^d)}^2\,ds.
\end{eqnarray*}
By assumption, all the terms on the right hand side except the last one are negligible. By Gronwall's inequality, the negligibility of $\|\p_tu_\eps\|_{L^2(\stripn)}$ and then in turn of $\|u_\eps\|_{L^2(\stripn)}$ follows.
\end{proof}

If the initial data are $\Ginf$-regular, so is the solution, provided the coefficient $C\in \cG_{\infty,2}(\R^d)$ satisfies slow scale estimates, more precisely, assume that:
\begin{itemize}
\item[(A3)] $C$ and all its derivatives are of slow scale type, that is, there is a representative $(c_\eps)_{\eps \in (0,1]}$ all whose derivatives have global $L^\infty$-bounds in terms of a slow scale net;
\item[(A4)] $C$ is \emph{slow scale invertible}, that is, there is a slow scale net $(r_\eps)_{\eps \in (0,1]}$ such that, for every representative,
    $0 < 1/r_\eps \leq \inf_{x\in\R^d}\,c_\eps(x)$ as $\eps\downarrow 0$.
\end{itemize}

\begin{corollary}\label{cor:conswave}
Assume that $C$ satisfies the assumptions $(A3)$ and $(A4)$. Let $U_0, U_1\in\Gtwo^\infty(\R^d)$.
Then the solution $U$ of $(\ref{eq:conswaveG})$ belongs to $\Gtwo^\infty(\stripn)$.
\end{corollary}

\begin{proof}
It is quite clear by following the same arguments. At each step, products of derivatives of $c_\eps$ and $u_\eps$ occur (when computing the initial data and when solving for the highest order derivative), containing only two factors each. Take $p_j > 0$, $\sum_{j=1}^\infty p_j = 1$, where each $p_j$ corresponds to one of the recursion steps. Any derivative of $c_\eps$ is of order $O(\eps^{-p_j})$ for whatever $j$. At each step, the previously estimated terms are multiplied by a factor of order $O(\eps^{-p_j})$. Hence all derivatives of $u_\eps$ will be of order $O(\eps^{-N-1})$, where $O(\eps^{-N})$ is the order of the $\Ginf$-initial data.
\end{proof}

\begin{remark}
The equation
\begin{equation}
\begin{array}{l}
   \p_t^2 u(t,x) - \div\big(c(x)\nabla u(t,x)\big) + \Phi(x) u(t,x) = 0, \quad t\in[0,T],\ x\in \R^d, \vspace{4pt} \\
   u(0,x) = u_0(x), \quad \p_t u(0,x) = u_1(x), \quad x\in \R^d
\end{array}
\end{equation}
with a non-negative potential $\Phi$ belonging to $\cG_{\infty,2}(\R^d)$ can be treated in the same way.
\end{remark}
\begin{remark}\label{rem:noncons}
In the non-conservative case
\begin{equation}\label{eq:nonconswave}
\begin{array}{l}
   \p_t^2 u_\eps(t,x) - c_\eps(x)\Delta u_\eps(t,x) = 0, \quad t\in[0,T],\ x\in \R^d, \vspace{4pt} \\
   u_\eps(0,x) = u_{0\eps}(x), \quad \p_t u_\eps(0,x) = u_{1\eps}(x), \quad x\in \R^d
\end{array}
\end{equation}
the energy estimate leads to
\[
    \int_{\R^d}\big(\p_t^2u_\eps\,\p_t u_{\varepsilon} + c_\eps(x) \nabla u_\eps\cdot\nabla\p_t u_\eps + \nabla c_\eps\cdot\nabla u_\eps\p_t u_\eps\big)\,dx = 0
\]
and so
\begin{equation}\label{eq:nonconsenergy}
   \frac12\frac{d}{dt}\int_{\R^d}\big(|\p_tu_\eps|^2 + c_\eps(x) |\nabla u_\eps|^2\big)\,dx \leq \int_{\R^d} |\nabla c_\eps|\,\big(|\nabla u_\eps|^2 + |\p_t u_\eps|^2\big)\,dx.
\end{equation}
If the initial data and the $L^\infty$-norm of $|\nabla c_\eps|$ are of order $\eps^{-N}$, an application of the Gronwall inequality
results in a bound on the $H^1$-norm of $u_\eps$ of the form $\eps^{-N}\ee^{T\eps^{-N}}$, which clearly is not moderate. We see that logarithmic estimates on $\nabla c_\eps$ would be needed to get a moderate bound on the $H^1$-norm of $u_\eps$, at least when continuing with this energy estimate.

This problem can be circumvented in the case of one space dimension.
\end{remark}

\subsection{Existence theorems for non-conservative wave equations}\label{subsec : 3.2}

We consider the one-dimensional non-conservative wave equation
\begin{equation}\label{eq:nonconswave1D}
\begin{array}{l}
   \p_t^2 u(t,x) - c(x)\p_x^2 u(t,x) = 0, \quad t\geq 0,\ x\in \R, \vspace{4pt} \\
   u(0,x) = u_0(x), \quad \p_t u(0,x) = u_1(x), \quad x\in \R.
\end{array}
\end{equation}
Taking one derivative with respect to $x$ results in the conservative wave equation
\begin{equation}\label{eq:nonconswave1Dderived}
\begin{array}{l}
   \p_t^2\p_x u(t,x) - \p_x\big(c(x)\p_x\p_x u(t,x)\big) = 0, \quad t\geq 0,\ x\in \R, \vspace{4pt} \\
   u(0,x) = u_0'(x), \quad \p_t\p_x u(0,x) = u_1'(x), \quad x\in \R
\end{array}
\end{equation}
for $\p_xu$. Thus the results from the previous subsection can be applied, and this will be the path around the difficulty mentioned in Remark\;4
. Note that this argument relies on the fact that the space dimension equals one.

Again, we interpret Equation (\ref{eq:nonconswave1D}) as
\begin{equation}\label{eq:nonconswave1DG}
\begin{array}{lr}
   \p_t^2 U - C\p_x^2 U = 0 & {\rm in\ }  \Gtwo(\strip)\vspace{4pt} \\
   U|_{t = 0} = U_0,\quad \partial_tU|_{t = 0} = U_1 & \mbox{in\ }  \G_{2,2}(\mathbb{R})
\end{array}
\end{equation}
where the coefficient $C$ is a generalized function depending on $x\in\R$ only. The assumptions on the coefficient $C \in \cG_{\infty,2}(\R)$ are the same as in Subsection\;\ref{subsec : 3.1} with $n=1$.
\begin{theorem}\label{thm:nonconswave}
Assume that $U_0, U_1 \in\Gtwo(\R)$ and that the assumptions $(A1)$ and $(A2)$ hold. Let $T>0$. Then the non-conservative wave equation $(\ref{eq:nonconswave1DG})$ has a unique solution $U \in \Gtwo([0,T]\times \R)$.
\end{theorem}

\begin{proof} We begin by constructing a prospective representative $(u_\eps)_{\eps \in (0,1]}$ of the solution $U$. Observe that the classical solution $u_\eps$ to (\ref{eq:nonconswave}) belongs to $H^\infty(\strip)$, as follows from applying the energy estimate from Remark\;4 
to all derivatives. As lined out in Remark\;4
, we cannot use the energy estimate (\ref{eq:nonconsenergy}) directly to obtain moderate bounds on the $H^1$-norm of $u_\eps$. Rather, we employ the observation that $\p_x u_\eps(t,x)$ solves the conservative wave equation (\ref{eq:nonconswave1Dderived}). As shown in the proof of Theorem\;\ref{thm:conswave}, $(\p_x u_\eps)_{\eps \in (0,1]}$ belongs to $\EMtwo(\strip)$. It follows from the wave equation (\ref{eq:nonconswave}) and assumption (A1) that
\[
   \p_t^2 u_\eps(t,x) = c_\eps(x)\p_x^2 u_\eps(t,x)
\]
belongs to $\EMtwo(\strip)$ as well. Taking into account the initial data
\[
   \p_t u_\eps(0,x) = u_{1\eps}(x), \qquad u_\eps(0,x) = u_{0\eps}(x),
\]
a double integration finally shows that $(u_\eps)_{\eps \in (0,1]}$ belongs to $\EMtwo(\strip)$, and hence may serve as a representative of the generalized solution $U$.

Uniqueness is proven along the same lines.
\end{proof}
Global regularity follows immediately from Corollary\;\ref{cor:conswave}, because a generalized function $U$ has the $\Ginf$-property, whenever its derivatives from a certain order onwards have the $\Ginf$-property.
\begin{corollary}\label{cor:nonconswave}
Assume that $C$ satisfies the assumptions $(A3)$ and $(A4)$. Let $U_0, U_1\in\Gtwo^\infty(\R)$.
Then the solution $U$ of $(\ref{eq:nonconswave1DG})$ belongs to $\Gtwo^\infty(\strip)$.
\end{corollary}

\begin{remark} Construction of solutions in $\cG(\stripn)$. Every global $\EMtwo$-estimate implies local $\EM$-estimates. Thus the results of the previous subsections show that Equations\;(\ref{eq:conswaveG}) and (\ref{eq:nonconswave1DG}) possess solutions in
$\cG(\stripn)$, respectively in $\cG(\strip)$, provided the initial data have a representative in $\EMtwo$. This is the case, for example, if the initial data have compact support.

However, uniqueness is lost when $\Gtwo$ is replaced by $\cG$. As an example, consider the wave equation
\[
\begin{array}{l}
   \displaystyle \p_t^2 u_\eps(t,x) - \frac{1}{\eps^2}\p_x^2 u_\eps(t,x) = 0,  \vspace{4pt} \\
   u_\eps(0,x) = \varphi(x-1/\eps), \quad \p_t u_\eps(0,x) = 0
\end{array}
\]
where $\varphi \in \cD(\R)$ with $\varphi(0) = 2$, say. This net of functions represents the zero element in $\cG(\R)$. Nevertheless, the corresponding solution
\[
   u_\eps(t,x) = \frac12\big(\varphi(x - 1/\eps - t/\eps) + \varphi(x - 1/\eps + t/\eps)\big)
\]
does not belong to $\cN(\strip)$. Indeed, $u_\eps(1,0) = 1$ for all sufficiently small $\eps$. Note, however, that $\varphi(x-1/\eps)$ does not belong to $\cN_{2,2}(\R)$, hence does not represent the zero element of $\Gtwo(\R)$, for which case the solution has been proven to be unique.
\end{remark}

\subsection{The time-dependent case}\label{subsec : 3.3}

This subsection is devoted to the wave equation with coefficient $c(t)$ depending on time only. We consider
\begin{equation}\label{eq:timewave}
\begin{array}{l}
   \p_t^2 u(t,x) - c(t)\Delta u(t,x) = 0, \quad t\geq 0,\ x\in \R^d, \vspace{4pt} \\
   u(0,x) = u_0(x), \quad \p_t u(0,x) = u_1(x), \quad x\in \R^d.
\end{array}
\end{equation}
Looking for generalized solutions in the Colombeau algebra $\Gtwo(\stripn)$, we again interpret Equation (\ref{eq:timewave}) as
\begin{equation}\label{eq:timewaveG}
\begin{array}{lr}
   \p_t^2 U - C\Delta U = 0 & {\rm in\ }  \Gtwo(\stripn)\vspace{4pt} \\
   U|_{t = 0} = U_0,\quad \partial_tU|_{t = 0} = U_1 & \mbox{in\ }  \G_{2,2}(\mathbb{R}^d)
\end{array}
\end{equation}
where the coefficient $C$ is a generalized function depending on $t\in[0,\infty)$ only, about which we make the following assumptions. \begin{itemize}
\item[(A5)] $C$ belongs to $\cG_{\infty,2}[0,\infty)$, that is, it has a representative $(c_\eps)_{\eps \in (0,1]}$ all whose derivatives have global moderate $L^\infty$-bounds;
\item[(A6)] $C$ is bounded away from zero, that is, there is $b>0$ such that $0 < b \leq \inf_{t\in[0,\infty)}\,c_\eps(t)$ as $\eps\downarrow 0$;
\item[(A7)] $C'$ has global $L^1$-bounds in the sense that for all $T>0$ and some representative, $\|c_\eps'\|_{L^1(0,T)}$ is bounded independently of $\eps$ as $\eps\downarrow 0$.
\end{itemize}
\begin{theorem}\label{thm:timewave}
Assume that $U_0, U_1 \in\Gtwo(\R^d)$ and that the assumptions $(A5)-(A7)$ hold. Let $T>0$. Then the  wave equation $(\ref{eq:timewaveG})$ has a unique solution $U \in \Gtwo([0,T]\times \R^d)$.
\end{theorem}

\begin{proof}
Given representatives of $U_0$, $U_1$ and $C$, classical existence theory provides a representative
$(u_\eps)_{\eps \in (0,1]}$ where each $u_\eps$ belongs to $H^\infty(\stripn)$. It satisfies the equation
\begin{equation}\label{eq:timewave_eps}
\begin{array}{l}
   \p_t^2 u_\eps(t,x) - c_\eps(t)\Delta u_\eps(t,x) = 0, \quad t\in[0,T],\ x\in \R^d, \vspace{4pt} \\
   u_\eps(0,x) = u_{0\eps}(x), \quad \p_t u_\eps(0,x) = u_{1\eps}(x), \quad x\in \R^d.
\end{array}
\end{equation}
In order to prove that $(u_\eps)_{\eps \in (0,1]}$ and all its derivatives are moderate in the $L^2$-norms, multiply the differential equation (\ref{eq:timewave_eps}) by $\p_t u_\eps$ and integrate by parts (with respect to $x$):
\[
    \int_{\R^d}\big(\p_t^2u_\eps\,\p_tu_\eps + c_\eps(t) \nabla u_\eps\cdot\nabla\p_t u_\eps\big)\,dx = 0.
\]
Observing that
\[
   \p_t\big(c_\eps(t)|\nabla u_\eps|^2\big) = c_\eps(t)\p_t|\nabla u_\eps|^2 + c_\eps'(t)|\nabla u_\eps|^2 = 2c_\eps(t) \nabla u_\eps\cdot\nabla\p_t u_\eps + c_\eps'(t)|\nabla u_\eps|^2
\]
we arrive at
\[
   \frac12\frac{d}{dt}\int_{\R^d}\big(|\p_tu_\eps|^2 + c_\eps(t) |\nabla u_\eps|^2\big)\,dx = c_\eps'(t)\frac12\int_{\R^d}|\nabla u_\eps|^2\,dx
\]
whence
\[
   \int_{\R^d} |\nabla u_\eps|^2\,dx \leq \dfrac{1}{b}\int_{\R^d} (|u_{1\eps}|^2+c_\eps(0)|\nabla u_{0\eps}|^2)\,dx + \int_0^t\frac{|c_\eps'(s)|}{b}\int_{\R^d}|\nabla u_\eps|^2\,dx\,ds.
\]
By Gronwall's inequality,
\[
   \int_{\R^d} |\nabla u_\eps|^2\,dx \leq \dfrac{1}{b}\int_{\R^d} (|u_{1\eps}|^2+c_\eps(0)|\nabla u_{0\eps}|^2)\,dx\exp(\int_0^t\frac{|c_\eps'(s)|}{b}\,ds)
\]
which is finite by assumption. From
\[
   \int_{\R^d}|\p_tu_\eps|^2 \,dx \leq \int_{\R^d}(|u_{1\eps}|^2+c_\eps(0)|\nabla u_{0\eps}|^2) \,dx + \int_0^t|c_\eps'(s)|\int_{\R^d}|\nabla u_\eps|^2\,dx\,ds
\]
we get a bound on the $L^2$-norm of $\p_t u_\eps$, and then in turn on $u_\eps$ on the strip $\stripn$.

Next, observe that all $x$-derivatives of $u_\eps$ satisfy the same equation (\ref{eq:timewave_eps}), so similar bounds are obtained on their $L^2$-norms. Bounds on the $t$- and mixed derivatives can be obtained successively from the equation.

Uniqueness is proven in the same way.
\end{proof}

\begin{remark}
(i) As can be seen from the proof, it would suffice to have a lower bound on $c_\eps$ in terms of the inverse of a logarithmic net (in place of the constant $b$).

(ii) Usually, the wave equation
\[
   \p_t^2 u(t,x) - c(t)^2\Delta u(t,x) = 0, \quad t\geq 0,\ x\in \R^d
\]
is written with $c(t)^2$ in place of $c(t)$. In this case one has to assume a uniform, global upper bound $B$ on $c_\eps(t)$, in addition, in order that the derivative of $c_\eps(t)^2$ remains integrable.
\end{remark}

\subsection{Scalar transport equations}
\label{subsec:transport}

Existence and uniqueness results for hyperbolic first order systems of partial differential and pseudodifferential equations have been obtained in the Colombeau setting in \cite{GO:2011b, LO:1991, O:1989, O:2008}, requiring certain logarithmic or slow scale estimates on the propagation speed. In the next subsection, we will show that these estimates can be replaced by requiring upper and lower bounds. The present subsection serves to prepare the arguments in the scalar case.

We consider the scalar transport equation
\begin{equation}\label{eq:transport1D}
\begin{array}{l}
  (\p_t + c(x)\p_x)u(t,x) = 0,  \vspace{4pt} \\
   u(0,x) = u_0(x)
\end{array}
\end{equation}
with transport velocity depending on $x$ only. We are interested in existence and uniqueness of a solution in the Colombeau algebra $\cG(\hpl)$. Thus we interpret problem (\ref{eq:transport1D}) as
\begin{equation}\label{eq:transport1DG}
\begin{array}{lr}
   \p_t U + C\p_x U = 0 & {\rm in\ }  \cG(\hpl)\vspace{4pt} \\
   U|_{t = 0} = U_0, & \mbox{in\ }  \cG(\R)
\end{array}
\end{equation}
where the coefficient $C$ is a generalized function depending on $x\in\R$ only. In the quoted earlier results, the coefficient $C$ was allowed to depend on $x$ and $t$, but it was required that $C$ was of bounded type and $\p_xC$ of logarithmic type. Here, the logarithmic type will be replaced by an estimate from below. We make the following assumptions.
\begin{itemize}
\item[(A8)] $C$ belongs to $\cG_{\infty,2}(\R)$, that is, it has a representative $(c_\eps)_{\eps \in (0,1]}$ all whose derivatives have global moderate $L^\infty$-bounds;
\item[(A9)] $C$ is bounded away from zero and is of bounded type, that is, there are constants $b_0, b_1 > 0$ such that, for every representative,
    $0 < b_0 \leq c_\eps(x)\leq b_1$ for all $x\in\R$ as $\eps\downarrow 0$.
\end{itemize}

Note that at fixed $\eps$ the coefficient $c_\eps(x)$ does not change sign on $\R$. The standard situation, arising from regularization of piecewise constant functions which are either positive everywhere or negative everywhere, is that $c_\eps(x)$ has the same sign for all $x$ and all $\eps$.

For $x\in \R$, define
\[
   C_\eps(x) = \int_0^x\frac{dy}{c_\eps(y)}.
\]
For the sake of exposition, consider the case when $c_\eps > 0$. Then the inequalities $0 < b_0 \leq c_\eps(x) \leq b_1$
imply that
\begin{eqnarray*}
   \frac{x}{b_1} \leq C_\eps(x) \leq \frac{x}{b_0} && (x > 0),\\
   \frac{x}{b_0} \leq C_\eps(x) \leq \frac{x}{b_1} && (x < 0).
\end{eqnarray*}
Thus $C_\eps$ is c-bounded (cf. (\ref{cbounded})) and surjective. In the case under discussion, it is also strictly increasing. Its inverse satisfies
\begin{eqnarray*}
   b_0y \leq C_\eps^{-1}(y) \leq b_1y && (y > 0),\\
   b_1y \leq C_\eps^{-1}(y) \leq b_0y && (y < 0).
\end{eqnarray*}
Thus $C_\eps^{-1}(y)$ is c-bounded, and its derivative is given by $\big(C_\eps^{-1}\big)'(y) = c_\eps\big(C_\eps^{-1}(y)\big)$.
Due to the c-boundedness of $C_\eps^{-1}(y)$, this is a well-defined generalized function. In conclusion, $(C_\eps)_{\eps\in(0,1]}$ represents an invertible generalized function, which is c-bounded, together with its inverse.

The characteristic curves $x(\tau) = \gamma_\eps(t,x,\tau)$ are obtained as solutions to the differential equation
\[
  \dot{x}(\tau) = c_\eps(x(\tau)),\qquad x(t) = \gamma_\eps(t,x,t) = x
\]
where the dot signifies differentiation with respect to $\tau$. Integrating this equality gives
\[
    \int_t^\tau\frac{ \dot{x}(s)\,ds}{c_\eps(x(s))} = \tau - t,
\]
or
\[
   C_\eps(x(\tau)) - C_\eps(x) = \tau - t,
\]
\[
   \gamma_\eps(t,x,\tau) = x(\tau) = C_\eps^{-1}\big(C_\eps(x) +\tau -t\big).
\]
Due to the c-boundedness of all participating generalized functions, this represents again a well-defined generalized function, which is c-bounded.
\begin{theorem}\label{thm:transport1DG}
Assume that $U_0 \in\G(\R)$ and that the assumptions $(A8)$ and $(A9)$ hold. Then problem $(\ref{eq:transport1DG})$ has a unique solution $U\in\cG(\hpl)$.
\end{theorem}
\begin{proof} A representative of the prospective solution is $u_\eps(t,x) = u_{0\eps}(\gamma_\eps(t,x,0))$. Since $\gamma_\eps$ is c-bounded, $u_\eps$ represents a well-defined element of $\cG(\hpl)$. A simple calculation shows that
\[
    \p_x\gamma_\eps(t,x,0) = c_\eps(\gamma_\eps(t,x,0))/c_\eps(x),\qquad
        \p_t \gamma_\eps(t,x,0) = -c_\eps(\gamma_\eps(t,x,0)),
\]
thus $u_\eps$ solves the differential equation.

To prove uniqueness, it suffices again to establish the null estimate of order zero. Thus assume that some moderate $v_\eps$
satisfies
\[
\begin{array}{l}
  (\p_t + c_\eps(x)\p_x)v_\eps(t,x) = h_\eps(t,x),  \vspace{4pt} \\
   v_\eps(0,x) = h_{0\eps}(x)
\end{array}
\]
where $h_\eps$ and $h_{0\eps}$ represent null elements. It holds that
\[
  v_\eps(t,x) = h_{0\eps}(\gamma_\eps(t,x,0)) + \int_0^t h_\eps(\tau,\gamma_\eps(t,x,\tau))\,d\tau.
\]
Due to the c-boundedness of $\gamma_\eps(t,x,\tau)$, null estimates of $h_\eps$ and $h_{0\eps}$ on compact sets translate into similar estimates on $v_\eps$, which hence is negligible.
\end{proof}
\begin{remark}
In the same way, existence and uniqueness of a solution in $\cG_{\infty,\infty}(\hpl)$ can be proven.
\end{remark}
\begin{example}
The boundedness of the coefficient $c_\eps(x)$ does not suffice to obtain a moderate representative of the solution, as shown by the following example in which $c_\eps(x)$ changes sign. Indeed, consider equation (\ref{eq:transport1D}) with
\[
   c_\eps(x) = -\tanh\frac{x}{\eps}.
\]
Again, the characteristic curves $x(\tau) = \gamma_\eps(t,x,\tau)$ are obtained as solutions to the differential equation
\[
  \dot{x}(\tau) = -\tanh\frac{x(\tau)}{\eps},\qquad x(t) = \gamma_\eps(t,x,t) = x
\]
Integrating this equality gives
\[
    \int_t^\tau\frac{ \dot{x}(s)\,ds}{\tanh\frac{x(s)}{\eps}}
       = \eps \int_{x/\eps}^{x(\tau)/\eps}\frac{dy}{\tanh y} = -(\tau - t),
\]
or
\[
   \log\left(\frac{\sinh \frac{x(\tau)}{\eps}}{\sinh \frac{x}{\eps}}\right) = \frac{t}{\eps} - \frac{\tau}{\eps}
\]
and finally
\[
   \gamma_\eps(t,x,\tau) = x(\tau) = \eps\Arsinh\big(\ee^{(t-\tau)/\eps}\,\sinh\frac{x}{\eps}\big).
\]
For the sake of simplicity, assume the initial data are classical, i.e., $u_0\in H^\infty(\R)$. Then the corresponding solution is given by
\begin{equation}\label{eq:solution_ex}
   u_\eps(t,x) = u_0(\gamma_\eps(t,x,0)) = u_0\Big(\eps\Arsinh\big(\ee^{t/\eps}\,\sinh\frac{x}{\eps}\big)\Big).
\end{equation}
Observe that
\[
   \p_x u_\eps(t,x) = u_0'\Big(\eps\Arsinh\big(\ee^{t/\eps}\,\sinh\frac{x}{\eps}\big)\Big)
       \left(1 + \big(\ee^{t/\eps}\,\sinh\frac{x}{\eps}\big)^2\right)^{-1/2}\ee^{t/\eps}\,\cosh\frac{x}{\eps}.
\]
Thus $\p_xu_\eps(t,0) = u_0'(0)\ee^{t/\eps}$ is not moderate if $u_0'(0)\neq 0$.

To further understand the behavior of the solutions $u_\eps(t,x)$, observe that they have a distributional limit, namely
\[
   u_\eps(t,x) \to \left\{\begin{array}{ll}
                   u_0(x+t), & x > 0,\\
                   u_0(x-t), & x < 0.
                   \end{array}\right.
\]
Indeed, it follows from l'H\^{o}pital's rule that
\[
  \lim_{\eps\downarrow 0} \eps\Arsinh\big(\ee^{t/\eps}\,\sinh\frac{x}{\eps}\big) = x + t\sign x
\]
for $t>0, x\neq 0$. As $u_0$ was supposed to be a classical bounded function, the assertion follows.
\end{example}
\begin{example}
Consider again equation (\ref{eq:transport1D}), taking the coefficient
\[
   c_\eps(x) = \tanh\frac{x}{\eps}
\]
of the opposite sign. In this case, the solution
\[
   u_\eps(t,x) =  u_0\Big(\eps\Arsinh\big(\ee^{-t/\eps}\,\sinh\frac{x}{\eps}\big)\Big)
\]
is moderate, more precisely, defines an element of $\EM(\hpl)$. This can be most easily seen by observing that all partial derivatives of the function $v(t,x) = \Arsinh\big(\ee^{-t}\,\sinh x\big)$ are bounded and that $u_\eps(t,x) = u_0\big(\eps v(t/\eps,x/\eps)\big)$. The solutions $u_\eps(t,x)$ again have a distributional limit, namely
\[
   u_\eps(t,x) \to \left\{\begin{array}{ll}
                   u_0(x-t), & x > t,\\
                   u_0(0), & |x| < t,\\
                   u_0(x+t), & x < -t.
                   \end{array}\right.
\]
Indeed, for $|x| < t$, we use that $|\sinh \frac{x}{\eps}| \leq \ee^{|x|/\eps}$, so $\eps\Arsinh\big(\ee^{-t/\eps}\,\sinh\frac{x}{\eps}\big) \to 0$ as $\eps\downarrow 0$. For $|x| > t$, $\Arsinh\big(\ee^{-t/\eps}\,\sinh\frac{x}{\eps}\big) \to \pm\infty$ and l'H\^{o}pital's rule gives again the result.
\end{example}

\subsection{Hyperbolic systems}

The results on transport equations can be extended to hyperbolic ($m\times m$)-systems in diagonal form
\begin{equation}\label{eq:system}
\begin{array}{l}
  (\p_t + c_i(x)\p_x)u_i(t,x) = \sum_{j=1}^m a_{ij}(t,x)u_j(t,x), \vspace{4pt} \\
   u_i(0,x) = u_{0i}(x)
\end{array}
\end{equation}
for $i = 1,\ldots, m$. Concerning the coefficients $c_i(x)$, the same assumptions as in the previous subsection are in effect. Concerning the coefficients $a_{ij}(t,x)$, logarithmic type has to be assumed as in earlier papers. The novelty of the result is that no logarithmic estimates on the derivatives of $c_i(x)$ are required.

Again, we interpret problem (\ref{eq:system}) in the Colombeau setting as
\begin{equation}\label{eq:systemG}
\begin{array}{lr}
   \p_t U_i + C_i\p_x U = \sum_{j=1}^m A_{ij}U_j & {\rm in\ }  \cG(\hpl)\vspace{4pt} \\
   U_i|_{t = 0} = U_{0i}, & \mbox{in\ }  \cG(\R)
\end{array}
\end{equation}
for $i = 1,\ldots, m$, where the coefficients $C_i$ are generalized functions depending on $x\in\R$ only.

\begin{theorem}\label{thm:systemG}
Assume that the initial data $U_{0i}$ belong to $\cG(\R)$, the coefficients $C_i$ satisfy assumptions $(A8)$ and $(A9)$, and the coefficients $A_{ij}$ belong to $\cG(\hpl)$ and are of local logarithmic type. Then problem $(\ref{eq:systemG})$ has a unique solution $u\in\cG(\hpl)$.
\end{theorem}

\begin{proof} Given representatives of $U_{0i}$, $C_i$ and $A_{ij}$, classical existence theory provides a representative
$(u_\eps)_{\eps \in (0,1]}$ of the prospective solution where each $u_\eps$ belongs to $\Cinf(\hpl)$ and solves equation (\ref{eq:system}), with $u_{0i\eps}$, $c_{i\eps}$ and $a_{ij\eps}$ in place of $u_{0i}$, $c_i$, $a_{ij}$. Integrating along characteristic curves, we may rewrite (\ref{eq:system}) as the system of integral equations
\[
     u_{i\eps}(t,x) = u_{0i\eps}(\gamma_{i\eps}(t,x,0))
          + \int_0^t \sum_{j=1}^m a_{ij\eps} u_{j\eps}(\tau,\gamma_{i\eps}(t,x,\tau))\,d\tau.
\]
Fix $T > 0$ and a compact interval $K\subset\R$ and let $K_T$ be the region bounded by $t=0$, $t=T$ and by the straight lines with slope $\sup c_{i1\eps}$ (left) and $\inf c_{i0\eps}$ (right), determined according to condition (A9). If a point $(t,x)$ belongs to $K_T$, all backward characteristic curves of the system connect it to a point in $K$. Thus we can take the $L^\infty$-norms on both sides of the system of integral equations on $K_T$. Gronwall's inequality and the logarithmic type of $a_{ij\eps}$ show that the $L^\infty$-norm of $u_{i\eps}$ on $K_T$ is moderate, $i = 1,\ldots, m$.

Next, $\p_tu_{i\eps}$ satisfies the equation
\begin{equation}\label{eq:dtsystem}
\begin{array}{rcl}
  (\p_t + c_{i\eps}(x)\p_x)\p_tu_{i\eps}(t,x)& = &\sum_{j=1}^m a_{ij\eps}(t,x)\p_tu_{j\eps}(t,x) \vspace{4pt} \\
        & + & \sum_{j=1}^m \p_t a_{ij\eps}(t,x)u_{j\eps}(t,x), \vspace{4pt} \\
   \p_t u_{i\eps}(0,x) & = & \sum_{j=1}^m a_{ij\eps}(0,x)u_{0j\eps}(x) - c_{i\eps}(x)u_{0i\eps}'(x).
\end{array}
\end{equation}
The initial data are still moderate; the second summand in the differential equation is known to be moderate from the previous step, and so the Gronwall argument can be applied to $\p_t u_{i\eps}$ to obtain its moderateness. Finally, equation (\ref{eq:system}) and the lower bound on $c_{i\eps}(x)$, due to assumption (A9), entail the moderateness of $\p_x u_{i\eps}$.

Differentiating the equation again, the moderateness of $\p_t^2 u_{i\eps}$ is obtained in the same way. The previous equation (\ref{eq:dtsystem}) now entails the moderateness of $\p_x\p_t u_{i\eps}$. Next, we differentiate the system (\ref{eq:system}) with respect to $x$ to obtain
\begin{equation}\label{eq:dxsystem}
\begin{array}{rcl}
  (\p_t + c_{i\eps}(x)\p_x)\p_xu_{i\eps}(t,x) & = &\sum_{j=1}^m a_{ij\eps}(t,x)\p_x u_{j\eps}(t,x) - c_{i\eps}'(x)\p_x u_{i\eps}(t,x)  \vspace{4pt} \\
         & + & \sum_{j=1}^m \p_x a_{ij\eps}(t,x)u_{j\eps}(t,x), \vspace{4pt} \\
   \p_x u_{i\eps}(0,x) & = & u_{0i\eps}'(x).
\end{array}
\end{equation}
At this stage, all terms in the equation (\ref{eq:dxsystem}) have already been proven to be moderate, except $c_{i\eps}(x)\p_x^2u_{i\eps}(t,x)$, which hence must be moderate as well. Using the lower bound on $c_{i\eps}(x)$ we arrive at the moderateness of $\p_x^2u_{i\eps}(t,x)$.

Next one proceeds step by step: first $\p_t^3 u_{i\eps}$ is estimated by differentiating the system of equations once more and using the Gronwall argument, then $\p_t^2\p_x u_{i\eps}$ by collecting terms in (\ref{eq:dtsystem}), differentiated with respect to $t$, then $\p_t\p_x^2 u_{i\eps}$ is shown to be moderate by collecting terms in (\ref{eq:dxsystem}), differentiated with respect to $t$, and finally $\p_x^3 u_{i\eps}$ by differentiating (\ref{eq:dxsystem}) with respect to $x$, and so on.

Uniqueness is proven along the same lines.
\end{proof}

Note that the argument depends essentially on the fact that the coefficients $C_i$ depend on one variable only. If the $C_i$ depended on both variables, the recursive procedure could not be initiated without logarithmic assumptions on the first derivatives of $C_i$.

\subsection{Singularities due to non-compatibility}

If the coefficients of a hyperbolic equation have a jump discontinuity at some point $x_0$, a singularity emanating from $(0, x_0)$ will be produced, in general. This happens even for smooth initial data, unless suitable compatibility conditions at the point $x_0$ are satisfied. This section serves to exhibit this phenomenon in a simple example and to emphasize that it may arise in general. In the subsequent sections we will not further elaborate on this effect and usually consider only data which actually vanish near the jumps of the coefficients.

For the purpose of illustration, consider the equation
\begin{equation}\label{eq:transportjump}
\begin{array}{l}
  (\p_t + c(x)\p_x)u(t,x) = 0,  \vspace{4pt} \\
   u(0,x) = u_0(x)
\end{array}
\end{equation}
where $c(x)$ is piecewise constant with a jump at $x=0$. The purpose of this subsection is to study the singularity emanating from the origin, both in the classical case and in the Colombeau framework, and point out compatibility conditions of the jump data with the initial data. In particular, we want to demonstrate in a simple example that the Colombeau solution is not $\Ginf$-regular along the characteristic curve emanating from the origin, in general.

For the sake of exposition, assume that
\[
   c(x) = \left\{
          \begin{array}{ll}
          1,& x < 0,\\
          2,& x > 0.
          \end{array}\right.
\]
\begin{figure}[htbp]
\begin{center}
\includegraphics[width = 7cm]{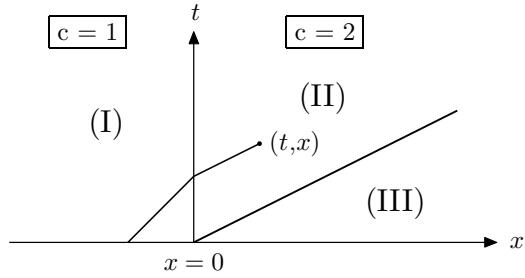}
\end{center}
\caption{Geometry of the transmission problem.}
\label{fig:CornerSingularity}
\end{figure}

We first recall the classical arguments, interpreting problem classically as a transmission problem. The standard transmission condition is that $u(t,x)$ should be continuous across $x=0$. Observe that in Figure\ref{fig:CornerSingularity} the backward characteristic through $(t,x)$ hits the $t$-axis at time $t-x/2$. Thus the solution of the transmission problem is
\[
   u(t,x) = \left\{
          \begin{array}{ll}
          u_0(x-t)& \mbox{in\  (I)},\\
           u_0(-t + x/2)& \mbox{in\ (II)},\\
            u_0(x-2t)& \mbox{in\ (III)}.
          \end{array}\right.
\]
Regions (II) and (III) are separated by the line $x=2t$. It is seen that the solution is continuous across this line.
Next,
\[
   \p_xu(t,x) = \left\{
          \begin{array}{ll}
          u_0'(x-t)& \mbox{in\  (I)},\\
           \frac12u_0'(-t + x/2)& \mbox{in\ (II)},\\
            u_0'(x-2t)& \mbox{in\ (III)}.
          \end{array}\right.
\]
Thus $\p_xu(t,x)$ is continuous across $x = 2t$ if and only if $u_0'(0) = 0$. It is never continuous across $x=0$, unless $u_0'$ vanishes on the negative half-axis.

Next, we want to show that the same phenomenon is observed in the Colombeau setting. We regularize the coefficient $c(x)$ by means of a nonnegative, standard mollifier $\varphi_\eps$ as in Example\;1
, i.e., $\varphi$ is smooth, nonnegative, has integral one, is symmetric, and its support is contained in $[-1,1]$, $\varphi_\eps(x) = \varphi(x/\eps)/\eps$. Let $C$ be the class of $(c\ast\varphi_\eps)_{\eps\in (0,1]}$ in $\cG(\R)$. Given initial data $U_0\in\cG(\R)$, problem (\ref{eq:transportjump}) has a unique solution $U\in\cG(\hpl)$. Assume the initial data belong to $\Ginf(\R)$. We want to show that the singularities across $x=0$ and $x = 2t$ can be detected, i.e., that $U$ is not $\Ginf$ across these lines, in general.

For the sake of the argument, we take $U_0$ as the $\Cinf$-function $u_0(x) = x$ and
\[
   c_\eps(x) = 1 + \int_{-\infty}^x \varphi_\eps(y)\,dy = 1 + \int_{-\infty}^{x/\eps} \varphi(y)\,dy.
\]
Recall from the Subsection\;\ref{subsec:transport} that the characteristic curves are given by
\[
   \gamma_\eps(t,x,\tau) = x(\tau) = C_\eps^{-1}\big(C_\eps(x) +\tau -t\big).
\]
Due to the simple initial data we use, the solution is given by
\[
   u_\eps(t,x) = \gamma_\eps(t,x,0).
\]
Recall further that
\[
    \p_x\gamma_\eps(t,x,0) = \frac{c_\eps(\gamma_\eps(t,x,0))}{c_\eps(x)},\qquad
        \p_t \gamma_\eps(t,x,0) = -c_\eps(\gamma_\eps(t,x,0)).
\]
We compute
\[
    \p_x^2\gamma_{\eps}(t,x,0) =  -\frac{c_{\eps}'(x)c_{\eps}(\gamma_{\eps}(t,x,0))}{c_{\eps}(x)^2}
          + \frac{c_{\eps}'(\gamma_{\eps}(t,x,0))\p_x\gamma_{\eps}(t,x,0)}{c_{\eps}(x)}.
\]
For $t>0$ and $x$ near $0$, we have that
\[
   c_\eps(\gamma_\eps(t,x,0)) \equiv 1,\qquad c_\eps '(\gamma_\eps(t,x,0)) \equiv 0.
\]
Thus for those points we obtain
\[
   \p_x\gamma_\eps(t,x,0) = \frac{1}{c_\eps(x)},\quad \p_x^2\gamma_\eps(t,x,0) =  -\frac{c_\eps'(x)}{c_\eps(x)^2},
\]\[
       \quad \p_x^3\gamma_\eps(t,x,0) = \frac{2(c_\eps'(x))^2}{c_\eps(x)^3} -\frac{c_\eps''(x)}{c_\eps(x)^2}.
\]
Due to the symmetry assumption on $\varphi$ we have that $c_\eps(0) = \frac32$. Further, $c_\eps'(x) = \frac1\eps\varphi(\frac{x}\eps)$, $c_\eps''(x) = \frac1{\eps^2}\varphi'(\frac{x}\eps)$, and so on. Thus if we take $\varphi(0) = a \neq 0$ and $\varphi'(0) = \varphi''(0) = \ldots = 0$, we obtain
\[
   \p_x\gamma_\eps(t,0,0) = \frac23, \quad \p_x^2\gamma_\eps(t,0,0) = -\frac{4a}{9\eps},\quad
    \p_x^3\gamma_\eps(t,0,0) = \frac{16a^2}{27\eps^2}
\]
and so on. Thus $u_\eps$ does not have the $\Ginf$-property near $x=0$.

To prove that the solution is not $\Ginf$ across $x = 2t$, we consider a point $(t,x)$ near this line.
We have
\[
  \p_t\gamma_\eps(t,x,0) = -c_\eps\big(\gamma_\eps(t,x,0)\big),\quad
     \p_t^2\gamma_\eps(t,x,0) = c_\eps'\big(\gamma_\eps(t,x,0)\big)c_\eps\big(\gamma_\eps(t,x,0)\big),
\]
\[
  \p_t^3\gamma_\eps(t,x,0) = -\big(c_\eps'\big(\gamma_\eps(t,x,0)\big)\big)^2c_\eps\big(\gamma_\eps(t,x,0)\big)
     - c_\eps''\big(\gamma_\eps(t,x,0)\big)\big(c_\eps\big(\gamma_\eps(t,x,0)\big)\big)^2
\]
and so on. Consider a point
\[
   x = \gamma_\eps(0,0,t)
\]
which lies exactly on the characteristic curve emanating from the origin. We have
\[
  \gamma_\eps(t,\gamma_\eps(0,0,t),0) = 0,
\]
thus along this characteristic curve,
\[
  \p_t\gamma_\eps = -c_\eps(0),\quad
     \p_t^2\gamma_\eps = c_\eps'(0)c_\eps(0), \quad
  \p_t^3\gamma_\eps = -\big(c_\eps'(0)\big)^2c_\eps(0)
      - c_\eps''(0)\big(c_\eps(0)\big)^2.
\]
As before, this expression is seen to violate the $\Ginf$-property. As $\eps\downarrow 0$, the curve $x = \gamma_\eps(0,0,t)$ converges to the line $x = 2t$. Hence in every neighborhood of a point on $x=2t$, the $\Ginf$-property is violated, and so $u_\eps$ is not $\Ginf$ across $x = 2t$.


\section{Propagation of singularities in wave equations}
\label{sec : 4}

We study the phenomenon of propagation of singularities in the generalized solution to the Cauchy problems for the one-dimensional wave equation with propagation speed depending on $x$
\begin{equation}\label{eq : wave with DC in x}
	\begin{array}{l}
		\partial_t^2u - c(x)^2\partial_x^2u = 0,  \quad t\in[0,T], \ x \in \mathbb{R}, \vspace{4pt}\\
		u|_{t = 0} = u_{0},\quad \partial_tu|_{t = 0} = u_{1}, \quad x \in \mathbb{R}
	\end{array}
\end{equation}
and for the wave equation in any space dimension with $t$-depending propagation speed
\begin{equation}\label{eq : wave with DC in t}
	\begin{array}{l}
		\partial_t^2u - c(t)^2\Delta u = 0,  \quad t\in[0,T], \ x \in \mathbb{R}^d, \vspace{4pt}\\
		u|_{t = 0} = u_{0},\quad \partial_tu|_{t = 0} = u_{1}, \quad x \in \mathbb{R}^d.
	\end{array}
\end{equation}
In either case, the coefficient $c$ is assumed to be strictly positive and piecewise constant. We work on a finite time interval $[0,T)$, but with arbitrary $T>0$. (See the comments at the beginning of Subsection\;\ref{subsec : 3.1}.)

As outlined in the Introduction, the issue is to find upper and lower bounds on the $\Ginf$-singular support, that is, to present methods that admit the detection of singularities in the generalized solution from its asymptotic behavior in $\eps$.

\subsection{The space-dependent case}\label{subsec : 4.1}

We begin with problem $(\ref{eq : wave with DC in x})$. The coefficient $c(x)$ is assumed to be given by
\[
	c(x) = c_0 + (c_1 - c_0)H(x),
\]
where $c_0$, $c_1 > 0$, $c_0 \ne c_1$ and $H$ is the Heaviside function. Specifically, we take
\[
u_0 \equiv 0,\qquad u_1 = \delta(x+1),
\]
where $\delta$ is the delta function. We define a generalized function $C \in \cG_{\infty,2}(\mathbb{R})$ by means of a representative $(c_{\varepsilon})_{\varepsilon \in (0,1]} = (c \ast \varphi_{\varepsilon})_{\varepsilon \in (0,1]}$ with a mollifier $\varphi_{\varepsilon}$ as in Example 1
. Then $\singsupp_{\Ginf} C = \{x=0\}$. We also produce $U_0 \equiv 0$ and $U_1 \in \G_{2,2}(\mathbb{R})$ as the class of $(\varphi_{\varepsilon}(x+1))_{\varepsilon \in (0,1]}$, where $\varphi_{\varepsilon}$ again is a mollifier as in Example 1
. (Actually, the mollifier need not be the same as the one chosen for the regularization of the coefficient $c$.) Thus we interpret problem $(\ref{eq : wave with DC in x})$ as the problem
\begin{equation}\label{eq : generalized wave with DC in x}
	\begin{array}{lr}
		\partial_t^2U - C^2\partial_x^2U = 0 & \mbox{in}\ \G_{2,2}([0,T]\times\mathbb{R}),\vspace{4pt} \\
		U|_{t = 0} = U_0,\quad \partial_tU|_{t = 0} = U_1 & \mbox{in}\ \G_{2,2}(\mathbb{R})
	\end{array}
\end{equation}
with $C$, $U_0$ and $U_1$ defined above. By Theorem\;\ref{thm:nonconswave}, problem $(\ref{eq : generalized wave with DC in x})$ has a unique solution $U \in \G_{2,2}([0,T] \times \mathbb{R})$ for any $T > 0$. We are going to show that, if the coefficient has a small jump, then reflection and refraction of the singularity actually do occur at the point of discontinuity of the coefficient  (see Figure \ref{fig : singsupp of U}).

\begin{figure}[htbp]
\begin{center}
\includegraphics{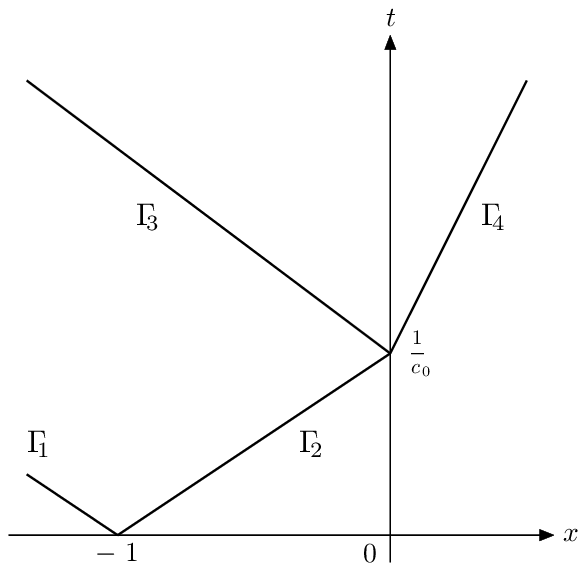}
\end{center}
\caption{The $\Ginf$-singular support of the solution $U$.}
\label{fig : singsupp of U}
\end{figure}

\begin{theorem}\label{thm : propagation1}
Let $c_0$, $c_1 > 0$ be such that $2 < \sqrt{c_0/c_1} + \sqrt{c_1/c_0} < 4$ and let $T > 1/c_0$. Furthermore,  let $C$, $U_0$ and $U_1$ be as described above and let $U \in \G_{2,2}([0,T] \times \mathbb{R})$ be the solution to problem $(\ref{eq : generalized wave with DC in x})$. Then it holds that
\begin{equation}\label{eq : singsupp of U1}
	\begin{split}
		\singsupp_{\Ginf} U
		& = \{(t,x) \mid x=-1-c_0 t,\ 0 \le t < T\} \\
		&\quad \cup \{(t,x) \mid x=-1+c_0 t,\ 0 \le t \le 1/c_0\} \\
		&\quad \cup \{(t,x) \mid x=1-c_0 t,\ 1/c_0 \le t < T\} \\
		&\quad \cup \{(t,x) \mid x=-c_1/c_0+c_1t,\ 1/c_0 \le t < T\} \\
		& =:  \Gamma_1 \cup \Gamma_2 \cup \Gamma_3 \cup \Gamma_4.
	\end{split}
\end{equation}
\end{theorem}

\begin{proof}
Assertion $(\ref{eq : singsupp of U1})$ will hold if we show that
\begin{equation}\label{eq : singsupp of U2}
	\singsupp_{\Ginf} U \subset \cup_{k=1}^{4} \Gamma_k.
\end{equation}
To see this, for the sake of simplicity, we assume that $T > 2/c_0$. We change the roles of $t$ and $x$ and consider $(U|_{x=-1},\partial_x U|_{x=-1})$ as initial data. By $(\ref{eq : singsupp of U2})$, the initial data $(U|_{x=-1},\partial_x U|_{x=-1})$ can have singularities only at $t=0$ and $t=2/c_0$. We observe how their singularities propagate as $x$ goes from $-1$ to $-\infty$ and from $-1$ to $\infty$. By d'Alembert's formula, $U$ has a representative $((1/2c_0)\int_{x+1-c_0t}^{x+1+c_0t}\varphi_{\varepsilon}(y)\,dy)_{\varepsilon \in (0,1]}$ near $(t,x) = (0,-1)$. Hence $\singsupp_{\Ginf} U|_{x=-1} \supset \{0\}$ and $\partial_x U|_{x=-1} = 0$ near $t = 0$. Let us focus on the singularity of $U|_{x=-1}$ at $t = 0$. By d'Alembert's formula again, it propagates along $\Gamma_1$ as $x$ goes from $-1$ to $-\infty$. As $x$ goes from $-1$ to $\infty$, it propagates along $\Gamma_2$ and then splits into a transmitted and a refracted wave at $x=0$. This can be shown similarly to the proof of Theorem 4.1 of \cite{DHO:2013}. Thus the singularity of $U|_{x=-1}$ at $t = 0$ propagates along $\Gamma_1$, $\Gamma_2$, $\Gamma_4$ and
\[
	\Gamma_5:=\{(t,x) \mid x=c_1/c_0- c_1t,\ 0\le t\le 1/c_0\}.
\]
Note that $\Gamma_5$ does not belong to $\singsupp_{\Ginf} U$. Hence the initial data $U|_{x=-1}$, $\partial_xU|_{x=-1}$ must have a singularity at $t=2/c_0$ which cancels out the singularity on $\Gamma_5$ emanating from $U|_{x=-1}$ at $t = 0$. This singularity at $t=2/c_0$ propagates along at least $\Gamma_3$ and $\Gamma_5$. It also may propagate along $\Gamma_4$. Even if that is the case, it does not cancel out the singularity on $\Gamma_4$ emanating from $U|_{x=-1}$ at $t = 0$, i.e., $\Gamma_4$ belongs to $\singsupp_{\Ginf} U$. In fact, if $\Gamma_4$ does not belong to $\singsupp_{\Ginf} U$, then $\singsupp_{\Ginf} U$ coincides with the union of $\Gamma_1$, $\Gamma_2$ and $\Gamma_3$. However, by Proposition \ref{prop : regularity} below, this is not possible. Thus $(\ref{eq : singsupp of U2})$ implies assertion $(\ref{eq : singsupp of U1})$.

We next prove that $(\ref{eq : singsupp of U2})$ holds. Let $D$ be the region bounded by $\Gamma_3$ and $\Gamma_4$, i.e.,
\[
	D:= \{(t,x) \mid 1-c_0 t < x < -c_1/c_0+c_1t,\ 1/c_0 < t < T\}.
\]
Indeed, $U$ is zero to the left of $\Gamma_1$ and to the right of $\Gamma_2\cup\Gamma_4$, and constant in the region enclosed by $\Gamma_1, \Gamma_2, \Gamma_3$. Therefore,
\[
	\singsupp_{\Ginf} U \subset  \subset \bigl(\cup_{k=1}^{4}\Gamma_k\bigr) \cup D.
\]
Thus it remains to show that $U$ is $\Ginf$-regular in $D$. Put $V= \partial_t U - C\partial_xU$ and $W= \partial_tU + C\partial_xU$. Then problem $(\ref{eq : generalized wave with DC in x})$ can be rewritten as the Cauchy problem for a first-order hyperbolic system
\begin{equation}\label{eq : system1}
	\begin{array}{rclcl}
		(\partial_t + C\partial_x)V &=& C^{\prime}(V-W)/2 &\ & \mbox{in}\ \G_{2,2}([0,T]\times\mathbb{R}),\\[2pt]
		(\partial_t - C\partial_x)W &=& C^{\prime}(V-W)/2 && \mbox{in}\ \G_{2,2}([0,T]\times\mathbb{R}),\\[2pt]
		V|_{t=0}\ =\ V_0 &=& U_1 && \mbox{in}\ \G_{2,2}(\mathbb{R}),\\[2pt]
      		W|_{t=0}\ =\ W_0&=& U_1 && \mbox{in}\ \G_{2,2}(\mathbb{R}).
	\end{array}
\end{equation}
We split the initial data $(U_1,U_1)$ into two parts,
\[
	(U_1,U_1) = (0,U_1) + (U_1,0),
\]
and let $(V_1,W_1)$ and $(V_2,W_2)$ be the solutions to problem $(\ref{eq : system1})$ with initial data $(0,U_1)$ and $(U_1,0)$, respectively. Obviously $(V_1,W_1) = (0,0)$ in $D$. Hence, if we show that $(V_2,W_2) = 0$ in $D$, the solution $U$ will be $\Ginf$-regular in $D$. To do this, we first consider the case $c_0 > c_1$. The solution $(V_2,W_2)$ has a representative $(v_{2,\varepsilon},w_{2,\varepsilon})_{\varepsilon \in (0,1]}$, which satisfies the Cauchy problem
\[
	\begin{array}{rclcl}
		(\partial_t + c_{\varepsilon}(x)\partial_x)v_{2,\varepsilon} &=& c_{\varepsilon}^{\prime}(x)(v_{2,\varepsilon}-w_{2,\varepsilon})/2, &\ &0 < t < T,\ x \in \mathbb{R},\\[2pt]
		(\partial_t - c_{\varepsilon}(x)\partial_x)w_{2,\varepsilon} &=& c_{\varepsilon}^{\prime}(x)(v_{2,\varepsilon}-w_{2,\varepsilon})/2, &  &0 < t < T,\ x \in \mathbb{R},\\[2pt]			 
		v_{2,\varepsilon}|_{t = 0}	& = & \varphi_{\varepsilon}(x+1),&  & x \in \mathbb{R},\\[2pt]
		w_{2,\varepsilon}|_{t = 0} & =& 0, &  & x \in \mathbb{R}.
	\end{array}
\]
Consider the characteristic curves $\gamma_{\varepsilon}^{+}(t,x,\tau)$ and $\gamma_{\varepsilon}^{-}(t,x,\tau)$ passing through $(t,x)$ at time $\tau = t$, which are the solutions of the equations
\begin{align*}
	& \partial_{\tau}\gamma_{\varepsilon}^{+}(t,x,\tau) = c_{\varepsilon}(\gamma_{\varepsilon}^{+}(t,x,\tau)), \qquad \gamma_{\varepsilon}^{+}(t,x,t) = x, \\
	& \partial_{\tau}\gamma_{\varepsilon}^{-}(t,x,\tau) = -c_{\varepsilon}(\gamma_{\varepsilon}^{-}(t,x,\tau)), \qquad \gamma_{\varepsilon}^{-}(t,x,t) = x.
\end{align*}
We can check that
\[
	\begin{split}
		\supp v_{2,\varepsilon}
		& \subset \bigl\{(t,x) \mid 0 \le t < T, \\
		& \qquad\ \min\{-\varepsilon,\gamma_{\varepsilon}^{+}(0,-1-\varepsilon,t)\} \le x \le \gamma_{\varepsilon}^{+}(0,-1+\varepsilon,t)\bigr\}, \\
		\supp w_{2,\varepsilon}
		& \subset \bigl\{(t,x) \mid 0 \le t < T, \\
		& \qquad\ \gamma_{\varepsilon}^{-}((1-2\varepsilon)/c_0,-\varepsilon,t) \le x \le \min\{\varepsilon, \gamma_{\varepsilon}^{+}(0,-1+\varepsilon,t)\}\bigr\}.
	\end{split}
\]
We now put $b_{\varepsilon} = 1/c_{\varepsilon}$ and $(\widetilde{v}_{\varepsilon},\widetilde{w}_{\varepsilon}) = (-b_{\varepsilon}v_{2,\varepsilon}, b_{\varepsilon}w_{2,\varepsilon})$, and change the roles of $t$ and $x$.
Then $(\widetilde{v}_{\varepsilon},\widetilde{w}_{\varepsilon})$ satisfies the systems
\begin{equation}\label{eq : representative2}
	\begin{array}{rclcl}
		(\partial_x + b_{\varepsilon}(x)\partial_t)\widetilde{v}_{\varepsilon} &=& n_{\varepsilon}(x)(\widetilde{v}_{\varepsilon}-\widetilde{w}_{\varepsilon}), &\ &0 < t < T,\ x \in \mathbb{R},\\[2pt]
		(\partial_x-b_{\varepsilon}(x)\partial_t) \widetilde{w}_{\varepsilon} &=& n_{\varepsilon}(x)(\widetilde{w}_{\varepsilon}-\widetilde{v}_{\varepsilon}), &  &0 < t < T,\ x \in \mathbb{R},\\[2pt]			
		\widetilde{v}_{\varepsilon}|_{x = -\varepsilon} & = & 0, &  & 1/c_0 \le t < T,\\[2pt]
		\widetilde{w}_{\varepsilon}|_{x = \varepsilon} & =& 0, &  & t_{\varepsilon} \le t < T,
	\end{array}
\end{equation}
where
\begin{equation}\label{eq:beps}
n_{\varepsilon}(y) = \frac{b_{\varepsilon}^{\prime}(y)}{2b_{\varepsilon}(y)} = \frac{d}{d y}\log b_\eps(y)
\end{equation}
and $t_{\varepsilon}$ is such that $\gamma_{\varepsilon}^{+}(0,-1+\varepsilon,t_{\varepsilon}) = \varepsilon$. The geometry can be read off from Figure\;\ref{fig:Geometry}.
\begin{figure}[htbp]
\begin{center}
\includegraphics{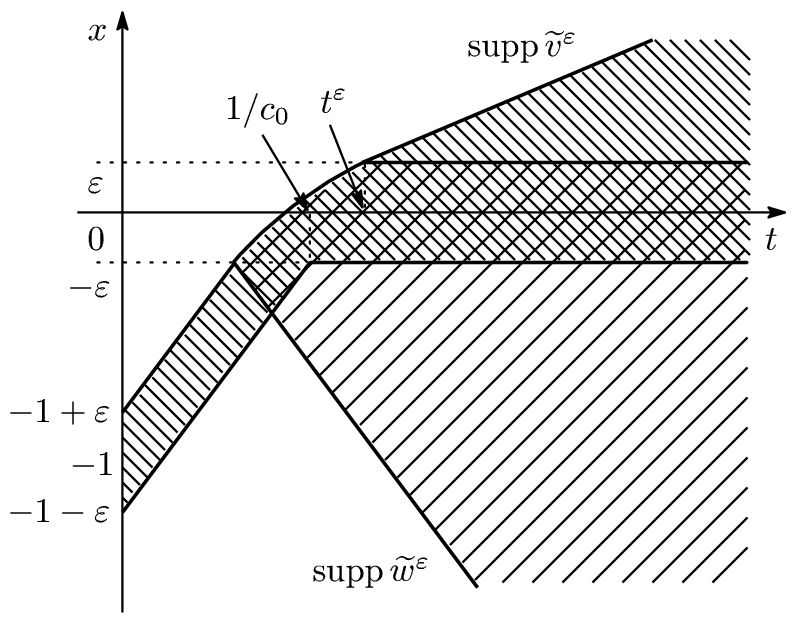}
\end{center}
\caption{Geometry of the supports of the solution to (\ref{eq : representative2}). Note the interchange of the roles of $x$ and $t$.}
\label{fig:Geometry}
\end{figure}

The characteristic curves $\eta_{\varepsilon}^{+}(t,x,\xi)$ and $\eta_{\varepsilon}^{-}(t,x,\xi)$ passing through $(t,x)$ at $\xi = x$ are the solutions of the equations
\begin{align*}
	& \partial_{\xi}\eta_{\varepsilon}^{+}(t,x,\xi) = b_{\varepsilon}(\xi), \qquad \eta_{\varepsilon}^{+}(t,x,x) = t, \\
	& \partial_{\xi}\eta_{\varepsilon}^{-}(t,x,\xi) = -b_{\varepsilon}(\xi), \qquad \eta_{\varepsilon}^{-}(t,x,x) = t.
\end{align*}
Along these characteristic curves, $\widetilde{v}_{\varepsilon}$ and $\widetilde{w}_{\varepsilon}$ are respectively calculated as
\begin{equation}\label{eq : integral eq1}
	\begin{split}
	\widetilde{v}_{\varepsilon}(t,x)
	& = - \int_{-\varepsilon}^{x}n_{\varepsilon}(\xi)\widetilde{w}_{\varepsilon}(\eta_{\varepsilon}^{+}(t,x,\xi),\xi) \exp\left(-\int_{-\varepsilon}^{\xi}n_{\varepsilon}(y)\,dy\right)d\xi \\
	& \qquad \cdot \exp\left(\int_{-\varepsilon}^{x}n_{\varepsilon}(y)\,dy\right),\\
	\widetilde{w}_{\varepsilon}(t,x)
	& = \widetilde{w}_{\varepsilon}(\eta_{\varepsilon}^{-}(t,x,-\varepsilon),-\varepsilon) \exp\left(\int_{-\varepsilon}^{x}n_{\varepsilon}(y)\,dy\right) \\
	& \quad - \int_{-\varepsilon}^{x}n_{\varepsilon}(\xi)\widetilde{v}_{\varepsilon}(\eta_{\varepsilon}^{-}(t,x,\xi),\xi) \exp\left(-\int_{-\varepsilon}^{\xi}n_{\varepsilon}(y)\,dy\right)d\xi \\
	& \qquad \cdot \exp\left(\int_{-\varepsilon}^{x}n_{\varepsilon}(y)\,dy\right)
	\end{split}
\end{equation}
for $0 \le t < T$ and $x \in \mathbb{R}$ such that $\eta^{+}_{\varepsilon}(t,x,-\varepsilon) \ge 1/c_0$. Let $t_{1,\varepsilon} = \eta^+_{\varepsilon}(1/c_0,-\varepsilon,\varepsilon)$ and $t_{2,\varepsilon} = \eta^-_{\varepsilon}(t_{1,\varepsilon},\varepsilon,-\varepsilon)$. Then $t_{2,\varepsilon} > t_{1,\varepsilon} > 1/c_0$ and $t_{2,\varepsilon} - 1/c_0 = O(\varepsilon)$ as $\varepsilon \downarrow 0$ (see Figure\;\ref{fig:Geometry}). For $k \in \mathbb{N}$, put
\begin{align*}
	I_k & = \left\{(t,-\varepsilon) \mid 1/c_0 + (k-1)(t_{2,\varepsilon} - 1/c_0) \le t \le t_{2,\varepsilon} + (k-1)(t_{2,\varepsilon} - 1/c_0)\right\}, \\
	J_k & = \left\{(t,\varepsilon) \mid t_{1,\varepsilon} + (k-1)(t_{2,\varepsilon} - 1/c_0) \le t \le t_{1,\varepsilon} + k(t_{2,\varepsilon} - 1/c_0)\right\}.
\end{align*}
Define
\begin{align*}
	(\widetilde{w}_{\varepsilon})_{+}(t,-\varepsilon) & = \max\left\{0, \widetilde{w}_{\varepsilon}(t,-\varepsilon)\right\},\\
	(\widetilde{w}_{\varepsilon})_{-}(t,-\varepsilon) & = -\min\left\{0, \widetilde{w}_{\varepsilon}(t,-\varepsilon)\right\}.
\end{align*}
Let $(\widetilde{v}_{1,\varepsilon}, \widetilde{w}_{1,\varepsilon})$ and $(\widetilde{v}_{2,\varepsilon}, \widetilde{w}_{2,\varepsilon})$ be the continuous solutions to the integral equations $(\ref{eq : integral eq1})$ with initial data $(0,(\widetilde{w}_{\varepsilon})_+(t,-\varepsilon))$ and $(0,(\widetilde{w}_{\varepsilon})_-(t,-\varepsilon))$, respectively. These solutions are obtained by iteration, see \cite{O:1992}. From this, we find that $\widetilde{v}_{1,\varepsilon}(t,x), \widetilde{v}_{2,\varepsilon}(t,x) \le 0$ and $\widetilde{w}_{1,\varepsilon}(t,x), \widetilde{w}_{2,\varepsilon}(t,x) \ge 0$ for $0 \le t < T$ and $x \in \mathbb{R}$ such that $\eta^{+}_{\varepsilon}(t,x,-\varepsilon) \ge 1/c_0$. Using this fact, $(\ref{eq : integral eq1})$ with $x = \varepsilon$ and (\ref{eq:beps}), we obtain, for $t \ge t_{1,\varepsilon}$,
\begin{equation}\label{eq : 1}
	\begin{split}
	\widetilde{w}_{1,\varepsilon}(t,\varepsilon)
	& \ge \sqrt{\dfrac{c_0}{c_1}} (\widetilde{w}_{\varepsilon})_+(\eta^-_{\varepsilon}(t,\varepsilon,-\varepsilon),-\varepsilon), \\
	\widetilde{w}_{2,\varepsilon}(t,\varepsilon)
	& \ge \sqrt{\dfrac{c_0}{c_1}} (\widetilde{w}_{\varepsilon})_-(\eta^-_{\varepsilon}(t,\varepsilon,-\varepsilon),-\varepsilon).
	\end{split}
\end{equation}
Note that $\widetilde{w}_{\varepsilon} = \widetilde{w}_{1,\varepsilon} - \widetilde{w}_{2,\varepsilon}$ and further that $\widetilde{w}_{\varepsilon}(t,\varepsilon) = 0$ for $t \ge t_{1,\varepsilon}$. Hence $\widetilde{w}_{1,\varepsilon}(t,\varepsilon) = \widetilde{w}_{2,\varepsilon}(t,\varepsilon)$ for $t \ge t_{1,\varepsilon}$. This and $(\ref{eq : 1})$ yield that, for $t \ge t_{1,\varepsilon}$,
\[
	\widetilde{w}_{1,\varepsilon}(t,\varepsilon) = \widetilde{w}_{2,\varepsilon}(t,\varepsilon)
	\ge \sqrt{\dfrac{c_0}{c_1}} \left|\widetilde{w}_{\varepsilon}(\eta^-_{\varepsilon}(t,\varepsilon,-\varepsilon),-\varepsilon) \right|,
\]
so that
\begin{equation}\label{eq : 2}
	\|\widetilde{w}_{1,\varepsilon} + \widetilde{w}_{2,\varepsilon}\|_{L^{\infty}(J_1)}
	\ge 2\sqrt{\dfrac{c_0}{c_1}} \|\widetilde{w}_{\varepsilon}\|_{L^{\infty}(I_2)}.
\end{equation}
On the other hand, the solution $(\widetilde{v}_{3,\varepsilon},\widetilde{w}_{3,\varepsilon})$ to $(\ref{eq : representative2})$ with initial data $(0,\|\widetilde{w}_{\varepsilon}\|_{L^{\infty}(I_1 \cup I_2)})$ at $x = -\varepsilon$ is given by
\[
	\widetilde{v}_{3,\varepsilon} = \dfrac{1}{2}\left(1-\dfrac{c_0}{c_{\varepsilon}(x)}\right) \|\widetilde{w}_{\varepsilon}\|_{L^{\infty}(I_1 \cup I_2)},
	\quad \widetilde{w}_{3,\varepsilon} = \dfrac{1}{2}\left(1+\dfrac{c_0}{c_{\varepsilon}(x)}\right) \|\widetilde{w}_{\varepsilon}\|_{L^{\infty}(I_1 \cup I_2)},
\]
and $\widetilde{w}_{1,\varepsilon} + \widetilde{w}_{2,\varepsilon}$ is smaller than $\widetilde{w}_{3,\varepsilon}$ on $J_1$, i.e.,
\begin{equation}\label{eq : 3}
	\|\widetilde{w}_{1,\varepsilon} + \widetilde{w}_{2,\varepsilon}\|_{L^{\infty}(J_1)}
	\le \dfrac{1}{2}\left(1+\dfrac{c_0}{c_1}\right) \|\widetilde{w}_{\varepsilon}\|_{L^{\infty}(I_1 \cup I_2)}.
\end{equation}
Combining $(\ref{eq : 2})$ and $(\ref{eq : 3})$, we get
\begin{equation}\label{eq : 4}
	\|\widetilde{w}_{\varepsilon}\|_{L^{\infty}(I_2)}
	\le \dfrac{1}{4}\left(\sqrt{\dfrac{c_1}{c_0}}+\sqrt{\dfrac{c_0}{c_1}}\right) \|\widetilde{w}_{\varepsilon}\|_{L^{\infty}(I_1 \cup I_2)}.
\end{equation}
Since by assumption,
\[
	\dfrac{1}{4}\left(\sqrt{\dfrac{c_1}{c_0}}+\sqrt{\dfrac{c_0}{c_1}}\right) < 1,
\]
it must hold from $(\ref{eq : 4})$ that
\[
	\|\widetilde{w}_{\varepsilon}\|_{L^{\infty}(I_1 \cup I_2)} = \|\widetilde{w}_{\varepsilon}\|_{L^{\infty}(I_1)}.
\]
Together with this, $(\ref{eq : 4})$ leads to
\[
	\|\widetilde{w}_{\varepsilon}\|_{L^{\infty}(I_2)} \le \dfrac{1}{4}\left(\sqrt{\dfrac{c_1}{c_0}}+\sqrt{\dfrac{c_0}{c_1}}\right) \|\widetilde{w}_{\varepsilon}\|_{L^{\infty}(I_1)} .
\]
Repeat this process to get for $k \in \mathbb{N}$,
\[
	\|\widetilde{w}_{\varepsilon}\|_{L^{\infty}(I_k)} \le \left[\dfrac{1}{4}\left(\sqrt{\dfrac{c_1}{c_0}}+\sqrt{\dfrac{c_0}{c_1}}\right)\right]^{k-1}\|\widetilde{w}_{\varepsilon}\|_{L^{\infty}(I_1)} .
\]
We now fix $t_0 > 1/c_0$ arbitrarily and choose $k_{\varepsilon} \in \mathbb{N}$ so that $t_0 \in I_{k_{\varepsilon}}$. Then
\[
	\|\widetilde{w}_{\varepsilon}\|_{L^{\infty}(I_{k_{\varepsilon}})} \le \left[\dfrac{1}{4}\left(\sqrt{\dfrac{c_1}{c_0}}+\sqrt{\dfrac{c_0}{c_1}}\right)\right]^{(t_0 - 1/c_0)/(t_{2,\varepsilon} - 1/c_0)-1}\|\widetilde{w}_{\varepsilon}\|_{L^{\infty}(I_1)}.
\]
Recall that $t_{2,\varepsilon} - 1/c_0 = O(\varepsilon)$ as $\varepsilon \downarrow 0$, and note that $\|\widetilde{w}_{\varepsilon}\|_{L^{\infty}(I_1)} = O(\varepsilon^{-p})$ for some $p \ge 0$. Hence, for any $N \ge 0$,
\[
	\|\widetilde{w}_{\varepsilon}\|_{L^{\infty}(I_{k_{\varepsilon}})} = O(\varepsilon^N) \quad {\rm as}\ \varepsilon \downarrow 0.
\]
Since $t_0 \in I_{k_{\varepsilon}}$ and $\|\widetilde{w}_{\varepsilon}\|_{L^{\infty}(I_k)}$ is decreasing in $k \in \mathbb{N}$, we get
\[
	\|\widetilde{w}_{\varepsilon}\|_{L^{\infty}(\{(t,-\varepsilon) \mid t \ge t_0\})} = O(\varepsilon^N) \quad {\rm as}\ \varepsilon \downarrow 0.
\]
In the region $\{(t,x) \mid t \ge \eta^-_{\varepsilon}(t_0,-\varepsilon,x),\ x \le -\varepsilon\}$ the right-hand side of the second equation in (\ref{eq : representative2}) vanishes, so $\widetilde{w}_{\varepsilon}$ is constant along any characteristic curve there, and we get
\begin{equation}\label{eq : 5}
	\|\widetilde{w}_{\varepsilon}\|_{L^{\infty}(\{(t,x) \mid t \ge \eta^-_{\varepsilon}(t_0,-\varepsilon,x),\ x \le -\varepsilon\})} = O(\varepsilon^N) \quad {\rm as}\ \varepsilon \downarrow 0.
\end{equation}
Furthermore, using the fact that $\widetilde{v}_{\varepsilon} = 0$ on $\{(t,-\varepsilon) \mid t \ge t_0\}$, we find that, for any $N \ge 0$,
\begin{equation}\label{eq : 6}
	\begin{split}
	\|\widetilde{v}_{\varepsilon}\|_{L^{\infty}(\{(t,x) \mid t \ge \eta^+_{\varepsilon}(t_0,-\varepsilon,x),\ x \ge -\varepsilon\})} = O(\varepsilon^N) \quad {\rm as}\ \varepsilon \downarrow 0, \\
	\|\widetilde{w}_{\varepsilon}\|_{L^{\infty}(\{(t,x) \mid t \ge \eta^+_{\varepsilon}(t_0,-\varepsilon,x),\ -\varepsilon \le x \le \varepsilon\})} = O(\varepsilon^{N}) \quad {\rm as}\ \varepsilon \downarrow 0,
	\end{split}
\end{equation}
invoking the uniqueness argument as in the proof of \cite[Theorem 3.1]{DHO:2013}.
Estimates $(\ref{eq : 5})$ and $(\ref{eq : 6})$ are rewritten with $(\widetilde{v}_{\varepsilon},\widetilde{w}_{\varepsilon}) = (-b_{\varepsilon}v_{2,\varepsilon}, b_{\varepsilon}w_{2,\varepsilon})$ as follows: for any $N \ge 0$,
\[
	\begin{split}
		\|w_{2,\varepsilon}\|_{L^{\infty}(\{(t,x) \mid t \ge t_0,\ \gamma^-_{\varepsilon}(t_0,-\varepsilon,t) \le x \le -\varepsilon\})} = O(\varepsilon^{N}) \quad {\rm as}\ \varepsilon \downarrow 0,\\
	\|v_{2,\varepsilon}\|_{L^{\infty}(\{(t,x) \mid t \ge t_0,\ -\varepsilon \le x \le \gamma^+_{\varepsilon}(t_0,-\varepsilon,t)\})} = O(\varepsilon^{N}) \quad {\rm as}\ \varepsilon \downarrow 0, \\
	\|w_{2,\varepsilon}\|_{L^{\infty}(\{(t,x) \mid t \ge t_0,\ -\varepsilon \le x \le \min\{\varepsilon,\gamma^+_{\varepsilon}(t_0,-\varepsilon,t)\}\})} = O(\varepsilon^{N}) \quad {\rm as}\ \varepsilon \downarrow 0.
	\end{split}
\]
This and Lemma 1.2.3 in \cite{GKOS:2001} show that $(V_2,W_2) = 0$ in the region
\[
	\{(t,x) \mid 1-c_0 t < x < -c_1/c_0+c_1t,\ t_0 < t < T\}.
\]
Since $t_0 > 1/c_0$ is arbitrary, we conclude that $(V_2,W_2) = 0$ in $D$.

The case $c_0 < c_1$ can be treated by the same arguments, using the fact that $(-v_{2,\varepsilon},w_{2,\varepsilon})$ satisfies the Cauchy problem
\begin{equation}\label{eqn : representative3}
	\begin{array}{rclcl}
		(\partial_x + b_{\varepsilon}(x)\partial_t)(-v_{2,\varepsilon}) &=& -n_{\varepsilon}(x)((-v_{2,\varepsilon})+w_{2,\varepsilon}), &\ &0 < t < T,\ x \in \mathbb{R},\\[2pt]
		(\partial_x-b_{\varepsilon}(x)\partial_t) w_{2,\varepsilon} &=& -n_{\varepsilon}(x)((-v_{2,\varepsilon})+w_{2,\varepsilon}), &  &0 < t < T,\ x \in \mathbb{R},\\[2pt]			 
		v_{2,\varepsilon}|_{x = -\varepsilon} & = & 0, &  & 1/c_0 \le t < T,\\[2pt]
		w_{2,\varepsilon}|_{x = \varepsilon} & =& 0, &  & t_{\varepsilon} \le t < T,
	\end{array}
\end{equation}
and considering the corresponding integral equations for $(-v_{2,\varepsilon},w_{2,\varepsilon})$. The proof of Theorem \ref{thm : propagation1} is now complete.
\end{proof}

As mentioned in Example 1
, one may regularize the piecewise constant propagation speed $c(x) = c_0 + (c_1 - c_0)H(x)$ in such a way that the corresponding element $C\in\cG_{\infty,2}(\R)$ is $\Ginf$-regular on $\R$. Indeed, it suffices to use a mollifier $\varphi_{h(\eps)}$, where $h(\eps) > 0$ and $(1/h(\eps))_{\varepsilon \in (0,1]}$ is a slow scale net. It has been shown in \cite{HOP:06} that an element $C\in \G(\R)$ of bounded type belongs to $\Ginf(\R)$ if and only if all derivatives are of slow scale type.

Consider problem $(\ref{eq : generalized wave with DC in x})$ with $U_0 \equiv 0$ and $U_1$ given by the class of $(\varphi_{\varepsilon}(x+1))_{\varepsilon \in (0,1]}$ as in Theorem $\ref{thm : propagation1}$. We shall show in the following theorem that the reflected ray
\[
	\Gamma_3 = \left\{(t,x) \mid x = 1 - c_0t,\ 1/c_0 \le t < T\right\}
\]
does not belong to $\singsupp_{\Ginf} U$, if $C$ is $\Ginf$-regular on $\R$ (see Figure \ref{fig : singsupp of U2}).
In this situation, the $\Ginf$-singularities in the solution arise solely from the flow-out of the singularities in the initial data. There are no reflected singularities, just as in the classical case of $\Cinf$-regular coefficients. This phenomenon has already been observed for scalar equations in the $\G$-setting in \cite{GO:2011a} and shows again that a non-logarithmic regularization is essential for capturing reflected singularities.

\begin{figure}[htbp]
\begin{center}
\includegraphics{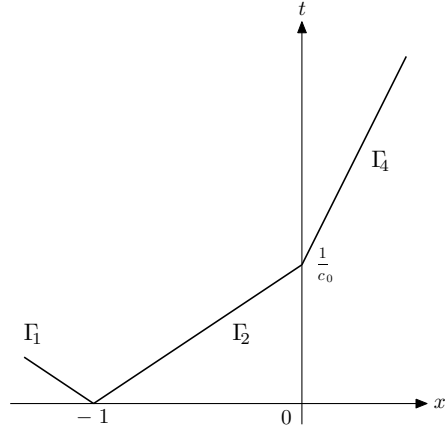}
\end{center}
\caption{The $\Ginf$-singular support of the solution $U$ with slow scale coefficient.}
\label{fig : singsupp of U2}
\end{figure}

\begin{theorem}\label{thm : propagation2}
Let $c_0$, $c_1 > 0$ be such that $c_0 \ne c_1$ and let $T > 1/c_0$. Furthermore, let $C$, $U_0$ and $U_1$ be as described above and let $U \in \G_{2,2}([0,T] \times \mathbb{R})$ be the solution to problem $(\ref{eq : generalized wave with DC in x})$. If $C$ is $\Ginf$-regular on $\R$, then
\begin{equation}\label{eq : singsupp of U3}
	\begin{split}
		\singsupp_{\Ginf} U
		& = \{(t,x) \mid x=-1-c_0 t,\ 0 \le t < T\} \\
		&\quad \cup \{(t,x) \mid x=-1+c_0 t,\ 0 \le t \le 1/c_0\} \\
		&\quad \cup \{(t,x) \mid x=-c_1/c_0+c_1t,\ 1/c_0 \le t < T\} \\
		& =:  \Gamma_1 \cup \Gamma_2 \cup \Gamma_4.
	\end{split}
\end{equation}
\end{theorem}

\begin{proof}
Assertion $(\ref{eq : singsupp of U3})$ will hold if we show that
\begin{equation}\label{eq : singsupp of U4}
	\singsupp_{\Ginf} U \subset \Gamma_1 \cup \Gamma_2 \cup \Gamma_4.
\end{equation}
This can be seen as follows. For the sake of simplicity, we assume that $T > 2/c_0$. As in the proof of Theorem \ref{thm : propagation1}, we interchange the roles of $t$ and $x$ and consider $(U|_{x=-1},\partial_x U|_{x=-1})$ as initial data. From $(\ref{eq : singsupp of U4})$ and d'Alembert's formula, we see that $\singsupp_{\Ginf} U|_{x=-1} = \{0\}$ and $\partial_x U|_{x=-1} = 0$ near $t = 0$. The singularity of $U|_{x=-1}$ at $t = 0$ propagates along $\Gamma_1$ as $x$ goes from $-1$ to $-\infty$. As $x$ goes from $-1$ to $\infty$, it propagates along $\Gamma_2$ and $\Gamma_4$. No splitting of the singularity occurs at $x=0$ unlike in Theorem \ref{thm : propagation1}. This can be proven similarly to the proofs of Theorems 4.1 and 5.1 of \cite{DHO:2013}. Thus $(\ref{eq : singsupp of U4})$ implies assertion $(\ref{eq : singsupp of U3})$.

To show that $(\ref{eq : singsupp of U4})$ holds, we consider system $(\ref{eq : system1})$. Put $\Gamma_6 = \{(t,0) \mid 1/c_0 \le t < T\}$. Let $D_1$ be the region bounded by $\Gamma_3$ and $\Gamma_6$, and $D_2$ be the region bounded by $\Gamma_4$ and $\Gamma_6$. It is immediate to check that
\[
	\begin{split}
		\supp V \subset \Gamma_2 \cup \Gamma_4 \cup \Gamma_6 \cup D_2, \\
		\supp W \subset \Gamma_1 \cup \Gamma_3 \cup \Gamma_6 \cup D_1.
	\end{split}
\]
Put $B = 1/C$ and $(\widetilde{V},\widetilde{W}) = (-BV, BW)$. Clearly
\begin{equation}\label{eq : supp of V and W}
	\begin{split}
		\supp \widetilde{V} \subset \Gamma_2 \cup \Gamma_4 \cup \Gamma_6 \cup D_2, \\
		\supp \widetilde{W} \subset \Gamma_1 \cup \Gamma_3 \cup \Gamma_6 \cup D_1.
	\end{split}
\end{equation}
Furthermore, $(\widetilde{V},\widetilde{W})$ satisfies the system
\[
	\begin{split}
		(\partial_x + B\partial_t)\widetilde{V} & = N(\widetilde{V}-\widetilde{W}),\\
		(\partial_x-B\partial_t) \widetilde{W} & = N(\widetilde{W}-\widetilde{V}),
	\end{split}
\]
where $N = B^{\prime}/(2B)$. Take the value of $(\widetilde{V},\widetilde{W})$ along some line $\{x=x_0\}$ with $x_0 > 0$ as initial data and consider the system above. Using the commutator argument as in the proof of Theorem 5.1 of \cite{DHO:2013}, the derivatives $(\partial_x+B\partial_t)^k \widetilde{W}$ can be estimated in terms of the $L^\infty$-norm of the zeroth derivative of $\widetilde{V}$ for whatever power $k$. The derivatives $(\partial_x-B\partial_t)^k \widetilde{W}$ are estimated on $D_1$ using the fact that $\widetilde{V}$ vanishes there. We conclude that $\widetilde{W}$ is $\Ginf$-regular on $\Gamma_3 \cup D_1$. With this in mind, consider now the value of $(\widetilde{V},\widetilde{W})$ along $\{x=x_1\}$ with $x_1 < 0$ as initial data.  Then Proposition \ref{prop : regularity2} below shows that $\widetilde{W}$ is $\Ginf$-regular along $\Gamma_6$ and $\widetilde{V}$ is $\Ginf$-regular on $\Gamma_6 \cup D_2$. Hence from $(\ref{eq : supp of V and W})$,
\begin{gather*}
	\singsupp_{\Ginf} \widetilde{V} \subset \Gamma_2 \cup \Gamma_4, \\
	\singsupp_{\Ginf} \widetilde{W} \subset \Gamma_1.
\end{gather*}

This implies $(\ref{eq : singsupp of U4})$. The proof of Theorem \ref{thm : propagation2} is now complete.
\end{proof}

\subsection{The time-dependent case in one dimension}

This subsection is devoted to problem $(\ref{eq : wave with DC in t})$ in one dimension:
\begin{equation}\label{eq:1Dwavetime}
	\begin{array}{l}
		\partial_t^2u - c(t)^2\partial_x^2u = 0,  \quad t \geq 0, \ x \in \mathbb{R}, \vspace{4pt}\\
		u|_{t = 0} = u_{0},\quad \partial_tu|_{t = 0} = u_{1}, \quad x \in \mathbb{R}.
	\end{array}
\end{equation}
The case of delta functions as initial data has been settled in \cite{DHO:2013}. The goal of this subsection is to extend the results of \cite{DHO:2013} to arbitrary $\Ginf$-point singularities in the initial data.
Assume that the coefficient $c(t)$ is given by
\[
	c(t) = c_0 + (c_1 - c_0)H(t-1)
\]
where $c_0$, $c_1 > 0$, $c_0 \ne c_1$ and $H$ is the Heaviside function. As in the previous subsection, we define a generalized function $C \in \G(\mathbb{R})$ by means of a representative $(c_{\varepsilon})_{\varepsilon \in (0,1]} = (c \ast \varphi_{\varepsilon})_{\varepsilon \in (0,1]}$ with a mollifier $\varphi_{\varepsilon}$ as in Example~1
. Then $\singsupp_{\Ginf} C = \{t=1\}$. Thus we will consider the Cauchy problem
\begin{equation}\label{eq : generalized wave with DC in t}
	\begin{array}{lr}
		\partial_t^2U - C^2\partial_x^2U = 0 & \mbox{in}\ \G([0,\infty)\times\mathbb{R}), \vspace{4pt}\\
		U|_{t = 0} = U_0,\quad \partial_tU|_{t = 0} = U_1 & \mbox{in}\ \G(\mathbb{R}).
	\end{array}
\end{equation}
By \cite[Theorem 3.1]{DHO:2013}, problem $(\ref{eq : generalized wave with DC in t})$ has a unique solution $U \in \G([0,\infty) \times \mathbb{R})$ for any initial data $U_0$, $U_1 \in \G (\mathbb{R})$.

An important ingredient in the arguments is the persistence of regularity from regular initial data. Thus we are first going to show that, if $U_0$, $U_1 \in \Ginf (\mathbb{R})$, then $U$ is $\Ginf$-regular off the $\Ginf$-singular support of the coefficient $C$.

\begin{proposition}\label{prop : regularity}
Let $C$ be as described above and let $U \in \G([0,\infty) \times \mathbb{R})$ be the solution to problem $(\ref{eq : generalized wave with DC in t})$ with initial data $U_0$, $U_1 \in \G(\mathbb{R})$. If $U_0$, $U_1 \in \Ginf(\mathbb{R})$, then
\begin{equation}\label{eq : regularity of U}
	\singsupp_{\Ginf} U \subset \left\{(t,x) \mid t = 1,\ x \in \mathbb{R}\right\}.
\end{equation}
\end{proposition}

\begin{proof}
The proof is similar to that of \cite[Theorem 3.1]{DHO:2013}, where an existence result for problem $(\ref{eq : generalized wave with DC in t})$ with a more general coefficient has been proven. Put $V= \partial_t U - C\partial_xU$ and $W= \partial_tU + C\partial_xU$. Then problem $(\ref{eq : generalized wave with DC in t})$ can be rewritten as the Cauchy problem for a first-order hyperbolic system
\begin{equation}\label{eq : generalized system}
	\begin{array}{rclcl}
		(\partial_t + C\partial_x)V &=& M(V-W) &\ & \mbox{in}\ \G([0,\infty)\times\mathbb{R}),\\[2pt]
		(\partial_t - C\partial_x)W &=& M(W-V) && \mbox{in}\ \G([0,\infty)\times\mathbb{R}),\\[2pt]
		V|_{t=0}\ =\ V_0 &=& U_1 - (C|_{t=0}) U_0^{\prime}&& \mbox{in}\ \G(\mathbb{R}),\\[2pt]
		W|_{t=0}\ =\ W_0 &=& U_1+ (C|_{t=0}) U_0^{\prime} && \mbox{in}\ \G(\mathbb{R}),
	\end{array}
\end{equation}
where $M = C^{\prime}/(2C) \in \G(\mathbb{R})$. The solution $(V,W)$ has a representative $(v_{\varepsilon},w_{\varepsilon})_{\varepsilon \in (0,1]}$ satisfying the Cauchy problem
\begin{equation}\label{eq : representative}
	\begin{array}{rclcl}
		(\partial_t + c_{\varepsilon}(t)\partial_x)v_{\varepsilon} &=& \mu_{\varepsilon}(t)(v_{\varepsilon}-w_{\varepsilon}), &\ &t > 0,\ x \in \mathbb{R},\\[2pt]
		(\partial_t - c_{\varepsilon}(t)\partial_x)w_{\varepsilon} &=& \mu_{\varepsilon}(t)(w_{\varepsilon}-v_{\varepsilon}), &  &t > 0,\ x \in \mathbb{R},\\[2pt]			
		v_{\varepsilon}|_{t = 0}\ =\ v_{0\varepsilon}& = & u_{1\varepsilon} - c_{\varepsilon}(0)u_{0\varepsilon}^{\prime},&  & x \in \mathbb{R},\\[2pt]
		w_{\varepsilon}|_{t = 0}\ =\ w_{0\varepsilon} & =& u_{1\varepsilon} + c_{\varepsilon}(0)u_{0\varepsilon}^{\prime}, &  & x \in \mathbb{R},
	\end{array}
\end{equation}
where $(u_{0\varepsilon})_{\varepsilon \in (0,1]}$, $(u_{1\varepsilon})_{\varepsilon \in (0,1]}$, $(c_{\varepsilon})_{\varepsilon \in (0,1]}$ and $(\mu_{\varepsilon})_{\varepsilon \in (0,1]}$ are representatives of $U_0$, $U_1$, $C$ and $M$, respectively, such that $\mu_{\varepsilon} = c_{\varepsilon}^{\prime}/(2c_{\varepsilon})$. Consider the characteristic curves $\gamma_{\varepsilon}^{+}(t,x,\tau)$ and $\gamma_{\varepsilon}^{-}(t,x,\tau)$ passing through $(t,x)$ at time $\tau = t$ which satisfy
\begin{align*}
	& \partial_{\tau}\gamma_{\varepsilon}^{+}(t,x,\tau) = c_{\varepsilon}(\tau), \qquad \gamma_{\varepsilon}^{+}(t,x,t) = x, \\
	& \partial_{\tau}\gamma_{\varepsilon}^{-}(t,x,\tau) = -c_{\varepsilon}(\tau), \qquad \gamma_{\varepsilon}^{-}(t,x,t) = x.
\end{align*}
Along these characteristic curves, $v_{\varepsilon}$ and $w_{\varepsilon}$ are respectively calculated as
\begin{equation}\label{eq : integral eq}
	\begin{split}
	v_{\varepsilon}(t,x) & = v_{0\varepsilon}(\gamma_{\varepsilon}^{+}(t,x,0)) + \int_0^t \mu_{\varepsilon}(s)(v_{\varepsilon}-w_{\varepsilon})(s,\gamma_{\varepsilon}^{+}(t,x,s))\,ds, \\
	w_{\varepsilon}(t,x) & = w_{0\varepsilon}(\gamma_{\varepsilon}^{-}(t,x,0)) + \int_0^t \mu_{\varepsilon}(s)(w_{\varepsilon}-v_{\varepsilon})(s,\gamma_{\varepsilon}^{-}(t,x,s))\,ds.
	\end{split}
\end{equation}
For each $T > 0$, we define $K_T$ as the trapezoidal region with corners $(0,-\xi)$, $(T,-\xi+c_2T)$, $(T,\xi-c_2T)$, $(0,\xi)$, where $c_2 = \max(c_0,c_1)$. From $(\ref{eq : integral eq})$, the following inequalities are easily deduced:
\begin{align*}
	\|v_{\varepsilon}\|_{L^{\infty}(K_T)} & \le \|v_{0\varepsilon}\|_{L^{\infty}(K_0)} + \int_0^T |\mu_{\varepsilon}(s)|(\|v_{\varepsilon}\|_{L^{\infty}(K_s)}+\|w_{\varepsilon}\|_{L^{\infty}(K_s)})\,ds, \\
	\|w_{\varepsilon}\|_{L^{\infty}(K_T)} & \le \|w_{0\varepsilon}\|_{L^{\infty}(K_0)} + \int_0^T |\mu_{\varepsilon}(s)|(\|w_{\varepsilon}\|_{L^{\infty}(K_s)}+\|v_{\varepsilon}\|_{L^{\infty}(K_s)})\,ds.
\end{align*}
We add these two inequalities and apply Gronwall's inequality to get
\begin{align*}
	& \|v_{\varepsilon}\|_{L^{\infty}(K_T)} + \|w_{\varepsilon}\|_{L^{\infty}(K_T)} \\
	& \qquad \le (\|v_{0\varepsilon}\|_{L^{\infty}(K_0)} + \|w_{0\varepsilon}\|_{L^{\infty}(K_0)}) \exp\left(2\int_0^T |\mu_{\varepsilon}(s)|\,ds\right).
\end{align*}
Differentiating the equations and using the same argument, we obtain for any $\alpha \in \mathbb{N}_0$,
\begin{equation}\label{eq : inequality}
	\begin{split}
	& \|\partial_x^{\alpha}v_{\varepsilon}\|_{L^{\infty}(K_T)} + \|\partial_x^{\alpha}w_{\varepsilon}\|_{L^{\infty}(K_T)} \\
	& \qquad \le \left(\left\|\partial_x^{\alpha}v_{0\varepsilon}\right\|_{L^{\infty}(K_0)} + \left\|\partial_x^{\alpha}w_{0\varepsilon}\right\|_{L^{\infty}(K_0)}\right) \exp\left(2\int_0^T |\mu_{\varepsilon}(s)|\,ds\right).
	\end{split}
\end{equation}
On the right-hand side, the terms involving $v_{0\varepsilon}$ and $w_{0\varepsilon}$ are of order $O(\varepsilon^{-N})$ for some $N \ge 0$ independent of $\alpha$. The exponential term is uniformly bounded in $\varepsilon$. Note that, on every compact set outside $\{(t,x) \mid t = 1,\ x \in \mathbb{R}\}$, $\p_t v_{\varepsilon}=-c\p_x v_{\varepsilon}$ and $\p_t w_{\varepsilon}=c \p_x w_{\varepsilon}$, when $\varepsilon$ is small enough. This and inequality $(\ref{eq : inequality})$ imply that $V$ and $W$ are $\Ginf$-regular outside $\{(t,x) \mid t = 1,\ x \in \mathbb{R}\}$. From the definitions of $V$ and $W$, assertion $(\ref{eq : regularity of U})$ follows.
\end{proof}

The central issue of this subsection is the propagation of the $\Ginf$-singular support from arbitrary initial data in $\cG(\R)$. As the prototypical case, we focus on a point singularity, say at the origin. We take $U_0$, $U_1 \in \Ginf (\mathbb{R}\setminus \{0\})$. Choose a cut-off function $\chi \in \Cinf_0(\mathbb{R})$ which is identically equal to one in a neighborhood of the origin. Split the initial data $(U_0,U_1)$ into two parts,
\[
 	(U_0,U_1) = (\chi U_0,\chi U_1) + ((1-\chi)U_0, (1-\chi)U_1).
\]
Then the $\Ginf$-singular support of the solution $\widetilde{U}$ to problem $(\ref{eq : generalized wave with DC in t})$ with initial data $(\chi U_0,\chi U_1)$ is contained in the set
\begin{equation}\label{eq:singsupp1}
\begin{aligned}
	&\left\{(t,x) \mid x = \pm\int_0^t c(s)\,ds,\ t \ge 0\right\} \\
	&\qquad\cup \left\{(t,x) \mid x = \pm\left(2\int_0^1 c(s)\,ds-\int_0^t c(s)\,ds\right),\ t \ge 1\right\}
\end{aligned}
\end{equation}
in the limit as the support of $\chi$ shrinks to $\{0\}$. This is most easily seen by introducing $\widetilde{V}= \partial_t \widetilde{U} - C\partial_x\widetilde{U}$, $\widetilde{W}= \partial_t\widetilde{U} + C\partial_x\widetilde{U}$, which satisfy a system of equations similar to (\ref{eq : generalized system}) with a corresponding cut-off in the initial data and hence are supported in a neighborhood of the lines defined by (\ref{eq:singsupp1}). Furthermore, by Proposition \ref{prop : regularity}, the solution to problem $(\ref{eq : generalized wave with DC in t})$ with initial data $((1-\chi)U_0,(1-\chi)U_1)$ is $\Ginf$-regular off $\{(t,x) \mid t=1,\ x \in \mathbb{R}\}$ for any $\chi$ as above, since $(1-\chi)U_0$ and $(1-\chi)U_1$ belong to $\Ginf(\mathbb{R})$ for any $\chi$ as above. We summarize the result in the following proposition (see also Figure~\ref{fig : upper bound}).
\begin{proposition}\label{prop : propagation_ub}
Let $C$, $U_0$ and $U_1$ be as described above and let $U \in \G([0,\infty) \times \mathbb{R})$ be the solution to problem $(\ref{eq : generalized wave with DC in t})$. Then the set
\begin{equation}\label{eq:singsupp2}
	\begin{array}{l}
	\displaystyle  \{(t,x) \mid t=1,\ x \in \mathbb{R}\} \cup \left\{(t,x) \mid x = \pm \int_0^t c(s)\,ds,\ t \ge 0\right\} \\[10pt]
	\qquad \displaystyle \cup \left\{(t,x) \mid x = \pm\left(2\int_0^1 c(s)\,ds-\int_0^t c(s)\,ds\right),\ t \ge 1\right\}
	\end{array}
\end{equation}
is an upper bound for the $\Ginf$-singular support of the solution $U$ to problem $(\ref{eq : generalized wave with DC in t})$.
\end{proposition}
In the case of delta function initial data, we have shown in \cite{DHO:2013} that the $\Ginf$-singular support of the generalized solution coincides with the set (\ref{eq:singsupp1}). The line $\{t=1\}$ generally may belong to the $\Ginf$-singular support of the solution, even for $\Cinf$-initial data, as shown by the example
\[
   u_\eps(t,x) = \frac{x^2}{2} + \int_0^t\int_0^s (c_\eps(r))^2\,drds,
\]
which solves the wave equation $\partial_t^2u_\eps - c_\eps(t)^2\partial_x^2u_\eps = 0$ with initial data $u_\eps(0,x) = x^2/2$, $\p_tu_\eps(0,x) = 0$. In general, the $\Ginf$-singular support of the generalized solution may be a proper subset of the set given by (\ref{eq:singsupp1}). This depends on the interplay between the initial data and the coefficient $c(t)$.

In the following theorem, we give a necessary and sufficient condition under which initial singularities propagate out and the transmitted rays
\begin{align*}
	\Gamma_{+}:&=\left\{(t,x) \mid x = \int_0^t c(s)\,ds,\ t \ge 0\right\},\\
	\Gamma_{-}:&=\left\{(t,x) \mid x = -\int_0^t c(s)\,ds,\ t \ge 0\right\}
\end{align*}
belong to $\singsupp_{\Ginf} U$.

\begin{figure}[htbp]
\begin{center}
\includegraphics{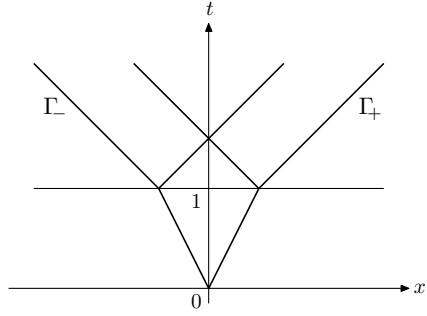}
\end{center}
\caption{An upper bound for the $\Ginf$-singular support of the solution $U$.}
\label{fig : upper bound}
\end{figure}

\begin{theorem}\label{thm : propagation3}
Let $C$, $U_0$ and $U_1$ be as described above and let $U \in \G([0,\infty) \times \mathbb{R})$ be the solution to problem $(\ref{eq : generalized wave with DC in t})$. Then:
\begin{itemize}
\item[{\rm (a)}] if $\singsupp_{\Ginf} (U_1-(C|_{t=0})U_0^{\prime}) =\{0\}$, then
\[
	\Gamma_{+}=\left\{(t,x) \mid x = \int_0^t c(s)\,ds,\ t \ge 0\right\} \subset \singsupp_{\Ginf} U;
\]
\item[{\rm (b)}] if $\singsupp_{\Ginf} (U_1-(C|_{t=0})U_0^{\prime}) =\emptyset$, then $U$ is $\Ginf$-regular along $\Gamma_{+} \setminus \{(1,c_0)\}$;
\item[{\rm (c)}] if $\singsupp_{\Ginf} (U_1+(C|_{t=0})U_0^{\prime}) =\{0\}$, then
\[
	\Gamma_{-}=\left\{(t,x) \mid x = -\int_0^t c(s)\,ds,\ t \ge 0\right\} \subset \singsupp_{\Ginf} U;
\]
\item[{\rm (d)}] if $\singsupp_{\Ginf} (U_1+(C|_{t=0})U_0^{\prime}) =\emptyset$, then $U$ is $\Ginf$-regular along $\Gamma_{-}\setminus \{(1,-c_0)\}$.
\end{itemize}
\end{theorem}

\begin{proof}
To show that (a) and (b) hold, we consider the solution $(V,W)$ to problem $(\ref{eq : generalized system})$ with $(V_0,W_0) = (\chi (U_1 - (C|_{t=0}) U_0^{\prime}),0)$. As may be seen from the proof of Proposition \ref{prop : regularity}, if $\singsupp_{\Ginf} (U_1-(C|_{t=0})U_0^{\prime}) =\emptyset$, then $V$ is $\Ginf$-regular outside $\{(t,x) \mid t=1,\ x \in \mathbb{R}\}$ for any $\chi$. From this, case (b) follows. As for (a), it suffices to show that if $\singsupp_{\Ginf} (U_1-(C|_{t=0})U_0^{\prime}) =\{0\}$, then
\begin{equation}\label{eq : singsupp of V}
	\singsupp_{\Ginf} V = \Gamma_+
\end{equation}
holds in the limit as the support of $\chi$ shrinks to $\{0\}$. It is easy to see that $\supp V \subset \Gamma_+$ holds as the support of $\chi$ shrinks to $\{0\}$. Hence, if we show that for any $\chi$, $V$ is not $\Ginf$-regular on $\mathbb{R}$ for any $t > 0$, then $(\ref{eq : singsupp of V})$ will follow. We take $(v_{\varepsilon},w_{\varepsilon})_{\varepsilon \in (0,1]}$ satisfying problem $(\ref{eq : representative})$ with $(v_{0\varepsilon},w_{0\varepsilon}) = (\chi(u_{1\varepsilon} - c_{\varepsilon}(0)u_{0\varepsilon}^{\prime}),0)$. Differentiate the first equation of system $(\ref{eq : representative})$ in $x$, $k$ times, multiply by $2(\partial_x^{k}v_{\varepsilon})$ and integrate to obtain
\[
	\dfrac{d}{dt} \int_{-\infty}^{\infty} (\partial_x^{k}v_{\varepsilon})^2\,dx = 2\mu_{\varepsilon}(t) \int_{-\infty}^{\infty} (\partial_x^{k}v_{\varepsilon})^2\,dx - 2\mu_{\varepsilon}(t) \int_{-\infty}^{\infty} (\partial_x^{k}v_{\varepsilon})(\partial_x^{k}w_{\varepsilon})\,dx.
\]
Similarly
\[
	\dfrac{d}{dt} \int_{-\infty}^{\infty} (\partial_x^{k}w_{\varepsilon})^2\,dx = 2\mu_{\varepsilon}(t) \int_{-\infty}^{\infty} (\partial_x^{k}w_{\varepsilon})^2\,dx - 2\mu_{\varepsilon}(t) \int_{-\infty}^{\infty} (\partial_x^{k}v_{\varepsilon})(\partial_x^{k}w_{\varepsilon})\,dx.
\]
Take their difference and integrate to get
\begin{align*}
	& \int_{-\infty}^{\infty} \bigl((\partial_x^{k}v_{\varepsilon})^2 - (\partial_x^{k}w_{\varepsilon})^2\bigr)\,dx \\
	& \hspace{30pt} = \int_{-\infty}^{\infty} \bigl((\p_x^k v_{0\varepsilon})^2 - (\p_x^k w_{0\varepsilon})^2\bigr)dx \cdot \exp\left(2\int_0^t \mu_{\varepsilon}(s)\,ds\right).
\end{align*}
Since $w_{0\varepsilon} \equiv 0$ and $\mu_{\varepsilon} = c_{\varepsilon}^{\prime}/(2c_{\varepsilon})$, it follows that
\begin{equation}\label{eq : inequality1}
	\int_{-\infty}^{\infty} (\partial_x^{k}v_{\varepsilon})^2\,dx \ge \dfrac{c_{\varepsilon}(t)}{c_{\varepsilon}(0)}\int_{-\infty}^{\infty} (\p_x^k v_{0\varepsilon})^2dx.
\end{equation}
Note that $v_{0\varepsilon}$ vanishes outside some compact set, independently of $\varepsilon \in (0,1]$. Then it is clear from finite propagation speed that, for some constant $C_1 > 0$,
\begin{equation}\label{eq : inequality2}
	\int_{-\infty}^{\infty} (\partial_x^{k}v_{\varepsilon})^2\,dx \le C_1 \|\partial_x^{k}v_{\varepsilon}\|_{L^{\infty}(\mathbb{R})}^2.
\end{equation}
The Sobolev embedding theorem yields that, for some constant $C_2 > 0$,
\begin{equation}\label{eq : inequality3}
	\| \p_x^k v_{0\varepsilon}\|_{L^{\infty}(\mathbb{R})}^2
	\le C_2 \sum_{\ell \le k+1} \|\p_x^{\ell}v_{0\varepsilon}\|_{L^{2}(\mathbb{R})}^2.
\end{equation}
Combining $(\ref{eq : inequality1})$-$(\ref{eq : inequality3})$, we get, for some constant $C_3 > 0$,
\begin{equation}\label{eq : inequality4}
	\|\p_x^k v_{0\varepsilon}\|_{L^{\infty}(\mathbb{R})}^2
	\le C_3 \sum_{\ell \le k+1} \|\partial_x^{\ell}v_{\varepsilon}\|_{L^{\infty}(\mathbb{R})}^2.
\end{equation}
By assumption, $V_0 = \chi (U_1 - (C|_{t=0}) U_0^{\prime})$ is not $\Ginf$-regular on $\mathbb{R}$ for any $\chi$. This and inequality $(\ref{eq : inequality4})$ imply that for any $\chi$, $V$ is not $\Ginf$-regular on $\mathbb{R}$ for any $t > 0$. Thus $(\ref{eq : singsupp of V})$ follows. Cases (c) and (d) can be argued similarly.
\end{proof}

When the coefficient $C$ in problem $(\ref{eq : generalized wave with DC in t})$ is $\Ginf$-regular, the same phenomenon as in Subsection\;\ref{subsec : 4.1} occurs: singularities can only be produced by the initial data, and there are no refracted rays, as will be seen next. Again, we use a mollifier $\varphi_{h(\eps)}$, where $h(\eps) > 0$ and $(1/h(\eps))_{\varepsilon \in (0,1]}$ is a slow scale net to regularize $c(t)$. Then the corresponding element $C \in \G(\mathbb{R})$ belongs to $\Ginf(\mathbb{R})$. First, we have the following global regularity result, which can be proven similarly to the proof of Proposition \ref{prop : regularity}.

\begin{proposition}\label{prop : regularity2}
Let $C$ be as described above and let $U \in \G([0,\infty) \times \mathbb{R})$ be the solution to problem $(\ref{eq : generalized wave with DC in t})$ with initial data $U_0$, $U_1 \in \G(\mathbb{R})$. If $C$, $U_0$, $U_1 \in \Ginf(\mathbb{R})$, then $U \in \Ginf([0,\infty) \times \mathbb{R})$.
\end{proposition}

Turning to propagation of a point singularity, we take $U_0$, $U_1 \in \Ginf(\mathbb{R}\setminus \{0\})$ as in Theorem $\ref{thm : propagation3}$. Again, the cut-off argument employed in the derivation of Proposition\;\ref{prop : propagation_ub} yields that the $\Ginf$-singular support of the solution $U$ to problem $(\ref{eq : generalized wave with DC in t})$ is contained in the set given by (\ref{eq:singsupp1}), noting that no singularity is present along the line $\{t=1\}$.

Next, the same argument as in the proof of Proposition 5.1 of \cite{DHO:2013} shows that the refracted rays
\[
	\left\{(t,x) \mid x = \pm\left(2\int_0^1 c(s)\,ds-\int_0^t c(s)\,ds\right),\ t \ge 1\right\}
\]
do not belong to the $\Ginf$-singular support of the solution $U$, and so the singularities can occur only along the transmitted rays (see Figure \ref{fig : singsupp of U3})
\begin{equation}\label{eq:singsupp3}
	\left\{(t,x) \mid x = \pm\int_0^t c(s)\,ds,\ t \ge 0\right\}.
\end{equation}

\begin{proposition}\label{prop : propagation_ub_reg}
Let $C$, $U_0$ and $U_1$ be as described above and let $U \in \G([0,\infty) \times \mathbb{R})$ be the solution to problem $(\ref{eq : generalized wave with DC in t})$. Then the set
$(\ref{eq:singsupp3})$ is an upper bound for the $\Ginf$-singular support of the solution $U$ to problem $(\ref{eq : generalized wave with DC in t})$.
\end{proposition}
As for the question whether the singularities actually propagate out, the same argument as in the proof of Theorem \ref{thm : propagation3} can be applied and thus the following theorem holds.

\begin{figure}[htbp]
\begin{center}
\includegraphics{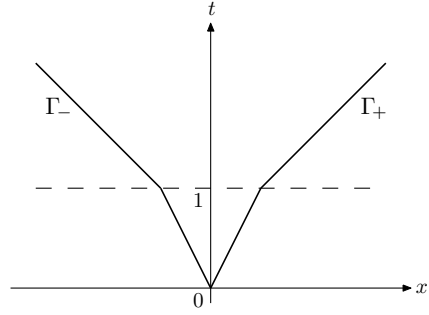}
\end{center}
\caption{An upper bound for the $\Ginf$-singular support of the solution $U$ with slow scale coefficient.}
\label{fig : singsupp of U3}
\end{figure}

\begin{theorem}\label{thm : propagation4}
Let $C$, $U_0$ and $U_1$ be as described above and let $U \in \G([0,\infty) \times \mathbb{R})$ be the solution to problem $(\ref{eq : generalized wave with DC in t})$. Let $C \in \Ginf(\R)$. Then:
\begin{itemize}
\item[{\rm (a)}] if $\singsupp_{\Ginf} (U_1-(C|_{t=0})U_0^{\prime}) =\{0\}$, then
\[
	\Gamma_{+} = \left\{(t,x) \mid x = \int_0^t c(s)\,ds,\ t \ge 0\right\} \subset \singsupp_{\Ginf} U;
\]
\item[{\rm (b)}] if $\singsupp_{\Ginf} (U_1-(C|_{t=0})U_0^{\prime}) =\emptyset$, then $U$ is $\Ginf$-regular along $\Gamma_{+}$;
\item[{\rm (c)}] if $\singsupp_{\Ginf} (U_1+(C|_{t=0})U_0^{\prime}) =\{0\}$, then
\[
	\Gamma_{-} = \left\{(t,x) \mid x = -\int_0^t c(s)\,ds,\ t \ge 0\right\} \subset \singsupp_{\Ginf} U;
\]
\item[{\rm (d)}] if $\singsupp_{\Ginf} (U_1+(C|_{t=0})U_0^{\prime}) =\emptyset$, then $U$ is $\Ginf$-regular along $\Gamma_{-}$.
\end{itemize}
\end{theorem}

\subsection{The time-dependent case in multiple dimensions}

We here consider problem $(\ref{eq : wave with DC in t})$ in multiple dimensions:
\[
	\begin{array}{l}
		\partial_t^2u - c(t)^2\Delta u = 0, \quad t \in [0,T], \ x \in \mathbb{R}^d, \vspace{4pt}\\
		u|_{t = 0} = u_{0},\quad \partial_tu|_{t = 0} = u_{1}, \quad x \in \mathbb{R}^d.
	\end{array}
\]
As in the one-dimensional case,
\[
	c(t) = c_0 + (c_1 - c_0)H(t-1),
\]
where $c_0$, $c_1 > 0$, $c_0 \ne c_1$ and $H$ is the Heaviside function. Specifically we take
\[
 	u_0 \equiv 0, \qquad u_1 = \delta,
\]
where $\delta$ is the delta function. We define a generalized function $C \in \cG_{\infty,2}[0,\infty)$ by means of a representative $(c_{\varepsilon})_{\varepsilon \in (0,1]} = (c \ast \varphi_{\varepsilon})_{\varepsilon \in (0,1]}$ with a mollifier $\varphi_{\varepsilon}$ as in Example 1
. We also produce $U_0 \equiv 0$ and $U_1 \in \G_{2,2}(\mathbb{R}^d)$ as the class of $(\varphi_{\varepsilon}(|x|))_{\varepsilon \in (0,1]}$, where $\varphi_{\varepsilon}$ again is a mollifier as in Example 1
. (Actually, the mollifier need not be the same as the one chosen for the regularization of the coefficient $c$.)
Thus we will consider
\begin{equation}\label{eq : generalized multi-dimensional wave}
	\begin{array}{lr}
		\partial_t^2U - C^2\Delta U = 0 & \mbox{in}\ \G_{2,2}([0,T]\times\mathbb{R}^d), \vspace{4pt}\\
		U|_{t = 0} = U_0 \equiv 0,\quad \partial_tU|_{t = 0} = U_1 & \mbox{in}\ \G_{2,2}(\mathbb{R}^d).
	\end{array}
\end{equation}
By Theorem \ref{thm:timewave}, problem $(\ref{eq : generalized multi-dimensional wave})$ has a unique solution $U \in \G_{2,2}([0,T] \times \mathbb{R}^d)$ for any $T > 0$.

We begin with the odd-dimensional case. It was shown in \cite[Theorem 4.1]{DHO:2013} that, in the one-dimensional case $(d = 1)$, it holds that
\begin{equation}\label{eq : one dimensional case}
	\begin{split}
		\singsupp_{\Ginf} U
		& = \left\{(t,x) \mid |x|=\int_0^t c(s)\,ds,\ 0 \le t \le T\right\} \\
		&\quad \cup \left\{(t,x) \mid |x|=\left|2\int_0^1 c(s)\,ds - \int_0^t c(s)\,ds\right|,\ 1 \le t \le T\right\}.
	\end{split}
\end{equation}
We are going to show that the same holds for all odd $d \ge 3$.

\begin{theorem}\label{thm : odd dimensional case}
Let $d \ge 3$ be odd. Assume that $C$, $U_0$ and $U_1$ are as described above and further that $U \in \G_{2,2}([0,T] \times \mathbb{R}^d)$ is the solution to problem $(\ref{eq : generalized multi-dimensional wave})$. Then the $\Ginf$-singular support of $U$ coincides with the set given by $(\ref{eq : one dimensional case})$. The support of $U$ also coincides with  the set given by $(\ref{eq : one dimensional case})$.

\end{theorem}

\begin{proof}
On the level of representatives, problem $(\ref{eq : generalized multi-dimensional wave})$ reads
\[
	\begin{array}{l}
		\partial_t^2u_{\varepsilon} - c_{\varepsilon}(t)^2\Delta u_{\varepsilon} = 0, \quad 0 < t < T,\ x \in \mathbb{R}^d,\\[2pt]
		u_{\varepsilon}|_{t=0} = 0,\quad \partial_t u_{\varepsilon}|_{t=0} = \varphi_{\varepsilon}(|x|), \quad x \in \mathbb{R}^d.
	\end{array}
\]
Since the initial data are radially symmetric, so is the solution $u_{\varepsilon}$. Put  $u_{d,\varepsilon}(t,r):= u_{\varepsilon}(t,x)$ for $|r| = |x|$, where the subscript $d$ represents the dependence of $u_{d,\varepsilon}$ on dimension. Then $u_{d,\varepsilon}$ satisfies the Cauchy problem
\begin{equation}\label{eq : u_d}
	\begin{array}{l}
		\partial_t^2u_{d,\varepsilon} - c_{\varepsilon}(t)^2\left(\partial_r^2 + \dfrac{d-1}{r}\partial_r \right)u_{d,\varepsilon} = 0, \quad 0 < t < T,\ r \in \mathbb{R},\\[2pt]	
		u_{d,\varepsilon}|_{t=0} = 0, \quad \partial_t u_{d,\varepsilon}|_{t=0} = \varphi_{\varepsilon}, \quad r \in \mathbb{R}.
	\end{array}
\end{equation}
According to \cite{G:1994}, if $u_{d,\varepsilon}$ satisfies the first equation in $(\ref{eq : u_d})$, then $(-1/r)\partial_r u_{d,\varepsilon}$ satisfies the first equation of $(\ref{eq : u_d})$ with $d$ replaced by $d+2$. This yields that, if $v_{n,\varepsilon}$, $n \in \mathbb{N}$, satisfies the Cauchy problem
\[
	\begin{array}{l}
		\partial_t^2v_{n,\varepsilon} - c_{\varepsilon}(t)^2\partial_r^2v_{n,\varepsilon} = 0, \quad 0 < t < T,\ r \in \mathbb{R},\\[2pt]
		v_{n,\varepsilon}|_{t=0} = 0, \quad r \in \mathbb{R},\\[2pt]
		\partial_t v_{n,\varepsilon}|_{t=0} = \int_{-\infty}^{r} (-s_n) \cdots \int_{-\infty}^{s_2}(-s_1)\varphi_{\varepsilon}(s_1)\,ds_1\cdots ds_n, \quad r \in \mathbb{R},
	\end{array}
\]
then $u_{2n+1,\varepsilon}:= [(-1/r)\partial_r]^n v_{n,\varepsilon}$ satisfies problem $(\ref{eq : u_d})$ with $d = 2n+1$. It is clear from the definition of $\varphi_{\varepsilon}$ that $\partial_t v_{n,\varepsilon}|_{t=0}$ is symmetric, $(\partial_t v_{n,\varepsilon}|_{t=0})^{\prime} \ge 0$ on $[-\varepsilon,0]$ and supp\,$\partial_t v_{n,\varepsilon}|_{t=0} \subset [-\varepsilon,\varepsilon]$. Hence the same argument as in the proof of Theorem 4.1 of \cite{DHO:2013} shows that the $\Ginf$-singular support of the class of $(v_{n,\varepsilon})_{\varepsilon \in (0,1]}$ and the support of the class of $(\partial_r v_{n,\varepsilon})_{\varepsilon \in (0,1]}$ are equal to the set given by $(\ref{eq : one dimensional case})$. This and the relationship $u_{2n+1,\varepsilon} = [(-1/r)\partial_r]^n v_{n,\varepsilon}$ imply that the $\Ginf$-singular support and the support of the class of $(u_{2n+1,\varepsilon})_{\varepsilon \in (0,1]}$ are the same as the set given by $(\ref{eq : one dimensional case})$. Thus the assertions follow.
\end{proof}

In the even-dimensional case, the following result holds (see Figure \ref{fig : upper bound2}).

\begin{figure}[htbp]
\begin{center}
\includegraphics[width=8cm]{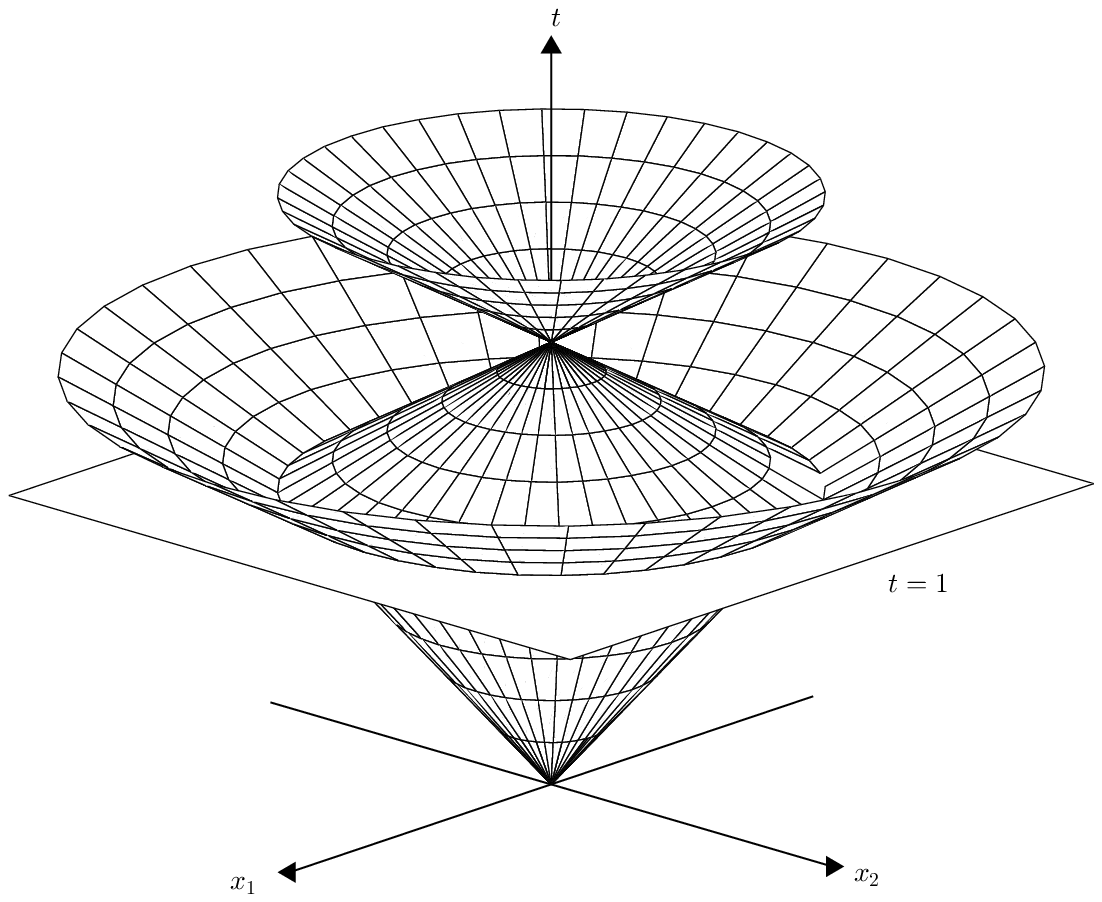}
\end{center}
\caption{An upper bound for the $\Ginf$-singular support of the solution $U$.}
\label{fig : upper bound2}
\end{figure}

\begin{theorem}\label{thm : even dimensional case}
Let $d$ be even. Assume that $C$, $U_0$ and $U_1$ are as in Theorem $\ref{thm : odd dimensional case}$ and further that $U \in \G_{2,2}([0,T] \times \mathbb{R}^d)$ is the solution to problem $(\ref{eq : generalized multi-dimensional wave})$. Then it holds that
\begin{equation}\label{eq : even dimensional case}
	\begin{split}
		\singsupp_{\Ginf} U
		& \subset \left\{(t,x) \mid t=1,\ |x|\le\int_0^1 c(s)\,ds\right\} \\[2pt]
		&\quad \cup \left\{(t,x) \mid |x|=\int_0^t c(s)\,ds,\ 0 \le t \le T\right\} \\[2pt]
		&\quad \cup \left\{(t,x) \mid |x|=\left|2\int_0^1 c(s)\,ds - \int_0^t c(s)\,ds\right|,\ 1 \le t \le T\right\}.
	\end{split}
\end{equation}
\end{theorem}

\begin{proof}
As in the proof of Theorem \ref{thm : odd dimensional case}, we consider problem $(\ref{eq : u_d})$. If $v_{n,\varepsilon}$, $n \in \mathbb{N}$, satisfies the Cauchy problem
\begin{equation}\label{eq : v2}
	\begin{array}{l}
		\partial_t^2v_{n,\varepsilon} - c_{\varepsilon}(t)^2\left(\partial_r^2 + \dfrac{1}{r}\partial_r \right)v_{n,\varepsilon} = 0, \quad 0 < t < T,\ r \in \mathbb{R},\\[2pt]	
		v_{n,\varepsilon}|_{t=0} = 0,  \quad r \in \mathbb{R},\\[2pt]
		\partial_t v_{n,\varepsilon}|_{t=0} \\[2pt]
		= \left\{
		\begin{array}{ll}
		\int_{-\infty}^{r} (-s_{n-1}) \cdots \int_{-\infty}^{s_2}(-s_1)\varphi_{\varepsilon}(s_1)\,ds_1\cdots ds_{n-1}, \quad r \in \mathbb{R},  & \mbox{if}\ n \ge 2,\vspace{4pt}\\
		\varphi_{\varepsilon}(r), \quad r \in \mathbb{R}, & \mbox{if}\ n = 1,
		\end{array}
		\right.
	\end{array}
\end{equation}
then $u_{2n,\varepsilon}:= [(-1/r)\partial_r]^{n-1} v_{n,\varepsilon}$ satisfies problem $(\ref{eq : u_d})$ with $d=2n$. By the argument as in \cite[p.66]{W:1961}, the solution $v_{n,\varepsilon}$ of problem $(\ref{eq : v2})$ is given by
\begin{equation}\label{eq : formula}
	v_{n,\varepsilon}(t,r) = \int_0^1 \dfrac{1}{\sqrt{1-\rho^2}} w_{n,\varepsilon}(t,r\rho)\,d\rho,
\end{equation}
provided that $(w_{n,\varepsilon})_{\varepsilon \in (0,1]} \in \EM([0,T]\times \mathbb{R})$ is symmetric and satisfies
\begin{align}
	& \partial_t^2w_{n,\varepsilon} - c_{\varepsilon}(t)^2\partial_r^2w_{n,\varepsilon} = 0, \quad 0 < t < T,\ r \in \mathbb{R}, \label{eq : w} \\[2pt]
	& w_{n,\varepsilon}|_{t=0} = 0, \quad r \in \mathbb{R}, \notag \\[2pt]
	& \partial_t v_{n,\varepsilon}|_{t=0} = \int_0^1 \dfrac{1}{\sqrt{1-\rho^2}} (w_{n,\varepsilon})_t(0,r\rho)\,d\rho, \quad  r \in \mathbb{R}. \label{eq : initial condition}
\end{align}
(For clarity of notation, we write $(w_{n,\varepsilon})_t$ for $(\p_t w_{n,\varepsilon})$.)
The existence of such $w_{n,\varepsilon}$ can be seen as follows. By the change of variable $y=r\rho$, we obtain from $(\ref{eq : initial condition})$ that
\begin{align*}
	\partial_t v_{n,\varepsilon}|_{t=0} & = \int_0^r \dfrac{|r|}{r\sqrt{r^2-y^2}}  (w_{n,\varepsilon})_t(0,y)\,dy = \int_0^{|r|} \dfrac{1}{\sqrt{r^2-y^2}} (w_{n,\varepsilon})_t(0,y)\,dy.
\end{align*}
Thanks to \cite{S:1963}, this integral equation can be solved and $\partial_t w_{n,\varepsilon}|_{t=0}$ is given by
\begin{align}
	\partial_t w_{n,\varepsilon}|_{t=0}
	= \dfrac{1}{\pi} \dfrac{d}{dr} \int_0^{1} \dfrac{2|r|\rho}{\sqrt{1-\rho^2}} (v_{n,\varepsilon})_t(0,r\rho)\,d\rho, \label{eq : initial condition'}
\end{align}
so that $(\partial_t w_{n,\varepsilon}|_{t=0})_{\varepsilon \in (0,1]}$ belongs to $\EM(\mathbb{R})$. This guarantees that $(w_{n,\varepsilon})_{\varepsilon \in (0,1]}$ exists. Note that the integrand is a smooth function of $r$, due to the symmetry of $(v_{n,\varepsilon})_t(0,\cdot)$.

We now investigate the regularity of $(w_{n,\varepsilon})_{\varepsilon \in (0,1]}$. Since supp\,$\partial_t v_{n,\varepsilon}|_{t=0} \subset [-\varepsilon,\varepsilon]$, we see from $(\ref{eq : initial condition'})$ that $(\partial_t w_{n,\varepsilon}|_{t=0})_{\varepsilon \in (0,1]}$ belongs to $\EMinf(\mathbb{R} \setminus \{0\})$. Therefore, as in the previous subsection, we can show that $(w_{n,\varepsilon})_{\varepsilon \in (0,1]}$ belongs to $\EMinf([0,T] \times \mathbb{R} \setminus \Gamma)$ with $\Gamma$ given by
\[
	\begin{split}
		& \Gamma:= \{(t,r) \mid t = 1,\ r \in \mathbb{R}\} \cup \left\{(t,r) \mid |r|=\int_0^t c(s)\,ds,\ 0 \le t \le T\right\} \\[2pt]
		&\qquad \cup \left\{(t,r) \mid |r|=\left|2\int_0^1 c(s)\,ds - \int_0^t c(s)\,ds\right|,\ 1 \le t \le T\right\}.
	\end{split}
\]

We next show that $(v_{n,\varepsilon})_{\varepsilon \in (0,1]}$ also belongs to $\EMinf([0,T] \times \mathbb{R} \setminus \Gamma)$. To this end, we take a cut-off function $\chi \in \mathcal{C}^{\infty}_0(\mathbb{R})$ which is identically equal to one in a neighborhood of the origin. Split $\partial_t w_{n,\varepsilon}|_{t=0}$ into two parts
\[
	\partial_t w_{n,\varepsilon}|_{t=0} = \chi(r) \partial_t w_{n,\varepsilon}|_{t=0} + (1-\chi(r)) \partial_t w_{n,\varepsilon}|_{t=0}.
\]
For any $\chi$, $(\chi \partial_t w_{n,\varepsilon}|_{t=0})_{\varepsilon \in (0,1]} \in \EMinf(\mathbb{R}\setminus\{0\})$ and $((1-\chi) \partial_t w_{n,\varepsilon}|_{t=0})_{\varepsilon \in (0,1]} \in \EMinf(\mathbb{R})$.
Let $w_{n,1,\varepsilon}$ and $w_{n,2,\varepsilon}$ be the solutions corresponding to the data $\chi \partial_t w_{n,\varepsilon}|_{t=0}$ and $(1-\chi) \partial_t w_{n,\varepsilon}|_{t=0}$, respectively. From the linearity of equation $(\ref{eq : w})$, $w_{n,\varepsilon} = w_{n,1,\varepsilon} + w_{n,2,\varepsilon}$. By $(\ref{eq : formula})$ we have
\begin{align*}
	v_{n,\varepsilon}(t,r)
	& = \int_0^1 \dfrac{1}{\sqrt{1-\rho^2}} w_{n,1,\varepsilon}(t,r\rho)\,d\rho + \int_0^1 \dfrac{1}{\sqrt{1-\rho^2}} w_{n,2,\varepsilon}(t,r\rho)\,d\rho \\
	& =:I_{n,1,\varepsilon}(t,r) + I_{n,2,\varepsilon}(t,r).
\end{align*}
It is clear from Proposition \ref{prop : regularity} that $(w_{n,2,\varepsilon})_{\varepsilon \in (0,1]} \in \EMinf(\{(t,r) \mid t \ne 1,\ r \in \mathbb{R}\})$ for any $\chi$.
This implies that $(I_{n,2,\varepsilon})_{\varepsilon \in (0,1]} \in \EMinf(\{(t,r) \mid t \ne 1,\ r \in \mathbb{R}\})$ for any $\chi$.
Hence, in order to show that $(v_{n,\varepsilon})_{\varepsilon \in (0,1]} \in \EMinf([0,T] \times \mathbb{R} \setminus \Gamma)$, it suffices to prove that $(I_{n,1,\varepsilon})_{\varepsilon \in (0,1]} \in \EMinf([0,T] \times \mathbb{R} \setminus \Gamma)$ in the limit as the support of $\chi$ shrinks to $\{0\}$. We divide $[0,T] \times \mathbb{R} \setminus \Gamma$ into the following six regions (see Figure\;\ref{fig:SixRegions}):
\[
	\begin{split}
		D_1 &:= \left\{(t,r) \mid |r| < \int_0^t c(s)\,ds,\ 0 < t < 1\right\}, \\[2pt]
		D_2 &:= \left\{(t,r) \mid |r| < 2\int_0^1 c(s)\,ds - \int_0^t c(s)\,ds,\ 1 < t < 1+c_0/c_1\right\},\\[2pt]
		D_3 &:= \left\{(t,r) \mid |r| < \int_0^t c(s)\,ds -2\int_0^1 c(s)\,ds,\ 1+c_0/c_1 < t < T\right\},\\[2pt]
		D_4 &:= \left\{(t,r) \mid \left|2\int_0^1 c(s)\,ds - \int_0^t c(s)\,ds\right| < |r| < \int_0^t c(s)\,ds,\ 1 < t < T\right\},\\[2pt]
		D_5 &:= \left\{(t,r) \mid |r| > \int_0^t c(s)\,ds,\ 0 < t < 1\right\},\\[2pt]
		D_6 &:= \left\{(t,r) \mid |r| > \int_0^t c(s)\,ds,\ 1 < t < T\right\}.
	\end{split}
\]
\begin{figure}[htbp]
\begin{center}
\includegraphics[width = 8cm]{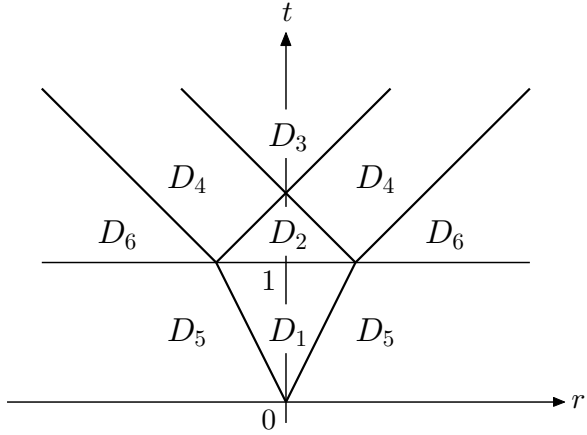}
\end{center}
\caption{The singular support in even space dimensions, radial cross section.}
\label{fig:SixRegions}
\end{figure}
It is easily checked that $(I_{n,1,\varepsilon})_{\varepsilon \in (0,1]} \in \EMinf(\sum_{j=1}^{3}D_j)$ for any $\chi$, using the fact that $(w_{n,1,\varepsilon})_{\varepsilon \in (0,1]} \in \EMinf(\sum_{j=1}^{3}D_j)$ for any $\chi$. As for the other regions, we first consider $D_4 \cap \{(t,r) \mid t < 1+c_0/c_1,\ r \ge 0\}$. Let $K \subset D_4 \cap \{(t,r) \mid t < 1+c_0/c_1,\ r \ge 0\}$ be any fixed compact set.
We start by proving the following support property of the derivatives of $w_{n,1,\varepsilon}$:
\begin{itemize}
\item[(SP)]
 For any $(t,r) \in K$, there exist $0 < r_0$, $r_1 < r$ independently of $\varepsilon$, such that $(w_{n,1,\varepsilon})_r(t,r^{\prime}) = (w_{n,1,\varepsilon})_t(t,r^{\prime}) = 0$ for $0\le r^{\prime} \le r_0$ and $r_1 \le r^{\prime} \le r$, possibly after shrinking the support of $\chi$.
\end{itemize}
Indeed, $w_{n,1,\varepsilon}$ solves the problem
\[
	\begin{array}{rclcl}
		\partial_t^2 w_{n,1,\varepsilon} - c_{\varepsilon}(t)^2\partial_r^2 w_{n,1,\varepsilon} &=& 0, &\ &t > 0,\ r \in \mathbb{R},\\[2pt]
		w_{n,1,\varepsilon}|_{t = 0}	& = & 0,&  & r \in \mathbb{R},\\[2pt]
		\partial_t w_{n,1,\varepsilon}|_{t = 0} & =& \chi \partial_t w_{n,\varepsilon}|_{t = 0}, &  & r \in \mathbb{R},
	\end{array}
\]
so
\begin{align*}
	& f_{n,\varepsilon} = (\partial_t - c_{\varepsilon}\partial_r)w_{n,1,\varepsilon},\\
	& g_{n,\varepsilon} = (\partial_t + c_{\varepsilon}\partial_r)w_{n,1,\varepsilon}
\end{align*}
satisfy the corresponding first order hyperbolic system
\[
	\begin{array}{rclcl}
		(\partial_t + c_{\varepsilon}\partial_r) f_{n,\varepsilon}  &=& \mu_{\varepsilon}(t)(f_{n,\varepsilon} - g_{n,\varepsilon}),
&\ &t > 0,\ r \in \mathbb{R},\\[2pt]
		(\partial_t - c_{\varepsilon}\partial_r) g_{n,\varepsilon}  &=& \mu_{\varepsilon}(t)(g_{n,\varepsilon} - f_{n,\varepsilon}),
&\ &t > 0,\ r \in \mathbb{R},\\[2pt]			
		f_{n,\varepsilon}|_{t = 0}	& = & \chi \partial_t w_{n,\varepsilon}|_{t = 0},&  & r \in \mathbb{R},\\[2pt]
		g_{n,\varepsilon}|_{t = 0}	& = & \chi \partial_t w_{n,\varepsilon}|_{t = 0},&  & r \in \mathbb{R},
	\end{array}
\]
where $\mu_{\varepsilon}(t) = c'_{\varepsilon}(t)/2c_{\varepsilon}(t)$. The supports of $f_{n,\varepsilon}$ and $g_{n,\varepsilon}$ are contained in the set
\[
	\begin{split}
		& \left\{(t,r) \mid r = \pm\int_0^t c(s)\,ds,\ t \geq 0\right\} \\[2pt]
		&\qquad \cup \left\{(t,r) \mid r = \pm\left(2\int_0^1 c(s)\,ds - \int_0^t c(s)\,ds\right),\ t\geq 1\right\}
	\end{split}
\]
in the limit as the support of $\chi$ shrinks to $\{0\}$. It follows from the definitions of $f_{n,\varepsilon}$ and
$g_{n,\varepsilon}$ that the supports of $\partial_t w_{n,1,\varepsilon}$ and $\partial_r w_{n,1,\varepsilon}$ are contained in the same set as the support of $\chi$ shrinks to $\{0\}$. This proves the indicated support property (SP).

We further observe that $c_{\varepsilon}(t)=c_1$ on $K$ for $\varepsilon \in (0,1]$ small enough. Taking the $2m$-th derivative of $I_{n,1,\varepsilon}$ with respect to $t$ and integrating by parts, we get, for $(t,r) \in K$ and $\varepsilon \in (0,1]$ small enough,
\begin{align*}
	\partial_t^{2m} I_{n,1,\varepsilon}(t,r)
	& = \dfrac{c_1^{2m}}{r^{2m-1}}\int_0^1 \dfrac{1}{\sqrt{1-\rho^2}} \partial_{\rho}^{2m-1}(w_{n,1,\varepsilon})_r(t,r\rho)\,d\rho \\
	& = \dfrac{(-1)^{2m-1}c_1^{2m}}{r^{2m-1}} \int_0^1 \left(\partial_{\rho}^{2m-1} \dfrac{1}{\sqrt{1-\rho^2}}\right) (w_{n,1,\varepsilon})_r(t,r\rho)\,d\rho.
\end{align*}
Note that the integration by parts is justified, because $(w_{n,1,\varepsilon})_r(t,r\rho) = 0$ near $\rho = 0$ and $\rho = 1$, due to property (SP).
Similarly
\begin{align*}
	\partial_t^{2m-1} I_{n,1,\varepsilon}(t,r)
	& = \dfrac{(-1)^{2m-2}c_1^{2m-2}}{r^{2m-2}} \int_0^1 \left(\partial_{\rho}^{2m-2} \dfrac{1}{\sqrt{1-\rho^2}}\right) (w_{n,1,\varepsilon})_{t}(t,r\rho)\,d\rho.
\end{align*}
Hence, for any $\alpha \in \mathbb{N}_0$, there exists $\kappa > 0$ such that
\begin{align*}
	\|\partial_t^{\alpha} I_{n,1,\varepsilon}\|_{L^{\infty}(K)}
	& \le \kappa \max\{\|w_{n,1,\varepsilon}\|_{L^{\infty}([0,T] \times \mathbb{R})}, \|\partial_{r}w_{n,1,\varepsilon}\|_{L^{\infty}([0,T] \times \mathbb{R})}, \\
	& \qquad \qquad \qquad \|\partial_{t}w_{n,1,\varepsilon}\|_{L^{\infty}([0,T] \times \mathbb{R})}\}
\end{align*}
for sufficiently small $\varepsilon \in (0,1]$. Analogous estimates for the $r$- and mixed derivatives can be obtained. The same argument applies in the case that $K \subset (D_4 \cup D_5 \cup D_6) \cap \{(t,r) \mid t \ne 1+c_0/c_1,\ r \in \mathbb{R}\}$ is any fixed compact set. Therefore, taking into account finite propagation speed, we obtain that $(I_{n,1,\varepsilon})_{\varepsilon \in (0,1]} \in \EMinf(D_4 \cup D_5 \cup D_6)$ as the support of $\chi$ shrinks to $\{0\}$. Thus $(I_{n,1,\varepsilon})_{\varepsilon \in (0,1]} \in \EMinf([0,T] \times \mathbb{R} \setminus \Gamma)$ as the support of $\chi$ shrinks to $\{0\}$.

We finally prove that $(u_{2n,\varepsilon})_{\varepsilon \in (0,1]} \in \EMinf([0,T] \times \mathbb{R} \setminus \Gamma^{\prime})$ with $\Gamma^{\prime}$ given by
\[
	\begin{split}
		& \Gamma^{\prime}:= \left\{(t,r) \mid t=1,\ |r| \le \int_0^1 c(s)\,ds\right\} \\[2pt]
		& \qquad \cup \left\{(t,r) \mid |r|=\int_0^t c(s)\,ds,\ 0 \le t \le T\right\} \\[2pt]
		&\qquad \cup \left\{(t,r) \mid |r|=\left|2\int_0^1 c(s)\,ds - \int_0^t c(s)\,ds\right|,\ 1 \le t \le T\right\}.
	\end{split}
\]
This means that assertion $(\ref{eq : even dimensional case})$ holds.
From the first equation of problem $(\ref{eq : v2})$, we have
\[
	-\dfrac{1}{r}\partial_r v_{n,\varepsilon}(t,r) = -\dfrac{\partial_t^2v_{n,\varepsilon}(t,r) - c_{\varepsilon}(t)^2\partial_r^2v_{n,\varepsilon}(t,r)}{c_{\varepsilon}(t)^2}.
\]
Note that, on every compact subset outside $\{(t,r) \mid t=1,\ r \in \mathbb{R}\}$, $c_{\varepsilon}(t) = c$, when $\varepsilon > 0$ is small enough. Therefore $((-1/r)\partial_r v_{n,\varepsilon})_{\varepsilon \in (0,1]} \in \EMinf([0,T]\times \mathbb{R} \setminus \Gamma)$. A direct calculation shows that $(-1/r)\partial_r v_{n,\varepsilon}$ satisfies the first equation in problem $(\ref{eq : u_d})$ with $d=4$. Thus we get
\[
	\left(-\dfrac{1}{r}\partial_r\right)^2 v_{n,\varepsilon} = -\dfrac{\partial_t^2\Bigl((-1/r)\partial_r v_{n,\varepsilon}\Bigr) - c_{\varepsilon}(t)^2\partial_r^2\Bigl((-1/r)\partial_r v_{n,\varepsilon}\Bigr)}{3c_{\varepsilon}(t)^2},
\]
so $([(-1/r)\partial_r]^2 v_{n,\varepsilon})_{\varepsilon \in (0,1]} \in \EMinf([0,T]\times \mathbb{R} \setminus \Gamma)$. By induction, we see that $(u_{2n,\varepsilon})_{\varepsilon \in (0,1]}=([(-1/r)\partial_r]^{n-1} v_{n,\varepsilon})_{\varepsilon \in (0,1]} \in \EMinf([0,T]\times \mathbb{R} \setminus \Gamma)$ for any $n \in \mathbb{N}$. In addition, by finite propagation speed,
\[
	(u_{2n,\varepsilon})_{\varepsilon \in (0,1]} \in \Neg\left(\left\{(t,r) \mid |r| > \int_0^t c(s)\,ds,\ 0 \le t \le T\right\}\right).
\]
Thus $(u_{2n,\varepsilon})_{\varepsilon \in (0,1]} \in \EMinf([0,T] \times \mathbb{R} \setminus \Gamma^{\prime})$ for any $n \in \mathbb{N}$. The proof of Theorem $\ref{thm : even dimensional case}$ is now complete.
\end{proof}

\begin{remark}
If the coefficient $c(t) = c_0 + (c_1 - c_0)H(t-1)$ is regularized in such a way that the corresponding element $C\in\cG_{\infty,2}(\R)$ is $\Ginf$-regular on $\R$, then as in the one-dimensional case, one can show that
\[
	\singsupp_{\Ginf} U = \displaystyle \left\{(t,x) \mid |x| = \int_0^t c(s)\,ds,\, 0 \le t \le T\right\}
\]
in any space dimension.
\end{remark}

\section{Appendix: Associated distributions}

In case the initial data and the coefficients are sufficiently smooth, hyperbolic equations and systems possess classical or distributional solutions. In numerous such situations it has been verified that the Colombeau solution admits the classical solution as associated distribution. This is generally the case when the coefficients are $\Cinf$-smooth (see e.g. \cite{LO:1991, O:1992}). An instance with merely $L^\infty$-coefficients is the linearized Euler system of isentropic gas dynamics, in which the association result has been proven in \cite{GO:2011b, O:1989}. For the one dimensional wave equation with a coefficient depending on time only and suffering a jump, i.e., problem\ (\ref{eq:1Dwavetime}), the Colombeau solution was shown to be associated with the piecewise distributional solution for arbitrary distributions as initial data in \cite{DHO:2013}.

If the coefficients are piecewise constant with a jump across a smooth hyperplane in space, the problem may be considered classically as a transmission problem. The one-dimensional problem
\begin{equation}\label{eq:wave1D}
   a(x)\p^2_t z(t,x) - \p_x\big(b(x)\p_x z(t,x) \big) = 0
\end{equation}
with initial conditions
\begin{equation}\label{eq:ICwave1D}
   z(0,x) = z_0(x),\qquad \p_t z(0,x) = z_1(x)
\end{equation}
has been treated in \cite{GO:2011b}, assuming that $a$ and $b$ are strictly positive and piecewise constant with a jump at $x=0$.
Viewing (\ref{eq:wave1D}) as a transmission problem across $x=0$, and assuming that ${z}_0$ is continuously differentiable and ${z}_1$ is continuous, there is a unique distributional solution $\overline{z} \in {\mathcal D}'(\R^2)$, the \emph{classical connected solution}, characterized by the following properties:
\begin{itemize}
\item[-] $\overline{z}$ is a distributional solution to \eqref{eq:wave1D} on the open half planes $\{x < 0\}$ and $\{x > 0\}$;
\item[-] $\overline{z} \in {\mathcal C}\big([0,\infty)\times \R\big)$ and $\overline{z}(0,\cdot) = z_0$;
\item[-] $\frac{\p}{\p t}\overline{z} \in {\mathcal C}\big([0,\infty)\times \R\big)$ and $\frac{\p}{\p t}\overline{z}(0,\cdot) = z_1$;
\item[-] the function $x\to {b}(x)\frac{\p}{\p x} \overline{z}(t,x)$ is continuous for almost all $t\geq 0$.
\end{itemize}
A Colombeau solution $Z\in \cG(\R^2)$ can be constructed corresponding to initial data $\iota(z_0)$, $\iota(z_1)$, with suitable interpretations of $a(x)$, $b(x)$ as elements $A$, $B$ of $\cG(\R)$. It was shown in \cite{GO:2011b} that the Colombeau solution $Z\in\cG(\R^2)$ admits the classical connected solution $\overline{z}$ as associated distribution. The problem (\ref{eq:nonconswave1D}) is a special case of (\ref{eq:wave1D}).

It is the purpose of this appendix to compute the classical connected solution, to extend the calculation to delta functions as initial data, to show that the Colombeau solution is associated with the distributional solution in this case as well, and to verify that the location of the singular support of the associated distribution coincides with the $\Ginf$-singular support computed in Theorem\;\ref{thm : propagation1}. We begin by computing the classical connected solution to the transmission problem
\begin{equation}\label{eq:wavetransmission}
\begin{array}{l}
   \p_t^2 u(t,x) - c(x)^2\p_x^2 u(t,x) = 0, \quad  t > 0,\ x\in \R,\ x\neq 0, \vspace{4pt} \\
   u(0,x) = u_0(x), \quad \p_t u(0,x) = u_1(x), \quad x\in \R
\end{array}
\end{equation}
where
\[
   c(x) = \left\{
          \begin{array}{ll}
          c_-, & x < 0,\\
          c_+, & x > 0
          \end{array}\right.
\]
and both $c_-$ and $c_+$ are strictly positive. As noted above, the classical transmission conditions at $x=0$ are
\begin{itemize}
\item[-] $u$ is continuous across $x=0$ and
\item[-] $\p_xu$ is continuous across $x=0$.
\end{itemize}
The first condition entails that also $\p_t u$ is continuous across $x=0$.
We introduce
\[
   v = (\p_t - c\p_x)u,\qquad w = (\p_t + c\p_x)u,
\]
\[
   v_0(x) = u_1(x) - c(x)u_0'(x),\qquad w_0(x) = u_1(x) + c(x)u_0'(x)
\]
and rewrite equation (\ref{eq:wavetransmission}) on either side of the line $x=0$ as the system
\begin{equation}\label{eq:transmissionsystem}
   \begin{array}{l}
   (\p_t + c\p_x)v = 0 \\[4pt]
   (\p_t - c\p_x)w = 0 \\[4pt]
   v(0,x) = v_0(x),\ w(0,x) = w_0(x)
   \end{array}
\end{equation}
where $c= c_+$ for $x>0$ and $c = c_-$ for $x<0$.
Note that we can recover
\[
  \p_tu = \frac12(v+w),\qquad \p_xu = \frac{1}{2c}(w-v),
\]
\[
    u(t,x) = u_0(x) + \int_0^t\p_tu(s,x)\,ds.
\]
We shall use the suffix $\pm$ to indicate parts of the solution on $x<0$ and on $x>0$. Consider the regions displayed in Figure\;\ref{fig:RegionsWaveEq}, which are separated by the lines $x = - c_-t$, $x = 0$, $x = c_+t$ (plotted with $c_- = 1$, $c_+ =2$).
\begin{figure}[htb]
\begin{center}
\includegraphics[width = 6cm]{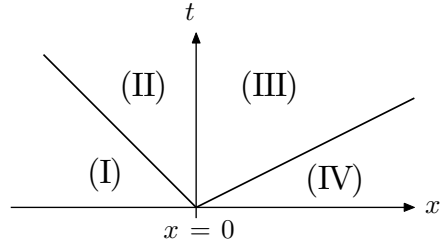}
\end{center}
\caption{Regions of the classical transmission problem.}
\label{fig:RegionsWaveEq}
\end{figure}

Clearly,
\begin{itemize}
\item[] $v_-(t,x) = v_0(x-c_-t)$ in regions (I) and (II),
\item[] $v_+(t,x) = v_0(x-c_+t)$ in region (IV),
\item[] $w_+(t,x) = w_0(x+c_+t)$ in regions (III) and (IV),
\item[] $w_-(t,x) = w_0(x+c_-t)$ in region (I).
\end{itemize}
It remains to determine $v_+$ in region (III) and $w_-$ in region (II). To this end we recall that $\frac12(v+w)$ and $\frac{1}{2c}(w-v)$ should not jump across $x=0$. This yields
\begin{eqnarray*}
   v_+ + w_+ &=& v_- + w_-\\
   \tfrac1{c_+}(w_+ - v_+) & = & \tfrac1{c_-}(w_- - v_-)
\end{eqnarray*}
along $x=0$. Solving this linear system of equations results in
\[
    v_+ = \frac{2c_+}{c_+ + c_-}v_- + \frac{c_- - c_+}{c_+ + c_-}w_+,\qquad w_- = \frac{2c_-}{c_+ + c_-}w_+ + \frac{c_+ - c_-}{c_+ + c_-}v_-.
\]
Together with the known formulas $v_-(t,x) = v_0(x-c_-t)$ in region (II), $w_+(t,x) = w_0(x+c_+t)$ in region (III), we get
\[
    v_+(t,0) = \frac{2c_+}{c_+ + c_-}v_0(-c_-t) + \frac{c_- - c_+}{c_+ + c_-}w_0(c_+t),
\]\[
     w_-(t,0) = \frac{2c_-}{c_+ + c_-}w_0(c_+t) + \frac{c_+ - c_-}{c_+ + c_-}v_0(-c_-t).
\]
Propagating the boundary information along characteristic lines, we obtain
\[
    v(t,x) = v_+\big(t-\frac{x}{c_+},0\big)\quad\mbox{for\ } (t,x) \mbox{\ in\ region\ (III)},
\]
\[
    w(t,x) = w_-\big(t+\frac{x}{c_-},0\big)\quad\mbox{for\ } (t,x) \mbox{\ in\ region\ (II)}.
\]
Collecting terms, we arrive at the following solution.
\medskip\\
In region (I):
\[
    v(t,x) = v_0(x-c_-t),\quad w(t,x) = w_0(x + c_-t);
\]
in region (II):
\begin{eqnarray*}
    v(t,x) &=& v_0(x-c_-t),\\ w(t,x) & = & \frac{2c_-}{c_+ + c_-}w_0\big(\frac{c_+}{c_-}(x+c_-t)\big) +
         \frac{c_+ - c_-}{c_+ + c_-}v_0(-x-c_-t);
\end{eqnarray*}
in region (III):
\begin{eqnarray*}
   v(t,x) &=& \frac{2c_+}{c_+ + c_-}v_0\big(\frac{c_-}{c_+}(x-c_+t)\big) + \frac{c_- - c_+}{c_+ + c_-}w_0(c_+t - x),\\
    w(t,x) &=& w_0(x + c_+t);
\end{eqnarray*}
in region (IV):
\[
    v(t,x) = v_0(x-c_+t),\quad w(t,x) = w_0(x + c_+t).
\]
Finally, the solution $u$ to (\ref{eq:wavetransmission}) is obtained by integrating $\frac12(v+w)$ with respect to $t$.

As a first consequence of the solution formula, we see that $w$ is not necessarily continuous across the line $x+c_-t = 0$, separating regions (I) and (II), and $v$ is not necessarily continuous across the line $x-c_+t = 0$, separating regions (III) and (IV). This shows the singularities emanating from the origin due to lack of compatibility conditions of the initial data with the jump data.

As a second consequence, we can demonstrate the existence of a reflected singularity by taking data
\[
   v_0(x) = w_0(x) = \delta(x+1),
\]
corresponding to delta function initial data $u_1(x) = \delta(x+1)$, $u_0(x) = 0$. In this case, the transmission condition has to be understood in the sense of limits as $x\to\pm 0$ in $\cD'(0,\infty)$ (functions of $x$ with values in the space of distributions with respect to $t$). Similarly, the initial data are taken in the sense of $\cC\big([0,\infty):\cD'(\R)\big)$. We obtain the solution
\[
   v(t,x) = \delta(x-c_-t + 1),\quad w(t,x) = \delta(x + c_-t + 1) + \frac{c_+ - c_-}{c_+ + c_-}\delta(-x-c_-t+1)
\]
for $x<0$, containing the reflected singularity, and
\[
   v(t,x) = \frac{2c_+}{c_+ + c_-}\delta\big(\frac{c_-}{c_+}(x-c_+t)+1\big),\quad w(t,x) = 0
\]
for $x>0$, containing the transmitted singularity. The geometry is depicted in Figure\;\ref{fig:DeltaReflection}
(with $c_- = 1$, $c_+ =2$).

\begin{figure}[htbp]
\begin{center}
\includegraphics[width = 6cm]{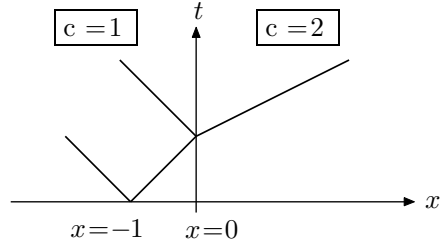}
\end{center}
\caption{Propagation of delta function at $x=-1$.}
\label{fig:DeltaReflection}
\end{figure}

By inspection, one verifies that $v+w$ and $(w-v)/c$ are continuous maps of $x$ with values in $\cD'(0,\infty)$, as required.
Further, we find that
\begin{eqnarray*}
   u(t,x) & = & \frac12\int_0^t(v(s,x) + w(s,x))\,ds\\
      & = &\frac1{2c_-}\Big(H(-x + c_-t - 1)H(x + c_-t+1) + \frac{c_+ - c_-}{c_+ + c_-}H(x+c_-t-1)\Big)
\end{eqnarray*}
for $x<0$. On the other hand,
\[
   \delta\big(\frac{c_-}{c_+}(x-c_+t)+1\big) = \frac{c_+}{c_-}\delta\big(x - c_+t + \frac{c_+}{c_-}\big)
\]
so that
\[
   u(t,x) = \frac1{2c_-}\cdot\frac{2c_+}{c_+ + c_-}H\big(-x + c_+t - \frac{c_+}{c_-}\big)
\]
for $x>0$. In the sector above $t=1$, $u(t,x)$ has the common value $\frac{c_+}{c_-(c_+ + c_-)}$ on either side of $x=0$.

It remains to be shown that the Colombeau solution $U\in\cG(\hpl)$ with initial data $U_0 = 0$, $U_1 = \iota(\delta(\cdot +1))$ is associated with the distributional solution $u(t,x)$ computed above. In fact, the problem can be reduced to (\ref{eq:wave1D}) with $a(x) = 1/c(x)^2$, $b(x) = 1$ and initial data $z_0(x) = 0$, $z_1(x) = (x+1)H(x+1)a_-$ with $a_- = 1/c_-^2$, more precisely, its Colombeau version
\begin{equation*}
\begin{array}{lr}
   A\p_t^2 Z - \p_x^2 Z = 0 & {\rm in\ }  \cG(\hpl),\vspace{4pt} \\
   Z|_{t = 0} = 0,\quad \partial_t Z|_{t = 0} = \iota(z_1) & \mbox{in}\ \cG(\mathbb{R})
\end{array}
\end{equation*}
with $A = \iota(a)$. Indeed, differentiating this equation twice with respect to $t$ shows that $U = \p_t^2 Z$ satisfies the same equation, but with initial data
\[
   U|_{t = 0} = \p_t^2 Z|_{t = 0} = 0,
\]\[
   \p_tU|_{t = 0} = \p_t\big(\p_t^2 Z|_{t = 0}\big) = \frac1{A}\p_x^2\iota(z_1) = \iota(\delta(\cdot+1)).
\]
Here we have used the fact that $\supp\iota(\delta(\cdot+1)) = \{-1\}$ and that $A$ is constant (equal to $a_-$) there.
Since $z_1(x) = (x+1) H(x+1)a_-$ is continuous, the quoted result from \cite{GO:2011b} can be applied to the classical connected solution, showing that it is the associated distribution corresponding to $Z$. Hence $U = \p_t^2 Z$ is associated with the second time derivative of the classical connected solution, which equals the function $u(t,x)$ just computed, as can be verified by direct calculation.

\section*{Acknowledgment}
Part of this work was done while the first author visited Universit\"{a}t Innsbruck, July 30--October 5, 2013. He expresses his heartfelt thanks to the Unit of Engineering Mathematics for the warm hospitality.


%
%
%


\begin{thebibliography}{1}
\bibitem{A:1975} R. Adams, {\em Sobolev Spaces}. Academic Press, New York, 1975.
\bibitem{B:1990} H. A. Biagioni, {\em A nonlinear theory of generalized functions}. Lect. Notes Math. 1421. Springer-Verlag, Berlin, 1990.
\bibitem{BO:1992} H. A. Biagioni, M. Oberguggenberger, {\em Generalized solutions to the Korteweg-de Vries and the regularized long-wave equations}. SIAM J. Math. Anal. {\bf 23} (1992), 923--940.
\bibitem{C:1985} J. F. Colombeau, {\em Elementary introduction to new generalized functions}. North-Holland Math. Stud. 113. North-Holland, Amsterdam, 1985.
\bibitem{Colombini:79}
F. Colombini, E. De Giorgi, and S. Spagnolo.
Sur les \'{e}quations hyperboliques avec des coefficients qui ne d\'{e}pendent que du temps.
{\em Ann. Scuola Norm. Sup. Pisa Cl. Sci. (4)}, 6:511--559, 1979.
\bibitem{Colombini:95}
F. Colombini, and N. Lerner.
Hyperbolic operators with non-Lipschitz coefficients.
{\em Duke Math. J.}, 77(3):657--698, 1995.
\bibitem{D:2011} H. Deguchi, {\em A linear first-order hyperbolic equation with a discontinuous coefficient: distributional shadows and propagation of singularities}. Electron. J. Differential Equations {\bf 2011}/76 (2011), 1--25.
\bibitem{DHO:2013} H. Deguchi, G.~H\"ormann, M.~Oberguggenberger, {\em The wave equation with a discontinuous coefficient depending on time only: generalized solutions and propagation of singularities}. In: S. Molahajloo, S. Pilipovi\'c, J. Toft and M. W. Wong (Eds.), Pseudo--differential operators, generalized functions and asymptotics, pp. 323--339, Birkh\"{a}user, Basel, 2013.
\bibitem{HMO:2008} M. de~Hoop, G. H{\"o}rmann, M. Oberguggenberger,
{\em Evolution systems for paraxial wave equations of {S}chr\"odinger-type
  with non-smooth coefficients.}
J. Differential Equations {\bf 245} (2008), 1413--1432.
\bibitem{G:2005}
C.~Garetto, {\em Topological Structures in Colombeau Algebras: Investigation of the Duals of $\Gc(\Omega)$, $\cG(\Omega)$ and $\cG_{\cS}(\R^n)$}. Monatsh. Math. {\bf 146} (2005), 203--226.
\bibitem{G:2006}
C.~Garetto, {\em Microlocal analysis in the dual of a {C}olombeau algebra: generalized wave front sets and noncharacteristic regularity}. New York J. Math. {\bf 12} (2006), 275--318.
\bibitem{G:ISAAC07}
C.~Garetto, {\em Generalized Fourier integral operators on spaces of Colombeau type}.
In: L. Rodino, M. W. Wong (Eds.), New developments in pseudo-differential operators, pp. 137--184, Birkh\"{a}user, Basel 2009.
\bibitem{GGO:2003}
C.~Garetto, T.~Gramchev, M.~Oberguggenberger, {\em Pseudodifferential operators with generalized symbols and regularity theory}. Electron. J. Differential Equations {\bf 2005}/116 (2005), 1--43.
\bibitem{GH:2005}
C.~Garetto, G.~H\"{o}rmann, {\em Microlocal analysis of generalized functions: pseudodifferential
  techniques and propagation of singularities}.
Proc. Edinburgh Math. Soc. {\bf 48} (2005), 603--629.
\bibitem{GH:2006} C. Garetto, G. H\"{o}rmann, {\em On duality theory and pseudodifferential techniques for Colombeau algebras: generalized delta functionals, kernels and wave front sets}. Bull. Cl. Sci. Math. Nat. Sci. Math. {\bf 31} (2006), 115--136.
\bibitem{GHO:2009}
C.~Garetto, G.~H\"ormann, M.~Oberguggenberger, {\em Generalized oscillatory integrals and {F}ourier integral operators}.
Proc. Edinburgh Math. Soc.  {\bf 52} (2009), 351--386.
\bibitem{GO:2011a}
C.~Garetto, M.~Oberguggenberger, {\em Fourier integral operator methods for hyperbolic equations with singularities}.
Proc. Edinburgh Math. Soc. {\bf  57} (2014), 423--463.
\bibitem{GO:2011b}
C.~Garetto, M.~Oberguggenberger, {\em Symmetrisers and generalised solutions for strictly hyperbolic systems with singular coefficients}. Math. Nachrichten {\bf 288} (2015), 185--205.
\bibitem{G:1994}
J. J.~Giambiagi, {\em Relations among solutions for wave and Klein-Gordon equations for different dimensions}.
Nuovo Cimento B (11) 109 (1994), 635--644.
\bibitem{GKOS:2001} M. Grosser, M. Kunzinger, M. Oberguggenberger, R. Steinbauer, {\em Geometric theory of generalized functions with applications to general relativity}. Mathematics and its Applications 537. Kluwer Acad. Publ., Dordrecht, 2001.
\bibitem{Haller:2009}
S. Haller, {\em Microlocal analysis of generalized pullbacks of {C}olombeau functions.}
Acta Appl. Math. {\bf 105} (2009), 83--109.
\bibitem{HallerH:2008}
S. Haller, G.~H\"{o}rmann, {\em Comparison of some solution concepts for linear first-order hyperbolic differential equations with non-smooth coefficients}. Publ. Inst. Math. {\bf 84(98)} (2008), 123--157.
\bibitem{Haneletal:2013}
C. Hanel,  G.~H\"{o}rmann, C. Spreitzer, R. Steinbauer, {\em Wave equations and symmetric first-order systems in case of low
regularity.} In: S. Molahajloo,  S.~Pilipovi{\'c}, J. Toft, M. W. Wong (Eds.), {\em  Pseudo-differential operators, generalized functions and asymptotics.} Oper. Theory Adv. Appl., 231, Birkh\"{a}user/Springer, Basel 2013, 283--296.
\bibitem{H:2004a}
G.~H{\"o}rmann, {\em H\"older-{Z}ygmund regularity in algebras of generalized functions.}
Z. Anal. Anwendungen {\bf 23} (2004), 139--165.
\bibitem{H:2004b}
G. H{\"o}rmann, {\em First-order hyperbolic pseudodifferential equations with generalizeds ymbols.}
J. Math. Anal. Appl. {\bf 293} (2004), 40--56.
\bibitem{HH:2001} G. H\"{o}rmann, M. V. de Hoop, {\em Microlocal analysis and global solutions of some hyperbolic equations with discontinuous coefficients}. Acta Appl. Math. {\bf 67} (2001), 173--224.
\bibitem{HK:2001}
G. H{\"o}rmann, M. Kunzinger, {\em Microlocal properties of basic operations in {C}olombeau algebras.}
J. Math. Anal. Appl. {\bf 261} (2001), 254--270.
\bibitem{HO:2004} G. H\"{o}rmann, M. Oberguggenberger, {\em Elliptic regularity and solvability for partial differential equations with Colombeau coefficients}. Electron. J. Differential Equations {\bf 2004}/14 (2004), 1--30.
\bibitem{HOP:06} G. H\"{o}rmann, M. Oberguggenberger, S. Pilipovi\'c, {\em Microlocal hypoellipticity of linear partial differential operators with generalized functions as coefficients}. Trans. Amer. Math. Soc. {\bf 358} (2006), 3363--3383.
\bibitem{HS:2012}
G. H\"{o}rmann, C. Spreitzer, {\em Symmetric hyperbolic systems in algebras of generalized functions and
 distributional limits.} J. Math. Anal. Appl. {\bf 388} (2012), 1166--1179.
\bibitem{LO:1991}
F.~Lafon, M.~Oberguggenberger, {\em  Generalized solutions to symmetric hyperbolic systems with
  discontinuous coefficients: the multidimensional case}.
J. Math. Anal. Appl. {\bf 160} (1991), 93--106.
\bibitem{NPS:1998}
M.~Nedeljkov, S.~Pilipovi{\'c},  D.~Scarpal{\'e}zos,
\newblock {\em The linear theory of {C}olombeau generalized functions}. Pitman Research Notes Math. 385.
\newblock Longman, Harlow 1998.
\bibitem{O:1989}
M.~Oberguggenberger, {\em Hyperbolic systems with discontinuous coefficients: generalized
  solutions and a transmission problem in acoustics}.
J. Math. Anal. Appl. {\bf 142} (1989), 452--467.
\bibitem{O:1992} M. Oberguggenberger, {\em Multiplication of distributions and applications to partial differential equations}. Pitman Research Notes Math. 259. Longman Scientific {\&} Technical, Harlow 1992.
\bibitem{O:2008} M. Oberguggenberger, {\em Hyperbolic systems with discontinuous coefficients: generalized wavefront sets}. In: L. Rodino, M. W. Wong (Eds.), New developments in pseudo--differential operators, pp. 117--136, Birkh\"{a}user, Basel, 2009.
\bibitem{MOFIOSSA:2014}
M. Oberguggenberger, M. Schwarz,
{\em Fourier Integral Operators in Stochastic Structural Analysis.} In: F. Werner, M. Huber,  T. Lahmer, T. Most, D. Proske (Eds.), {\em Proceedings of the 12th International Probabilistic Workshop.} Schriftenreihe des DFG Graduiertenkollegs 1462 (11). Bauhaus-Universit"{a}tsverlag, Weimar 2014, 250--257.
\bibitem{Scarp:2004}
D. Scarpal{\'e}zos, {\em Colombeau's generalized functions: topological structures; microlocal
  properties. {A} simplified point of view. {II}.
Publ. Inst. Math. (Beograd) (N.S.) {\bf 76(90)} (2004), 111--125.}
\bibitem{S:1963} R. P. Srivastav, {\em A note on certain integral equations of Abel-type}. Proc. Edinburgh Math. Soc. (2) {\bf 13} (1962/1963), 271--272.
\bibitem{W:1961} W. E. Williams, {\em Cauchy problem for the generalized radially symmetric wave equation}. Mathematika {\bf 8} (1961), 66--68.
\bibitem{Taylor:91}
M.E. Taylor. {\em Pseudodifferential Operators and Nonlinear PDE}. Birkh\"{a}user, Boston, 1991.

\end{thebibliography}
\end{document}